\documentclass[11pt]{article}

\usepackage{amsmath,amssymb,amsthm}
\usepackage{graphicx}

\usepackage{titlesec}

\titlespacing*{\paragraph}
  {0pt}   
  {5pt}   
  {0.5em}   

\usepackage{mathrsfs}

\usepackage{geometry,graphicx,xcolor,booktabs,enumitem}
\usepackage[colorlinks=true,linkcolor=blue!60!black,citecolor=blue!60!black]{hyperref}
\usepackage{tikz}
\usetikzlibrary{arrows.meta,positioning}
\usepackage{xcolor}
\usepackage{color,soul}
\newcommand{\black}{\color{black}{}}

\definecolor{myblue}{RGB}{90,140,200}

\definecolor{orange}{rgb}{0.8, .2, 0.0}

\definecolor{cR}{RGB}{33,102,172}
\definecolor{cS}{RGB}{214,96,77}
\definecolor{cA}{RGB}{26,150,65}
\definecolor{cB}{RGB}{178,24,43}

\theoremstyle{definition}
\newtheorem{definition}{Definition}[section]
\def\bR{{\mathbb R}}
\theoremstyle{plain}
\newtheorem{lemma}[definition]{Lemma}
\newtheorem{proposition}[definition]{Proposition}
\newtheorem{thmX}{Theorem}
 
\newtheorem{thmF}{Theorem}
 
\newtheorem{thmG}{Theorem}

\newtheorem{conH}{Conjecture}

\newtheorem{theorem}{Theorem}

\newtheorem{corollary}[definition]{Corollary}

\theoremstyle{remark}
\newtheorem{remark}[definition]{Remark}
\theoremstyle{conjecture}

\newtheorem*{lemmA}{Lemma A}

\newcommand{\eps}{\varepsilon}
\newcommand{\KX}{{\cal K}(X)}

\newcommand{\dH}{d_H}
\DeclareMathOperator{\id}{id}
\def\bZ{\mathbb Z}
\def\cA{{\cal A}}
\def\cB{{\cal B}}
\def\cC{{\cal C}}
\def\C{\cC}
\def\cD{{\cal D}}
\def\cF{{\cal F}}
\def\F{\cF}
\def\cE{{\cal E}}
\def\cG{{\cal G}}
\def\cK{{\cal K}}
\def\cH{{\cal H}}
\def\cI{{\cal I}}
\def\cL{{\cal L}}
\def\cM{{\cal M}}
\def\cN{{\cal N}}
\def\cO{{\cal O}}

\def\cQ{{\cal Q}}
\def\cR{{\cal R}}
\def\sR{{\cal R}}
\def\cS{{\cal S}}
\def\cT{{\cal T}}
\def\RF{\mathscr{R}_\cF}

\def\Gl{\cG_\cF}
\def\owR{\cR^{ow}_\cF}
\def\epsilon{\varepsilon}
\newcommand{\BF}{\bf\boldmath }

\def\Ato{\succcurlyeq_{\cA_\cF}}
\def\Dto{\succcurlyeq_D}
\def\Sto{\succcurlyeq_S}
\def\Fto{\succcurlyeq_\cF}
\def\Hto{\succcurlyeq_{H_\cF}}
\def\Cto{\succcurlyeq_{\cC_\cF}}
\def\Cdto{\succcurlyeq_{\cC_{\cF,d}}}
\def\CtoF{\succcurlyeq_{\cC_{\widehat F}}}
\def\CEto{\succcurlyeq_{\cC_\cE}}

\usepackage{lineno}
\newcommand{\beq}{\begin{linenomath}\begin{equation*}} 
\newcommand{\eeq}{\end{equation*}\end{linenomath}} 
\newcommand{\beqn}{\begin{linenomath}\begin{equation}} 
\newcommand{\eeqn}{\end{equation}\end{linenomath}} 
\def\Om{\Omega_\cF}
\def\P{\Phi_\cF^I}

\def\PJ{\Phi_\cF^J}
\def\Pn{\Phi_\cF^{I_n}}

\def\Ceq{\overset{\cC_{\cF}}{=}}

\def\Cdeqeta{\overset{\cC_{\cF,d,\eta}}{=}}

\def\Feq{\overset{\cF}{=}}
\def\Seq{\overset{\cS}{=}}
\def\Aeq{\overset{\cA_\cF}{=}}

\DeclareMathOperator{\Node}{Node}
\DeclareMathOperator{\Down}{Down}
\DeclareMathOperator{\Up}{Up}

\title{Recurrent and gradient dynamics\\
of Iterated Function Systems.}
\author{Roberto De Leo}
\date{\today}

\begin{document}
\maketitle

\begin{abstract}
    We study the qualitative dynamics of general Iterated Function Systems (IFSs) through chain recurrence and a decomposition into recurrent and gradient-like behavior.
    At the same time, we move the focus from the topological property of the phase space to the dynamical properties of the system and we say that an IFS has compact dynamics if it has an ``invariant'' compact set (global attractor) that attracts every compact subset of the phase space.

    For an IFS with compact dynamics, we introduce a distinction that has no counterpart for a single map: ordinary chain recurrence, which may be realized along a suitable itinerary, and orbitwise chain recurrence, which requires the whole semigroup orbit of a point to remain in its chain component. 
    This leads naturally to two directed graphs encoding the recurrent components and the downstream relations between them.

    We prove that the chain structure of an IFS with compact dynamics over a locally compact space is completely determined by the restriction of the IFS to its compact global attractor. 
    We then obtain a Conley-type trichotomy: every point of the global attractor has either orbitwise recurrent dynamics, mixed recurrent–gradient dynamics, or purely gradient dynamics. The three alternatives admit equivalent descriptions in terms of bitrajectories and strict Lyapunov functions. 
    We also give conditions guaranteeing connectedness of the chain graph and of the orbitwise chain graph.

    Finally, we compare the point dynamics of the IFS with two associated single-map systems. 
    The chain graph is naturally isomorphic to that of the full one-sided step skew product, while orbitwise chain components determine distinct chain components of the Hutchinson map. We also establish a localization principle extending several compact-space results from Conley theory for closed relations to IFSs whose global attractor has a compact neighborhood.
\end{abstract}

\section{Introduction}

The qualitative dynamics of a dynamical system is largely determined by the structure of its recurrent behavior and by the gradient-like dynamics connecting different recurrent components. 
For a single map, this point of view is classical and lies at the heart of Conley theory. 
In the setting of an iterated function system (IFS), however, the presence of several generators introduces a new distinction: recurrence may occur along a suitable choice of itinerary without being preserved by the action of the whole semigroup.

The purpose of this article is to develop this recurrent–gradient viewpoint for general IFSs. 
We distinguish between ordinary chain recurrence, which records recurrence obtainable along some itinerary, and orbitwise chain recurrence, which requires the entire semigroup orbit of a point to remain in the same chain component. 
This leads to two related decompositions of the dynamics and, in particular, to a trichotomy between orbitwise recurrent, mixed recurrent–gradient, and purely gradient behavior.

Our main tool is the notion of stream, introduced in~\cite{DLY25} for semiflows: a closed and transitive extension of the orbit relation encoding which points are dynamically downstream from which. 
The recurrent equivalence classes of a stream form the nodes of a directed graph, while downstream relations between distinct nodes describe the gradient-like part of the dynamics. 
Here we adapt this framework to IFSs, with particular emphasis on the chain stream generated by the $\eps$-chains of Barnsley and Vince.





\paragraph{General IFSs.}
Historically, IFSs have been used quite successfully for the construction of deterministic fractals~\cite{Bar88} in the context of contractive maps, beginning with the seminal article of Hutchinson~\cite{Hut81} and the introduction of the general theory of contractive IFS by Michael Barnsley in~\cite{Bar88}. 
Recently, Barnsley and Andrew Vince~\cite{BV12,Vin13} and, independently, the present author~\cite{DL08,DL15,DL15b} started investigating some self-projective IFSs on real and complex projective spaces whose maps are not contractions.
At the same time, Barnsley et al. started investigating properties of ``general'' IFSs~\cite{BV11,BV13,BL14}, namely IFSs where the only requirement is that maps be continuous, mostly from the point of view of the {\em chaos game}.
Here we rather study general IFSs as dynamical systems, with particular emphasis on recurrence, gradient-like behavior, and the global organization of their chain components.

\paragraph{Compact dynamics.} 
For simplicity, we restrict our attention to locally compact metric spaces.
We also require separability for technical reasons. 
We do not require completeness and we assume instead that our IFSs have ``compact dynamics''.
More precisely, we require our IFSs to have a non-empty global attractor, namely an $H$-invariant (i.e. invariant under the Hutchinson map) compact set that attracts every compact subset of the phase space.
We showed in~\cite{DLY26a} that the qualitative behavior of a semiflow with compact dynamics is equivalent to that of its restriction to its global attractor, and here we show that the same happens in case of IFSs.
This way, we switch the focus from the topological properties of the phase space to the dynamical properties of the IFS.

\paragraph{Streams.}
Given an IFS $\cF$ with compact dynamics on a locally compact metric space $X$, we use ``streams'' to decompose the dynamics of $\cF$ into its ``recurrent'' and ``gradient'' components and represent the qualitative dynamics of $\cF$ in a graph. 
A stream of $\cF$ is a closed and transitive binary relation on $X$ that establishes ``which point is downstream from which''. 
For any given point $x\in X$, every point on the orbit of $x$ is downstream from $x$ but in general, due to the closure and transitivity properties, the set of points downstream from a given point is larger than its sole orbit.
Given a stream $S$ of $\cF$, it is natural to single out those
subsets $N\subset X$ such that every point of $N$ is both upstream and downstream from every point of $N$.
We call each of these sets a ``node'', since these sets will be the nodes of the graph we will use to encode the qualitative behavior of the IFS.
One can think nodes as a generalization of the concept of periodic orbit.
A node $N$ is downstream from another node $M$ if the points of $N$ are downstream from those of $M$.
We encode the qualitative behavior of $S$ into a directed graph by taking all nodes of $S$ as the nodes of the graph and saying that there is an edge from $M$ to $N$ in the stream's graph if $N$ is downstream from $M$.
Due to transitivity, the graph of any stream has no cycles.


\paragraph{Chain recurrence and orbitwise chain recurrence.}
A new feature appears as soon as $m>1$.
The usual notion of chain recurrence for an IFS, introduced by Barnsley and Vince, describes recurrence that can be realized by a suitable choice of itinerary, but its recurrent set and its chain nodes need not be forward-invariant under the whole IFS. 
This leads us to distinguish between \emph{chain recurrence} and the  stronger notion of \emph{orbitwise chain recurrence}, in which every point of the semigroup orbit of $x$ is required to remain in the chain-equivalence
class of $x$. 
The two notions coincide for $m=1$, but can differ substantially for a genuine IFS. 
Accordingly, we associate to the chain stream two directed graphs: the chain graph, which records the possible recurrent and gradient dynamics obtained by choosing itineraries, and the orbitwise chain graph, whose nodes are forward-invariant under the whole free-monoid action.


\paragraph{Related literature.}
Our framework also admits two complementary interpretations already present in the literature. 
First, the IFS determines the one-step closed relation $\mathscr R_{\mathcal F}
=
\cup_{i=1}^m \operatorname{graph}(f_i),$ whose set-image operator is precisely the Hutchinson map. 
Thus, the chain theory developed here is closely related to the Conley theory of closed relations developed by Akin~\cite{Akin93} and Akin and Wiseman~\cite{AW17} and by McGehee~\cite{McG92} and McGehee--Wiandt~\cite{MW06}  (see Section~\ref{subsec: Akin Wiseman}).
What distinguishes our approach from these two predecessors is not the underlying object but rather what is done with it.
The goal of Akin and Wiseman is to extend Conley's chain-recurrence theory and Easton's strong-chain recurrence theory in two directions: first from maps to relations, and second from compact metric spaces to general uniform spaces.
The goal of McGehee and Wiandt is to show that the natural setting for the Conley Decomposition Theorem is iteration of closed relations on compact spaces.
In our case, we are not interested in streams per se but only as a tool to encode the qualitative behavior of a dynamical system into a graph, so that it can be easily represented and communicated.  
Hence, our results are mainly about the ingredients of the graph (nodes and edges) and its properties, in particular conditions that grant its connectedness.

Second, after adjoining the itinerary as part of the state, the IFS becomes the single step skew-product map $
\widehat F(\omega,x)
=
\bigl(\sigma\omega,f_{\omega_0}(x)\bigr)$.
From this map, one can extract information about the behavior of the IFS.
In Section~\ref{sec:skew} we show that the projection onto $X$ induces an isomorphism between the chain graph of this skew product and the chain graph of the original IFS. 
Thus, although a chain node downstairs need not itself be forward-invariant, it is the projection of a genuine invariant chain component of an ordinary dynamical system upstairs.
At the same time, we show that the skew product map does not retain information about the orbitwise nodes and so can tell only one part of the story.

An IFS can have infinitely many different streams, each highlighting some different aspects of its dynamics, but the most important is arguably the chain stream, based on an adaptation to IFS of Charles Conley's $\eps$-chains.



\paragraph{Main results.}
Assume that an IFS $\cF$ on a locally compact metrizable space $X$ has compact dynamics and denote by $\Gl$ its global attractor. 
First of all, we have the following:
    \begin{enumerate}
        \item The chain stream of $\cF$ has the same nodes and edges (and so the same graph) as the chain stream of the restriction of $\cF$ to $\Gl$ (Theorem~\ref{prop: C = C glob}).\par
        {\em This means that $\cF$ has the same qualitative dynamics as an IFS on a compact space.}
        \item For every orbitwise chain node $N$ of $\cF$, the set $\Om(N)$ is non-empty and is the greatest nonempty compact $H$-invariant subset of $N$.
        (Theorem~\ref{thm: N is H-invariant}).\par
        {\em In the IFS framework, H-invariance plays the role that invariance plays for semiflows.}  
        \item If $\Gl$ is connected, the chain graph of $\cF$ is connected (Theorem~\ref{thm: weak chain graph is connected}). 
        Moreover, if every point in $\Gl$ has a backward orbit whose limit points lie in some orbitwise chain node, then also the orbitwise chain graph of $\Phi_\cF$ is connected (Theorem~\ref{thm: G conn => graph conn}).\par
        {\em Connectedness is a strong topological condition for a graph. 
        In~\cite{DLY26a}, James Yorke and the present author proved that, given a semiflow with compact dynamics on a locally compact space, the connectedness of its global attractor is enough to grant the connectedness of its chain graph. 
        These results show what happen in case of IFSs.}
        \item If the smallest stream of $\cF$ has countably many nodes, then it coincides with its chain stream (Theorem~\ref{thm: Omega-stream}).\par
        {\em Hence, in these cases the chain stream is the smallest stream of $\cF$ and so the one that is most strictly related to the dynamics of $\cF$.}
        \item The Hutchinson map $H_\cF$ of $\cF$ has compact dynamics and its global attractor is a subset of $\cK(\Gl)$, the set of all compact subsets of $\Gl$ (Theorem~\ref{thm: F has comp dyn <=> H has comp dyn}).
        More generally, we show that $H_\cF$ has compact dynamics if and only if $\cF$ has compact dynamics.\par
        {\em The Hutchinson map of an IFS is able to reveal hidden aspects of its dynamics. 
        This result shows that our framework for semiflows, applied to $H_\cF$, aligns nicely the framework for IFS applied to $\cF$.
        Notice that, if the global attractor of $\cF$ is connected, then the global attractor of $H_\cF$ is connected as well and so, by Theorem~6 in~\cite{DLY26a}, the graph of $H_\cF$ is connected.}
        \item For each orbitwise chain node $N$ of $\cF$, the set $\Om(N)$, as a compact $H$-invariant set, is a chain-recurrent point of $H_\cF$ and for distinct orbitwise chain chain nodes $N,M$ of $\cF$ the sets $\Om(N),\Om(M)$ are contained in distinct chain nodes of $H_\cF$. 
        In particular, if the chain graph of $H_\cF$ consists of a single node, then the chain graph of $\cF$ has a single node as well, which coincides with $\Gl$ (Theorem~\ref{thm: H one node => F one node}).\par
        {\em The relation between the dynamics of $\cF$ and the dynamics of $H_\cF$ is often non-trivial. 
        Hence, general relations between the qualitative behavior of these two systems are particularly important.
        This result shows that, when the dynamics of $H_\cF$ is particularly simple (namely the whole global attractor is itself a pond), the same happens to $\cF$.
        The inverse result does not hold: even for IFSs consisting of a single map there are cases 
        where the chain graph of $\cF$ has a single node while the chain graph of $H_\cF$ has many.
        }
    \end{enumerate}
    %

    Since global attractors are often topologically complicated and hard to find analytically, we provide a criterion to establish their existence: an IFS $\cF$ has compact dynamics if and only if it has a compact global trapping region (Theorems~\ref{thm: Q => G} and~\ref{thm: compact dynamics = compact gl tr reg}).
    Often a system has many trapping regions, and these trapping regions are frequently topologically simple (e.g. convex open sets).

    As an application of this framework, we extend to IFSs the celebrated theorem of Conley on semiflows that, loosely speaking and expressed in our framework language, either a point is in a chain node or ``in an edge'', namely every point that is not chain-recurrent has the property that each bitrajectory through it has its $\alpha$- and $\omega$-limit sets lying in distinct chain nodes.
    In Theorem~\ref{thm: trichotomy} we confirm that the same holds for IFSs. 
    In case of IFSs, moreover, we distinguish between orbitwise recurrent points, which are those chain-recurrent points $x$ for which every trajectory of $x$ lies in the same node as $x$, and those for which, on the contrary, there are trajectories that are not contained in the same node as $x$.
    Hence, in case of IFSs, for every point $x\in\Gl$ we have exactly three possible (and mutually exclusive) possibilities: orbitwise recurrent dynamics, mixed recurrent-gradient dynamics and purely gradient dynamics.

    As a second application, we prove a localization principle (Theorem~\ref{thm:localization}) that shows that any proposition on compact spaces based solely on results concerning the chain-recurrent set, its chain components and their downstream order can be transferred immediately to IFS with compact dynamics whose global attractor has a compact neighborhood.
    We use this principle to extend to such IFSs a few theorems about the chain-recurrent set by Akin-Wiseman and McGehee-Wiandt (see Section~\ref{subsec: Akin Wiseman}).  
    
    Finally, we show that the chain graph of an IFS coincide with the chain graph of its skew product (Theorem~\ref{thm:skew}), a dynamical system often used to study the qualitative properties of a IFS. 
    At the same time, though, we argue that there are some critical information about the collective dynamics of the IFS that are not visible to its skew product (see Section~\ref{sec:skew}).


    \paragraph{Structure of the article.}
    In Section~\ref{sec: IFS}, we introduce the main definitions about the dynamics of an IFS.
    In Section~\ref{sec: global attractor}, we introduce global attractors and trapping regions and prove their main properties. 
    In Section~\ref{sec: streams}, we introduce streams of IFSs and prove several of their fundamental properties.
    In Section~\ref{sec: chains}, we study in some detail the fundamental properties of the chain stream of an IFS.
    In Section~\ref{sec: Hutch}, we apply our framework on semiflows to the Hutchinson map and compare the chain stream of an IFS with the chain stream of its Hutchinson map.
    Finally, in Section~\ref{sec: related}, we compare our work with the main literature related to the subject, import some result from it and extend them within our framework and finally prove some further related proposition.
%
\section{Iterated Function Systems}
\label{sec: IFS}
We introduce several definitions and some related basic results on the qualitative dynamics of IFSs.
Our goal is to be able to describe the qualitative behavior of general IFSs, namely IFSs whose maps are not necessarily contractions.

\medskip\noindent
{\bf The phase space.} Throughout the article, {\BF$X$} will denote a {\bf metrizable, separable and locally compact} topological space. 
We will usually denote points in $X$ by {\BF$x,y,z$} and {\BF$d(x,y)$} will denote the distance between $x$ and $y$ for some metric $d$ compatible with the topology of $X$. 

\medskip\noindent
{\bf Distance functions.} 
Throughout the article we use the following distance functions derived from $d(x,y)$.
In the last two items, we assume that $A,B$ are nonempty and compact.
\begin{enumerate}
    \item ({\bf point-to-set distance}) for $x\in X$ and $A\subset X$ we set {\BF$d(x,A)$}$=\inf_{a\in A}d(x,a)$;
    \item ({\BF one-sided excess of $A$ over $B$}) for $A,B\subset X$, we set {\BF$e(A,B)$}$=\sup_{a\in A}d(a,B)$;
    \item ({\BF Hausdorff distance}) for $A,B\subset X$, we set {\BF$d_H(A,B)$}$=\max\{e(A,B),e(B,A)\}$.
\end{enumerate}

\medskip\noindent
{\bf The multi-indices semigroup.} We denote by {\BF$\cI^m$} the free monoid generated by $\{1,\dots,m\}$ with identity 0.
Given $I=i_1\dots i_k\in\cI^m$, we denote by {\BF$|I|=k$} the number of indices in $I$ and call this number the {\BF length} of $I$.
We denote by {\BF $\cI^m_k$} the set of all indices of $\cI^m$ of length at least $k$.


\medskip\noindent
{\bf Iterated Function Systems.} 
%
    An {\bf iterated function system} (IFS) on a topological space $X$ is a finite collection {\BF$\cF = \{f_1,\dots,f_m\}$} of continuous maps of $X$ into itself.
    Given $I=i_1\dots i_k\in\cI^m$, we set
    $$
    \P = f_{i_k}\circ\dots\circ f_{i_1}.
    $$

    An IFS $\cF$ determines a continuous semigroup right action
    $$
     \Phi_\cF: \cI^m \times X \to X
    $$
    by $(I,x)\mapsto\Phi^I_\cF(x)$.
    In particular, this means that:
    \begin{enumerate}
        \item $\Phi_\cF^0=id_X$;
        \item $\Phi^I_\cF\circ\Phi^J_\cF=\Phi^{J\cdot I}_\cF$ for every $I,J\in\cI^m$.
    \end{enumerate}
    We call $\Phi_\cF$ the {\bf action} associated to $\cF$.
    Notice that, when $m=1$, this reduces to the definition of {\em discrete-time semiflow} on $X$.
    
    Finally, we say that $\cF$ is {\bf contractive} when each of its maps $f_i$ is a contraction.
 
\medskip\noindent
{\bf Invariance under $\Phi_\cF$.}
  We say that a set $A\subset X$ is {\bf forward-invariant} (resp. {\bf backward-invariant}) under $\Phi_\cF$ if $\Phi^I_\cF(A)\subset A$ (resp. $\Phi^I_\cF(A)\supset A$) for every $I\in\cI^m$.
  Equivalently, if $f_i(A)\subset A$ (resp. $f_i(A)\supset A$) for every $i=1,\dots,m$.
  We say that a set $A$ is {\bf invariant} under $\Phi_\cF$ if it is both backward- and forward-invariant, namely if $\Phi^I_\cF(A)=A$ for every $I\in\cI^m$.
  Finally, we write {\BF$H_\cF(A)=\cup_{i=1}^m f_i(A)$} and say that $A$ is {\bf $H$-invariant} under $\cF$ if $H_\cF(A)=A$.
%

\medskip\noindent
{\bf Orbits.}
    We write {\BF $x\Fto y$} if $y=\Phi^I_\cF(x)$ for some $I\in\cI^m$.
    We write {\BF $x\Feq y$} if $x\Fto y$ and $y\Fto x$.
    We call {\bf orbit relation} of $\cF$ the set 
    \beq
    \text{\BF$\cO_\cF$} = \{(x,y): x\Fto y\},
    \eeq
    so that the orbit of any given point $x$ under $F$ is given by
    \beq
    \text{\BF$\cO_\cF(x)$} = \{y: (x,y)\in\cO_\cF\}.
    \eeq


    

\medskip\noindent
{\bf Limit sets and recurrence.}
    The limit set of a point $x\in X$ under $\cF$ is the set
    \beq
    \text{\BF$\Om(x)$} = \{y: \text{ there are }\; I_n\in\cI^m\text{ such that }|I_n| \to\infty\text{ and }\Phi^{I_n}_\cF(x)\to y\}.
    \eeq
    Given a subset $A\subset X$, we define
    \beq
    \text{\BF$\Om(A)$} = \{y:\text{ there are }I_n\in\cI^m\text{ and }a_n\in A\text{ such that }|I_n|\to\infty\text{ and }\Phi^{I_n}_\cF(a_n)\to y\}.
    \eeq
    Notice that, for a singleton $A=\{x\}$, the definition reduces to the limit set of a point, so $\Om(\{x\})=\Om(x)$.

    
    We say that $x$ is {\bf fixed} for $\cF$, or that $x$ is a {\bf fixed-point} of $\cF$, if $\cO_\cF(x)=\{x\}$.
    Finally, we say that $x$ is {\bf recurrent} under $\cF$ if $x\in\Om(x)$ and we denote by $\cR_\cF$ the set of all points that are recurrent under $\cF$; if, moreover,  $x\in\Om(z)$ for every $z\in\cO_\cF(x)$, we say that $x$ is {\bf orbitwise recurrent}.
    We denote by $\cR^{ow}_\cF$ the set of all orbitwise recurrent points.

    \medskip\noindent
    {\bf Orbitwise recurrence.}
    For $m=1$, both recurrence and orbitwise recurrence reduce to the standard concept of recurrence for semiflows.
    For $m>1$, the two concepts in general differ.
    For instance, the recurrent set is not necessarily forward-invariant, as the following example shows.

\medskip
Let $X=[0,1]$ and consider the IFS $\cF=\{f,g\}$, where
\[
f(t)=\min\left\{t+\frac12,1\right\},
\qquad
g(t)=\min\{2t,1\}.
\]
Since $0$ is fixed under $g$, $0\in\Om(0)$ and so $0\in\cR_\cF$.
Now, let \(x>0\). 
For every \(t\in[0,1]\), one has $f(t)\geq \min\{2t,1\}$ and $g(t)=\min\{2t,1\}$.
Consequently, every word \(h_n\circ\cdots\circ h_1\), with \(h_j\in\cF\), satisfies
\[
h_n\circ\cdots\circ h_1(x)
\geq
\min\{2^n x,1\}.
\]
Thus, for all $I$ with $|I|$ sufficiently large, we have that $\P(x)=1$ and, therefore, $\Omega_{\cF}(x)=\{1\}$
for every $x\in(0,1]$.
Hence, $\cR_\cF=\{0,1\}$.
However, $f(0)=\frac12\notin\cR_\cF$, so $\cR_\cF$ is not forward-invariant.

\medskip
On the contrary, $\cR^{ow}_\cF$ is forward-invariant.
Indeed 1 is orbitwise recurrent since it is fixed for the whole IFS.
On the other side, since $\cO_\cF(0)=\{0,1/2,1\}$ and $\cO_\cF(1/2)=\{1/2,1\}$, $0\not\in\cR^{ow}_\cF$. 
Hence, $\cR^{ow}_\cF=\{1\}$, which is forward-invariant.
The following proposition shows that this is not a coincidence.

\begin{proposition}
\label{prop: forward invariant}
The sets $\cR^{ow}_\cF$ and $\Om(A)$, $A\subset X$, are forward-invariant (and so is $\Om(x)$ for all $x\in X$).
\end{proposition}
\begin{proof}
    Let $x\in \Om(A)$, and fix $i\in\{1,\dots,m\}$. 
    We show that $f_i(x)\in \Om(A)$.
    Since $x\in \Om(A)$, by definition there exist points $a_n\in A$ and multi-indices $I_n\in I^m$ such that $|I_n|\to\infty$ and $\Phi_\cF^{I_n}(a_n)\to x$.
    Since $f_i$ is continuous, it follows that
\[
f_i\bigl(\Phi_\cF^{I_n}(a_n)\bigr)=\Phi_\cF^{I_n i}(a_n)\to f_i(x).
\]
    Moreover, $|I_n i|=|I_n|+1\to\infty$. Therefore, by definition, $f_i(x)\in \Om(A)$, namely $\Om(A)$ is forward-invariant.

    Now, let $x\in \owR$, and fix $i\in\{1,\dots,m\}$. 
    We claim that $f_i(x)\in \owR$.
    Indeed, let $z\in \cO_\cF(f_i(x))$. 
    By transitivity, $z\in \cO_\cF(x)$. 
    Since $x\in \owR$, by definition we have $x\in \Om(z)$, so that there exists a sequence $I_n\in\cI^m$ such that $|I_n|\to\infty$ and $\Pn(z)\to x$.
    Applying $f_i$ and using continuity, we obtain $f_i\bigl(\Phi_\cF^{I_n}(z)\bigr)\to f_i(x)$.
    Equivalently, 
    $\Phi_\cF^{I_n i}(z)\to f_i(x)$.
    Since $|I_n i|=|I_n|+1\to\infty$, this shows that $f_i(x)\in \Om(z)$.
    Since $z$ was arbitrary, it follows that $f_i(x)\in \owR$, namely $\owR$ is forward-invariant.
\end{proof}



\medskip\noindent
{\bf Example: the orbitwise recurrent set of contractive IFSs.}
Contractive IFSs are one of the most studied kinds of IFS in literature. 
It is well-known~\cite{Hut81,Bar88} (see Theorem~\ref{thm: Hutch}) that a contractive IFS over a complete metric space $X$ admits a unique non-empty compact set $K$ which is $H$-invariant.
Next proposition shows that the recurrent set of any contractive IFS is precisely that invariant set.
\begin{proposition}
\label{prop: R_F=K}
    Assume that $X$ is complete and that $\cF$ is contractive.
    Denote by $K$ the Hutchinson attractor of $\cF$, i.e. the unique non-empty compact set such that $K=H_\cF(K)$.
    Then $\cR_\cF=\owR=K$.
\end{proposition}

\begin{proof}
    Let $\lambda\in(0,1)$ be such that $d(f_i(u),f_i(v))\le \lambda\, d(u,v)$ for every $u,v\in X$ and $i=1,\dots,m$.
    Such $\lambda$ exists because $\cF$ finite and each $f_i$ is a contraction.
    Hence, for every $I\in\cI^m$,
    \[
    d(\Phi_\cF^I(u),\Phi_\cF^I(v))\le \lambda^{|I|}d(u,v).
    \]

    \noindent
    {\BF 1. $K\subset \owR$.}
    Let $x\in K$ and $z\in \cO_\cF(x)$.
    Since $K$ is forward-invariant, $z\in K$. 
    Since $K=H_\cF(K)$, for each $n\geq1$ there is a multiindex $I_n$ with $|I_n|=n$ and a $y_n\in K$ such that $x=\Pn(y_n)$.
    Therefore
$$
d\bigl(\Pn(z),x\bigr)
=
d\bigl(\Pn(z),\Pn(y_n)\bigr)
\leq
\lambda^n d(z,y_n)
\leq
\lambda^n\operatorname{diam}(K),
$$
    so that $\Pn(z)\to x$.
    Since $|I_n|=n\to\infty$, this shows that $x\in\Om(z)$.
    Since $z\in \cO_\cF(x)$ was arbitrary, we conclude that $x\in \owR$, so that $K\subset \owR$.    

    \medskip\noindent
    {\BF 2. $\cR_\cF\subset K$.}
    Let $x\notin K$ and set $d(x,K)>0$.
    Since $X$ is complete, $\KX$ is complete, and $H_\cF$ is a contraction on $\KX$. 
    Therefore we have that $d_H\bigl(\cH_\cF^n(\{x\}),K\bigr)\to0$ and, consequently,
    \[
    \sup_{|I|=n}d\bigl(\P(x),K\bigr)
    =
    e\bigl(\cH_\cF^n(\{x\}),K\bigr)
    \leq
    d_H\bigl(\cH_\cF^n(\{x\}),K\bigr)
    \longrightarrow0.
    \] 
    Hence, for all sufficiently large $n$ and every $I\in\cI^m_n$,
\[
d(\P(x),x)\ge d(x,K)-d(\P(x),K)>d(x,K)/2.
\]
    Thus, no sequence of multiindices $L_n$ with $|L_n|\to\infty$ can satisfy $\Phi_\cF^{L_n}(x)\to x$.
    Hence $x\notin\Om(x)$, so $x$ is not recurrent.
    This means that $\cR_\cF\subset K$.

    Since $\owR\subset\cR_\cF$, (1) and (2) imply that $\cR_\cF=\owR=K$.
\end{proof}

\medskip\noindent
{\bf Limit sets are closed.} The following topological characterization of the limit set shows that every limit set is closed.
\begin{proposition}[Topological description of the limit set]
\label{prop: topological description of Om}
\[
\Om(A)
=
\bigcap_{n\ge 0}\overline{\bigcup_{|I|\ge n} \P(A)}.
\]
In particular, 
$\Om(A)$ is closed for every $A\subset X$ and $\Om(x)$ is closed for every $x\in X$.
\end{proposition}

\begin{proof}
We prove the two inclusions separately.

\medskip
\noindent
\emph{First inclusion:}
\[
\Om(A)\subset \bigcap_{n\ge 0}\overline{\bigcup_{|I|\ge n}\Phi_\cF^I(A)}.
\]

Let $y\in \Om(A)$. By definition, there exist points $a_n\in A$ and multi-indices
$I_n$ such that
$|I_n|\to\infty$
and $\Phi_\cF^{I_n}(a_n)\to y$.
Fix $N\ge 0$. Since $|I_n|\to\infty$, there exists $n_0$ such that
$|I_n|\ge N$ for all $n\ge n_0$.
Hence, for every $n\ge n_0$,
\[
\Phi_\cF^{I_n}(a_n)\in \bigcup_{|I|\ge N}\Phi_\cF^I(A).
\]
Passing to the limit, we obtain
\[
y\in \overline{\bigcup_{|I|\ge N}\Phi_\cF^I(A)}.
\]
Since $N$ was arbitrary,
\[
y\in \bigcap_{N\ge 0}\overline{\bigcup_{|I|\ge N}\Phi_\cF^I(A)}.
\]

\medskip
\noindent
\emph{Second inclusion:}
\[
\bigcap_{n\ge 0}\overline{\bigcup_{|I|\ge n}\Phi_\cF^I(A)}\subset \Om(A).
\]

Let
\[
y\in \bigcap_{n\ge 0}\overline{\bigcup_{|I|\ge n}\Phi_\cF^I(A)}.
\]
Then for every $n\ge 0$,
\[
y\in \overline{\bigcup_{|I|\ge n}\Phi_\cF^I(A)}.
\]
Choose a decreasing neighborhood basis $(U_n)_{n\ge 0}$ at $y$, with
\[
U_0\supset U_1\supset U_2\supset \cdots,
\qquad
\bigcap_{n\ge 0}U_n=\{y\}.
\]
Since
\[
U_n\cap \bigcup_{|I|\ge n}\Phi_\cF^I(A)\neq\varnothing,
\]
for each $n$ there exist a point $a_n\in A$ and a multi-index $I_n$ with
$|I_n|\ge n$, and so $|I_n|\to\infty$, and $\Phi_\cF^{I_n}(a_n)\in U_n$.
Moreover, since the neighborhoods $U_n$
shrink to $y$, it follows that
$
\Phi_\cF^{I_n}(a_n)\to y.
$
Therefore, by definition,
$
y\in \Om(A).
$

Combining the two inclusions, we conclude that
$$
\Om(A)=\bigcap_{n\ge 0}\overline{\bigcup_{|I|\ge n}\Phi_\cF^I(A)}.
$$

Finally, each set
$
\overline{\bigcup_{|I|\ge n}\Phi_\cF^I(A)}
$
is closed, so their intersection is closed.
\end{proof}

\begin{remark}
If $A$ is forward-invariant under $F$, then
\[
\bigcup_{|I|\ge n+1} \P(A)\subset \bigcup_{|I|\ge n} \P(A),
\]
so the family
\[
\overline{\bigcup_{|I|\ge n} \P(A)}
\]
is decreasing, and the preceding proposition expresses $\Om(A)$
as the intersection of this nested family of closed sets.
\end{remark}

\medskip\noindent
{\BF $H$-invariance plays the role of invariance in IFSs with $m>1$.}
When $m=1$, invariance is a fundamental concept: when $X$ is compact, or even just when a semiflow has compact dynamics (see Section~\ref{sec: global attractor}), the global attractor is invariant and so are the chain recurrent set and each of its nodes (see Section~\ref{sec: chains}).

On the contrary, when $m>1$, having invariant sets is a strong condition and often an IFS has no invariant set.
%
    For instance, consider an IFS $\cF=\{f_1,\dots,f_m\}$ on a compact space $X$ whose maps are all contractions and denote by $V_i$ the fixed point of $f_i$.
    Assume that there is at least a pair of  indices $i,j$ such that $V_i\neq V_j$.
    Then there is no non-empty compact subset of $X$ which is invariant under $\Phi_\cF$.
%

Next sections show that, for $m>1$, the role of invariance is played by H-invariance.
%


%
%

\section{Global attractor of an IFS}
\label{sec: global attractor}
%
Just as in case of semiflows, the IFSs with richest asymptotics are the ones whose dynamics is, in some sense, ``compact''.
Of course this includes all IFSs defined over a compact space $X$ but requiring $X$ to be compact is unnecessarily strong -- it is enough that the orbits of the system are attracted by and stay within some compact region. 
For instance, this is precisely what happens in case of the celebrated Lorenz ``butterfly'' system~\cite{Lor63}. 
In this section, we give a precise definition of the IFSs that can be studied within our framework.

\medskip\noindent
{\bf Attracting and absorbing.}
    We say that a set $A\subset X$ {\bf absorbs} a set $K$ under $\cF$ if
    there exists $\tau>0$ such that
    \[
    \P(K)\subset A \text{ for every }I\in\cI^m_\tau.
    \]
    Given an $\eps>0$, we set
    $$
    N_\eps(A)=\{x : d(x,A)<\eps\}.
    $$
    We say that $A$ {\bf attracts} $K$ under $\cF$ if, for every $\eps>0$, there exists $\tau>0$ such that 
    $$
    \P(K)\subset N_\eps(A)\text{ for all }I\in\cI^m_\tau.
    $$
%

\medskip
For a semiflow $F$, we defined in~\cite{DLY26a} the global attractor as an {\em invariant} compact set that attracts under $F$ every compact subset of $X$.
In the framework of IFSs, this definition would not work because, as we pointed out in the previous section, often flows of IFSs have no invariant set whatsoever.
It turns out, though, that most of the theory works in the framework of IFSs if we replace {\em invariance} by {\em H-invariance}.

\medskip\noindent
{\bf Global Attractor.}    
    We call {\bf global attractor} a compact $H$-invariant set $\cG\subset X$ that attracts, under $\cF$,  each compact subset $K\subset X$.

\medskip
%
\begin{proposition}
    \label{prop: G is unique}
    Assume that $\cF$ has global attractors $\cG$ and $\cG'$. 
    Then $\cG=\cG'$. 
\end{proposition}
\begin{proof}
Assume that $\cG$ and $\cG'$ are two global attractors of $\cF$. By definition, both $\cG$ and $\cG'$ are compact, $H$-invariant, and attract every compact subset of $X$.
Since $\cG$ is a global attractor, it attracts $\cG'$. 
Hence, for every $\eps>0$, there is a $N>0$ such that $\P(\cG')\subset N_\eps(\cG)$ for every $I\in\cI^m_N$.
In particular, $H_\cF^n(\cG')\subset N_\eps(\cG)$ for every $n\geq N$.
Since $\cG'$ is $H$-invariant, this means that $\cG'\subset N_\eps(\cG)$ for every $\eps>0$.
Hence, $\cG'\subset\cG$.
By symmetry, $\cG=\cG'$.
\end{proof}
The following result proves for IFSs a standard alternate characterization of the global attractor of a semiflow.
\begin{proposition}
    \label{prop: G_F minimal}
    Let $\cF$ be an IFS with a global attractor $\cG$.
    Then $\cG$
    is the minimal compact forward-invariant set of $\cF$ that attracts all compact subsets of $X$.
\end{proposition}


\begin{proof}
By definition, $\cG$ is compact, $H$-invariant, and attracts every compact subset of $X$. 
Since H-invariance implies forward-invariance, $\cG$ is a compact forward-invariant set that attracts every compact subset of $X$.

Let now $K\subset X$ be any compact forward-invariant set that attracts every compact subset of $X$. 
We claim that
$
\cG\subset K.
$

Since $\cG$ is $H$-invariant, for every $n\ge 1$ we have
$
H_\cF^n(\cG)=\cG.
$
Fix an $x\in \cG$. 
Then, for every $n\ge 1$, there exist a point $x_n\in \cG$ and a multiindex $I_n\in \cI^m$ with $|I_n|=n$ such that
$
x=\Phi_\cF^{I_n}(x_n).
$

Since $K$ attracts every compact subset of $X$, in particular it attracts $\cG$ and so $e(\Phi_\cF^I(\cG),K)\to 0$ for $|I|\to\infty$.
Applying this to the multiindices $I_n$, we obtain $d\!\left(\Phi_\cF^{I_n}(x_n),K\right)\to 0$.
Since $\Phi_\cF^{I_n}(x_n)=x$ for every $n$, it follows that $d(x,K)=0$, so that $x\in K$.
Since $x\in \cG$ was arbitrary, this proves that
$
\cG\subset K,
$
namely $\Gl$ is minimal.
\end{proof}

\medskip\noindent
{\bf Compact Dynamics.}
    We say that $\cF$ has {\bf compact dynamics} if it has a global attractor $\cG$.
From now on, we will denote by {\BF$\Gl$} the global attractor of an IFS $\cF$ with compact dynamics.
\begin{proposition}
    \label{prop: omega(K) subset G}
    Assume that $\cF$ has compact dynamics.
    Then 
    $\Om(K)\subset\Gl$ for every compact set $K\subset X$.
    In particular, $\Om(x)\subset\Gl$ for every $x\in X$.
    Moreover, if $K$ is forward-invariant, $\Om(K)$ is $H$-invariant.
\end{proposition}
\begin{proof}
    Let $K\subset X$ be compact. 
    Since $\Gl$ attracts $K$, $e(\P(K),\Gl)\to 0$ for $|I|\to\infty$.
    If $y\in\Om(K)$, there exist $x_n\in K$ and multiindices $I_n$ with $|I_n|\to\infty$ such that $\Pn(x_n)\to y$.
    Hence $y\in N_\eps(\Gl)$ for every $\eps>0$. Since $\Gl$ is closed, this is possible only if $y\in \Gl$. 
    Therefore, $\Om(K)\subset \Gl$.

    Assume now that $K$ is forward-invariant. 
    By Proposition~\ref{prop: forward invariant}, $\Om(K)$ is forward-invariant and so $H_\cF\bigl(\Om(K)\bigr)\subset \Om(K)$.

    For the reverse inclusion, let $y\in\Om(K)$. 
    Then there exist points $x_n\in K$ and multiindices $I_n\in I^m$ such that $|I_n|\to\infty$ and $\Pn(x_n)\to y$.
    For all sufficiently large $n$, write $I_n=J_n i_n$, $i_n\in\{1,\ldots,m\}$.
    Since the alphabet is finite, after passing to a subsequence we may assume that $i_n=i$ is constant. 
    Set $z_n=\Phi_\cF^{J_n}(x_n)$.
    Since $K$ is forward-invariant, $z_n\in K$ for every $n$. 
    Since $K$ is compact, after passing to a further subsequence we may assume that $z_n\to z\in K$.
    Moreover, $|J_n|=|I_n|-1\to\infty$, so, by definition, $z\in\Om(K)$.
Then
\[
y
=
\lim_{n\to\infty}\Pn(x_n)
=
\lim_{n\to\infty}f_i(z_n)
=
f_i(z),
\]
    so that $y\in f_i\bigl(\Om(K)\bigr)\subset H_\cF\bigl(\Om(K)\bigr)$.
    Hence $\Om(K)\subset H_\cF\bigl(\Om(K)\bigr)$ and so $H_\cF\bigl(\Om(K)\bigr)=\Om(K)$.
\end{proof}

\begin{lemma}
    \label{lemma: O(A)=A}
    Let $A\subset X$ be closed and forward-invariant.
    Then $\Om(A)\subset A$.
\end{lemma}
\begin{proof}
    Assume first that $A$ is closed and forward-invariant. 
    If $y\in\Om(A)$, then there exist $x_n\in A$ and multiindices $I_n$ with $|I_n|\to\infty$ such that $\Phi_\cF^{I_n}(x_n)\to y$.
    Since $A$ is forward-invariant, $\Phi_\cF^{I_n}(x_n)\in A$ for every $n$; since $A$ is closed, it follows that $y\in A$. 
    Hence, $\Om(A)\subset A$.
%
\end{proof}


\begin{lemma}
\label{lemma: Om}
    Assume that $\cF$ has compact dynamics. 
    Then, for every $x\in X$,
\[
\Om(x)\cap \owR \neq \varnothing .
\]
    In particular, $\owR\neq\emptyset$ and, for every $x\in X$, $\Om(x)\neq\emptyset$.
\end{lemma}

\begin{proof}
    Let \(x\in X\). 
    Since \(\Gl\) is compact and \(X\) is locally compact, $\Gl$ has a compact neighbourhood \(K\). 
    Since \(\Gl\) attracts  \(\{x\}\), there exists \(n_0\ge 0\) such that $\P(x)\in K$ for $|I|\ge n_0$.
    Hence, 
    $$
    \Om(x)=\bigcap_{n\ge n_0}\overline{\bigcup_{|I|\ge n}\P(x)}
    $$ 
    is a decreasing intersection of non-empty compact subsets of \(K\) and so is non-empty and compact. 
    By Proposition~\ref{prop: forward invariant} it is also forward-invariant.
    Now choose, by Zorn's lemma, a minimal non-empty compact forward-invariant subset $M\subset \Om(x)$.
    We claim that every point of \(M\) is orbitwise recurrent. 
    Let \(p\in M\) and let \(z\in \cO_\cF(p)\). 
    Since \(M\) is forward-invariant, \(z\in M\), and hence $\Om(z)\subset M$.
    Moreover, \(\Om(z)\) is non-empty, compact and forward-invariant.
    By the minimality of \(M\), we must have $\Om(z)=M$.   
    In particular, $p\in \Om(z)$.
    Since this holds for every \(z\in \cO_\cF(p)\), we have \(p\in \owR\).
    Thus, \(M\subset \owR\), and, since \(M\subset\Om(x)\), we obtain
    $\Om(x)\cap \owR\neq\varnothing$.
    The remaining assertions follow immediately.
\end{proof}

Notice that, in general, one cannot conclude that every point of $\Om(x)$ is orbitwise recurrent (or even just recurrent); the lemma only guarantees the existence of at least one orbitwise recurrent point in $\Om(x)$.


%

%

%
{\bf 
%

%
%

%
}

\begin{proposition}
\label{prop: R subset G}
Assume that $\cF$ has compact dynamics.
Then $\cR_\cF\subset\Gl$.
\end{proposition}

\begin{proof}
Let $x \in \cR_\cF$. 
Then $x \in \Om(x)$.
Since $\{x\}$ is compact and $\Gl$ attracts every compact subset of $X$,
Proposition~\ref{prop: omega(K) subset G} implies $\Om(x) \subset \Gl$.
Hence $x \in \Gl$ and so $\cR_\cF \subset \Gl$.
\end{proof}

\medskip\noindent
{\bf Trapping regions.}
Global attractors are often sets with a highly complicated structure (for instance, they are often not locally connected) and it is in general a hard problem finding out directly their existence. 
The concept of {\em trapping region} helps proving their existence.
Unlike global attractors, trapping regions are not unique and often are sets with an elementary structure, such as closed balls.
%
%
%

    We say that $Q\subset X$ is a {\bf trapping region} for $\cF$ if $Q$ is closed and forward invariant under $\Phi_\cF$.
    Given a trapping region $Q$, we denote by {\BF $\cF_Q$} the restriction of $\cF$ to $Q$, namely $\cF_Q=\{f_1|_Q,\dots,f_m|_Q\}$.
    We say that a trapping region $Q$ is {\bf global} if $Q$ absorbs, under $\Phi_\cF$, every compact set $K\subset X$.
    We denote by {\BF $\cK_\cF$} the set of all compact global trapping regions of $\cF$.
%


%
\begin{proposition}
    For any IFS $\cF$, the set $\cK_\cF$ is closed under $H_\cF$ and under finite intersections, namely:
    \begin{enumerate}
        \item if $Q\in\cK_\cF$ then $H_\cF(Q)\in\cK_\cF$;
        \item if $Q_1,Q_2\in\cK_\cF$, then $Q_1\cap Q_2\in\cK_\cF$.
    \end{enumerate}
    %
    
\end{proposition}
We omit the elementary proof.

\medskip
A fundamental role of trapping regions, as illustrated by the following result, is that one can replace the whole phase space $X$ by any global trapping region when studying the asymptotics of $\Phi_\cF$.
This means that {\bf it is the topological properties of compact global trapping regions, rather than those of the whole phase space $X$, that  play a role in the qualitative dynamics of an IFS with compact dynamics}.
\begin{proposition}
    \label{prop: G is natural under restrictions}
    Assume that $\cF$ has compact dynamics and let $Q\in\cK_\cF$.
    Then $\cF_Q$ has compact dynamics and $\cG_{\cF_Q}=\cG_\cF$.
\end{proposition}
\begin{proof}
    We first prove that $\Gl \subset Q.$
    Indeed, since $\Gl$ is compact and $Q$ absorbs every compact subset of $X$, there exists $\tau\ge 0$ such that $\Phi_\cF^I(\Gl)\subset Q$ for every  $I\in\cI^m_\tau$.
    Since $\Gl$ is $H$-invariant, for every $x\in \Gl$ and every $n\ge 1$ there exist $x_n\in \Gl$ and $I_n\in I^m$ with $|I_n|=n$ such that
    $
    x=\Phi_\cF^{I_n}(x_n).
    $
    Taking $n\ge \tau$, we get $x=\Phi_\cF^{I_n}(x_n)\in Q$. 
    Hence, $\Gl\subset Q$.

    As a subset of \(Q\), the set \(\Gl\) is compact, $H$-invariant under \(\cF_Q\), and attracts every compact subset of \(Q\), because it attracts every compact subset of \(X\). Hence, \(\Gl\) is a global attractor for \(\cF_Q\). 
    By Proposition~\ref{prop: G is unique}, \(\cF_Q\) has compact dynamics and ${\cal G}_{\cF_Q}=\Gl$.
\end{proof}

The next result also shows another fundamental side of compact global trapping regions: they grant the existence of a global attractor.
\begin{theorem}[Existence of a global attractor]
    \label{thm: Q => G}
    Assume that $\cF$ has a compact global trapping region $Q$. 
    Then $\cF$ has compact dynamics and $$\Gl=\Om(Q).$$
\end{theorem}
%
\begin{proof}
Recall that
\[
\Om(Q)=\bigcap_{n\ge 0} E_n,
\qquad
E_n=\bigcup_{|I|\ge n}\P(Q).
\]
Since $Q$ is forward-invariant, for every $n\ge 0$ we have
\[
E_n=\bigcup_{|I|=n}\P(Q).
\]
Hence each $E_n$ is a compact subset of $Q$, and
\[
E_{n+1}\subset E_n
\qquad\text{for all }n\ge 0.
\]
Therefore
\[
\Om(Q)=\bigcap_{n\ge 0}E_n
\]
is a non-empty compact subset of $Q$.

By Proposition~\ref{prop: forward invariant}, $\Omega_\cF(Q)$ is forward-invariant.
We now show that $\Om(Q)$ is also $H$-invariant. 
By its forward-invariance, 
\[
H_\cF(\Om(Q))\subset \Om(Q).
\]
It remains to prove the opposite inclusion.

Fix $x\in \Om(Q)$. Since $x\in E_n$ for every $n\ge 1$, for each $n\ge 1$ there exist $q_n\in Q$ and
a multiindex
\[
I_n=J_ni_n,\qquad |I_n|=n,
\]
such that
\[
x=\Phi_\cF^{I_n}(q_n)=f_{i_n}\bigl(\Phi_\cF^{J_n}(q_n)\bigr).
\]
Set
\[
y_n=\Phi_\cF^{J_n}(q_n).
\]
Then $y_n\in E_{n-1}\subset Q$ and
\[
x=f_{i_n}(y_n).
\]

Since $Q$ is compact and $\{1,\dots,m\}$ is finite, passing to a subsequence if necessary we may
assume that
\[
y_n\to y\in Q
\qquad\text{and}\qquad
i_n=i
\quad\text{for all }n.
\]
We claim that $y\in \Om(Q)$. 
Indeed, fix $p\ge 0$. 
If $n$ is sufficiently large, then $n-1\ge p$, and so
\[
y_n\in E_{n-1}\subset E_p.
\]
Since $E_p$ is closed, passing to the limit yields $y\in E_p$. Since $p$ was arbitrary, we conclude
that
\[
y\in \bigcap_{p\ge 0}E_p=\Om(Q).
\]
By continuity of $f_i$,
\[
x=\lim_{n\to\infty} f_i(y_n)=f_i(y),
\]
and since $y\in \Om(Q)$, this proves that $x\in H_\cF(\Om(Q))$. Hence
\[
\Om(Q)\subset H_\cF(\Om(Q)).
\]
Therefore
\[
H_\cF(\Om(Q))=\Om(Q),
\]
namely $\Om(Q)$ is $H$-invariant.

Now we prove that $\Om(Q)$ attracts every compact subset of $X$. 
Let $K\subset X$ be compact and
let $\varepsilon>0$. 
Since $Q$ is a global trapping region, there exists $N_0\ge 0$ such that
\[
\P(K)\subset Q \qquad\text{for every } I\in\cI^m_{N_0}.
\]
Since the family $(E_n)_{n\ge 0}$ is decreasing and
\[
\Om(Q)=\bigcap_{n\ge 0}E_n,
\]
compactness implies that there exists $n\ge 0$ such that
\[
e(E_n,\Om(Q))<\varepsilon.
\]
Now let $I\in\cI^m$ with $|I|\ge N_0+n$. Write $I=LJ$ with $|L|=N_0$ and $|J|\ge n$. Then
\[
\P(K)=\Phi_\cF^J(\Phi_\cF^L(K))\subset \Phi_\cF^J(Q)\subset E_n.
\]
Hence
\[
e(\P(K),\Om(Q))\le e(E_n,\Om(Q))<\varepsilon
\]
for every $|I|\ge N_0+n$. This proves that $\Om(Q)$ attracts every compact subset of $X$.

Thus, $\Om(Q)$ is compact and $H$-invariant and attracts every compact subset of $X$.
Hence, $\cF$ has compact dynamics and $\Gl=\Om(Q)$.
\end{proof}
\begin{corollary}[IFSs over compact spaces have compact dynamics]
    \label{cor: X compact => F has compact dynamics}
    Assume that $X$ is compact.
    Then any IFS $\cF$ on $X$ has compact dynamics and $\Gl=\Om(X)$
\end{corollary}
\begin{proof}
    By definition $X$ is forward-invariant and trivially absorbs every compact subset of itself, so $X$ is a compact global trapping region for $\cF$.
    Hence, the claim follows immediately by Theorem~\ref{thm: Q => G} with $Q=X$.
\end{proof}
\begin{corollary}[Contractive IFSs over locally compact complete spaces have compact dynamics]
    \label{cor: K=G_F}
    Assume that $\cF$ is a contractive IFS and that $X$ is locally compact and complete.
    Then $\cF$ has compact dynamics and has a unique $H$-invariant compact set $K$ and $\Gl=K$.
\end{corollary}
\begin{proof}
    Denote by $\lambda\in(0,1)$ a common contraction constant for its generators. 
    Since $X$ is complete, the hyperspace
    $\KX$ is complete, and the Hutchinson map $H_\cF$ is a contraction (see Section~\ref{sec: Hutch} for details on this point). 
    Hence, by Theorem~\ref{thm: Hutch}, $H_\cF$ has a unique fixed point $K\in\KX$ -- equivalently, $\cF$ has a unique nonempty $H$-invariant compact set $K\subset X$.

    By local compactness and compactness of $K$, there exists an $\eps>0$ such that $\overline{N_\eps(K)}$
    is compact.
    If $x\in \overline{N_\eps(K)}$, then $d(x,K)\leq\varepsilon$, and therefore
\[
d\bigl(f_i(x),K\bigr)
\leq\lambda d(x,K)
\leq\lambda\varepsilon
<\varepsilon.
\]
Hence
\[
f_i(\overline{N_\eps(K)})\subset N_\varepsilon(K)\subset \overline{N_\eps(K)},
\]
so $\overline{N_\eps(K)}$ is forward-invariant.

    Let now $C\subset X$ be compact and set $M:=\sup_{x\in C}d(x,K)<\infty$.
    Choose $N$ such that $\lambda^NM<\varepsilon$. 
    Then, for every $|I|\geq N$,
\[
\P(C)\subset N_\varepsilon(K)\subset \overline{N_\eps(K)}.
\]
    Thus, $\overline{N_\eps(K)}$ is a compact global trapping region and so,
    by Theorem~\ref{thm: Q => G}, $\cF$ has compact dynamics. 
    Since its global attractor is compact and $H$-invariant, the uniqueness of the fixed point of $H_\cF$ gives $\Gl=K$.
\end{proof}

\medskip\noindent
{\bf Example: trapping regions of the Sierpi\'nsky IFS.}
In Section~\ref{sec: IFS}, we discussed the recurrent set and global attractor of a contractive IFS.
Here we illustrate some global trapping regions of a particular contractive IFS, the Sierpi\'nski IFS, namely the contractive IFS whose attractor is the Sierpi\'nski gasket.
This is the IFS $\cS=\{s_A,s_B,s_C\}$ on $\bR^2$ where $A,B,C$ are the vertices of an equilateral triangle $T$ of side 1 and the map $s_P$, where $P=A,B,C$, is a homothety by a factor $1/2$ centered at $P$.
The phase space of $\cS$ is not compact but $\cS$ has compact dynamics. 

\begin{figure}
    \centering
    \includegraphics[width=0.5\linewidth]{sierp}
    \caption{{\bf The Sierpi\'nski gasket.} The picture shows the equilateral triangle $T$ with vertices $A,B,C$ and the set $H^4_\cS(T)$. 
    Triangles are colored based on the last map they are the image of.
    Recall that the Sierpi\'nski gasket is obtained as the limit of $H^k_\cS(T)$ for $k\to\infty$.
    The disc $Q$ is a compact global trapping region for the Sierpi\'nski IFS, $H_\cS(Q)$ is its image under the Hutchinson map.}
    \label{fig: Sierp}
\end{figure}

A quick way to see this is to show that this IFS admits compact global trapping regions. 
Indeed, let $Q$ be any closed disc centered at the barycenter of $T$ and containing some neighborhood of $T$.
Notice that $s_A(x)=\frac{x+A}{2}$ and similarly for the other vertices.
Since $Q$ is convex and each vertex of $T$ is in $Q$, it follows at once that $Q$ is forward-invariant under $\Phi_\cF$.
Moreover, let $\delta>0$ such that $N_\delta(T)\subset Q$.
Given any compact $L\subset\bR^2$, standard arguments show that $\sup_{x\in L}d(\P(x),T)$ can be made arbitrarily small for $|I|\to\infty$, namely there is an $N\geq1$ such that $\P(L)\subset N_\delta(T)$ for all $|I|\geq N$.
Hence, $Q$ absorbs all compact sets of $\bR^2$ and so $Q$ is a compact global trapping region for $\Phi_\cS$.
In Figure~\ref{fig: Sierp} we show a particular instance of $Q$ and its first image under $H_\cS$, together with an approximation of the Sierpi\'nski gasket.

Notice that, although also the triangle $T$ is compact and forward-invariant, it is not a global trapping region.
Indeed, rather than absorbing all compact sets, it only attracts them. 
On the other side, $T$ is not the global attractor because it is not $H$-invariant.
We recall the reader that, by Corollary~\ref{cor: K=G_F}, the global attractor of $\Phi_\cS$ is precisely its unique attractor, namely the Sierpi\'nski gasket.

\medskip\noindent
{\bf Trapping regions and global attractors.}
Below we show that global attractors  always arise as the limit set of a compact global trapping region.

%
\begin{lemma}
    \label{lemma: for-inv}
    Assume that $\cF$ has compact dynamics.
    Then, for every $\eps>0$, there is a compact global trapping region $Q$ such that $Q\subset N_\eps(\Gl)$.
\end{lemma}
\begin{proof}

Fix $\varepsilon>0$. Since $X$ is locally compact and $\Gl$ is compact,
there exists a compact neighborhood $K$ of $\Gl$ such that
\[
\Gl \subset \operatorname{int}(K)\subset K \subset N_\varepsilon(\Gl).
\]

Since $\Gl$ attracts $K$, and $\operatorname{int}(K)$ is a neighborhood of
$\Gl$, there exists $N\ge 0$ such that
\[
\P(K)\subset \operatorname{int}(K)
\qquad \text{for every } I\in\cI^m_N .
\]

Define
\[
Q=\overline{\bigcup_{|I|\ge N}\P(K)}.
\]

Since
\[
\bigcup_{|I|\ge N}\P(K)\subset K,
\]
the set $Q$ is a compact subset of $K$, hence
\[
Q\subset K\subset N_\varepsilon(\Gl).
\]

We now prove that $Q$ is forward-invariant. Let $i\in\{1,\dots,m\}$. Then
\[
f_i\!\left(\bigcup_{|I|\ge N}\P(K)\right)
\subset
\bigcup_{|J|\ge N+1}\PJ(K)
\subset
\bigcup_{|J|\ge N}\PJ(K).
\]

By continuity of $f_i$,
\[
f_i(Q)
=
f_i\!\left(\overline{\bigcup_{|I|\ge N}\P(K)}\right)
\subset
\overline{f_i\!\left(\bigcup_{|I|\ge N}\P(K)\right)}
\subset Q.
\]

Thus $Q$ is forward-invariant.

Finally we prove that $Q$ absorbs all compact sets.
Let $C\subset X$ be compact.
Since $\Gl$ attracts $C$ and $\operatorname{int}(K)$ is a neighborhood of
$\Gl$, there exists $M\ge 0$ such that
\[
\PJ(C)\subset K
\qquad \text{for every } J\in\cI^m_M .
\]

Now let $L\in \cI^m_{M+N}$. 
Write $L=VU$ with $|V|=M$ and $|U|\ge N$. 
Then
\[
\Phi_\cF^L(C)=\Phi_\cF^U(\Phi_\cF^V(C))\subset \Phi_\cF^U(K)
\subset \bigcup_{|I|\ge N}\Phi_\cF^I(K)\subset Q.
\]

Hence
\[
\Phi_\cF^L(C)\subset Q
\qquad \text{for every } L\in \cI^m_{M+N}.
\]

Therefore, $Q$ absorbs every compact set, so it is a global trapping region.
\end{proof}
%
\begin{theorem}
    \label{thm: compact dynamics = compact gl tr reg}
    $\cF$ has compact dynamics if and only if it has a compact global trapping region $Q$.
\end{theorem}
\begin{proof}
If $\cF$ has a compact global trapping region, then $\cF$ has compact dynamics by Theorem~\ref{thm: Q => G}.

Conversely, if $\cF$ has compact dynamics, then by Lemma~\ref{lemma: for-inv}, for every
$\varepsilon>0$ there exists a compact global trapping region
$Q\subset N_\varepsilon(\Gl)$. Hence $\cF$ has a compact global trapping
region.
\end{proof}

We are now able to provide a further characterization of the global attractor.

\begin{proposition}
\label{prop: char of G}
    Assume that $\cF$ has compact dynamics.
    Then 
    $\Gl$ is contained in every closed set of $X$ that attracts every compact subset of $X$.
\end{proposition}

\begin{proof}
Let $Q\in \cK_\cF$ be a compact global trapping region for $\cF$, whose existence is guaranteed by
Theorem~\ref{thm: compact dynamics = compact gl tr reg}. By Theorem~\ref{thm: Q => G},
\[
\Gl=\Om(Q)=\bigcap_{n\ge 0} E_n,
\qquad
E_n=\bigcup_{|I|\ge n}\Phi_\cF^I(Q).
\]

Let $B\subset X$ be closed and assume that $B$ attracts every compact subset of $X$.
Since $Q$ is compact, $B$ attracts $Q$. 
Therefore, for every $\varepsilon>0$, there exists $\tau\ge 0$ such that $\P(Q)\subset N_\varepsilon(B)$ for every $I\in \cI^m_\tau$.
It follows that
\[
E_\tau=\bigcup_{|I|\ge \tau}\P(Q)\subset N_\varepsilon(B).
\]
Since $\Gl\subset E_\tau$, we obtain $\Gl\subset N_\varepsilon(B)$.
Since $\varepsilon>0$ is arbitrary and $B$ is closed, this implies $\Gl\subset B$.
In particular, if $Q$ is a closed set that attracts every compact subset of $X$, then
$\Gl\subset Q$.
\end{proof}

\section{Streams and graphs of an IFS}
\label{sec: streams}


The idea of representing the qualitative dynamics of a semiflow with a graph appeared explicitly in literature, to our knowledge, in a seminal article by Stephen Smale~\cite{Sma67}, although it is implicit already in some work of Morse~\cite{Mor25}.
Smale's idea, which was limited to Axiom-A systems, was to decompose the recurrent set of system $F$ into elementary indecomposable compact ``blocks'' $N_i$, which are the nodes of the graph, and to draw an edge from node $N_i$ to node $N_j$, $i\neq j$, if there is bi-infinite sequence $\{x_n\}_{n\in\bZ}$ (which we call bitrajectory) with $x_{n+1}=F(x_n)$ for every $n$ whose forward limit points lie in $N_j$ and backward limit points lie in $N_i$. 
It turns out that each point has either a recurrent-type dynamics, if it belongs to a node, or a gradient-type dynamics, namely belongs to a bitrajectory that defines an edge of the graph.  

In~\cite{DLY25} we started developing a theory able to cover all semiflows, from discrete-time one-dimensional system to continuous-time infinite-dimensional ones. 
The optimal tool we devised to achieve this goal is the concept of {\em stream}, a closed and transitive extension of the orbit relation of the semiflow.
Given a semiflow $F$, any stream of $F$ defines a concept of {\em being downstream} and this way provides a natural way to extend the set of recurrent point of $F$ to a set that is closed and invariant and provides as well a natural way to decompose this set into elementary blocks.
These blocks can be represented as nodes of a graph and the transitivity of the relation allows to define edges between nodes so that there is an edge from node $N$ to node $M$ if points in $M$ are downstream from points $N$.

The main goal of this article is to improve our understanding of who are the key ingredients of the qualitative dynamics of an IFS.
To this end, in this section we extend the definition of stream to IFSs and prove some of its fundamental properties.
Then, in Section~\ref{sec: chains}, we will concentrate on the chain streams and provide for IFSs a similar decomposition of the types of dynamics that a point can have.  

\medskip\noindent
{\bf Stream of an IFS.}
Recall that a {\bf quasi-order} on $X$ is a binary relation on $X$ which is reflexive and transitive.
We call {\BF stream} of an IFS $\cF$ a closed quasi-order $S$ that contains $\cO_\cF$.
%
When $(x,y)\in S$, we say that $y$ is $S$-downstream from $x$ and we use the notation $$\text{{\BF $x\Sto y$}.}$$
Sometimes we use also the notation {\BF$y\in\Down_S(x)$} or {\BF$x\in\Up_S(y)$}.
When $x\Sto y$ and $y\Sto x$, we write 
$$\text{{\BF$x\Seq y$},}$$ and we say that $x$ and $y$ are {\BF$S$-equivalent}.

Unlike the case of semiflows, for $m>1$ we have more than one generator and so there are several possible definitions of ``recurrence''.
We say that $x$ is {\BF $S$-recurrent} if there is a generator $f_i$ such that $f_i(x)\Seq x$ and that $x$ is {\BF orbitwise $S$-recurrent} if $z\Seq x$ for every $z\in\cO_\cF(x)$.
    We denote by {\BF$\cR_S$} the set of all $S$-recurrent points and by {\BF$\cR^{ow}_S$} the set of all orbitwise $S$-recurrent points.
    We say that a set $M\subset\cR_{S}$ is $S$-equivalent if all points of $M$ are mutually $S$-equivalent.
    We call {\bf nodes} (resp. {\bf orbitwise nodes}) of $S$ the equivalence classes of $\cR_{S}$ (resp. of $\cR^{ow}_S$) with respect to $\Seq$.
    If $x\in\cR_S$, we denote by {\BF$\Node_S(x)$} the node of $S$ containing $x$.
    If $x\in\cR^{ow}_S$, we denote by {\BF$\Node^{ow}_S(x)$} the orbitwise node of $S$ containing $x$.
%
Notice that, as implicitly assumed in the definition above, $\Seq$ is an equivalence relation.

\begin{proposition}
    Each orbitwise node of a stream $S$ is a subset of a unique node of $S$.
\end{proposition}
\begin{proof}
    Let $N$ be an orbitwise node of $S$.
    The reader can verify that $N$ is an $S$-equivalent set and so it lies in a node of $S$, which proves the claim.
\end{proof}

\medskip\noindent
{\bf The largest and the smallest stream.} 
One way to give a concrete meaning to streams is either to think that $\cF$ is subject to some control or that $\cF$ is an approximated model of some natural or abstract phenomenon one wants to study.
In this framework, $x\Sto y$ for $y\not\in\cO_\cF(x)$ means that $y$ is reachable from $x$ by steering the orbit as the point moves under $\cF$ or that $y$ lies on an orbit compatible with the approximation of the model.
The reader can verify the following basic results:
\begin{enumerate}
    \item {\em $X\times X$ is a stream, in fact the largest stream of each IFS};
    \item {\em the intersection of streams is a stream, so each IFS has a smallest stream.}
\end{enumerate}
The largest stream corresponds to the trivial situation where the controls (or the errors) are so large that every point can be reached from any other point of the phase space.
In this case, the dynamics of $\cF$ becomes irrelevant.

At the other end, the smallest stream is the one most faithful to the dynamics of $\cF$: $x\Sto y$ means that $y$ can be reached from $x$ by making arbitrarily small ``adjustments'' to the orbit of $x$.
For instance, consider the time-1 map $f$ of the ODE 
$$
x'=|x(x-1)(x-2)|.
$$
This system has three fixed points, $x=0,1,2$, and every other point moves rightward. 
Each point in $(0,1)$ asymptotes to 1 and every point in $(1,2)$ asymptotes to 2. 
Given any stream $S$ of this system, every point is downstream from 0.
Indeed, every point $x\in (0,1)$ has a backward orbit $x_n$ (namely $f(x_n)=x_{n-1}$ asymptoting (backward in time) to 0.
Since $x_n\Fto x$, then $x_n\Sto x$ and, since $S$ is closed, $0=\lim x_n\Sto x$.
Moreover, $f^\KX$ asymptotes to 1 and so $0\Sto1$. 
By the same type of argument, $0\Sto x$ for every $x\in[0,2]$.
The reason for this is that an arbitrarily small control/error would be able, for the orbit of a point $x\in(0,1)$, to bypass the fixed point $x=1$: it is enough to wait enough iterates so that $f^\KX$ is close enough to $1$ and then make the point jump over 1 and keep moving in the $(1,2)$ segment.

The reader is invited to look at Section~\ref{sec: chains} for a concrete and arguably the most important type of streams, namely streams of chains.
In case of $\eps$-chains, one allows a control/error of $\eps$ at every step.
Continuing the Sierpi\'nski IFS example, we exhibit there the recurrent sets of the smallest stream and of other ``looser'' streams of $\eps$-chains.

\medskip\noindent
{\bf Streams, limit sets and recurrent sets.}
Before introducing streams of chains, below we discuss some properties of general streams, in particular their invariance and their relation with limit sets and recurrence.
In particular, these results show two important strong points of streams: 1. for any stream $S$, the set of $S$-recurrent points is closed, while $\cR_\cF$ is not necessarily so; 2. moreover, $S$ provides a natural way to sort the points of $\cR_S$ into equivalence classes (nodes) which are themselves all closed sets.
Nodes are a natural generalization of periodic orbits, namely ``invariant'' (in some sense) closed sets such that any other point is ``dynamically equivalent'' to every other point of the set.
In case of the chain stream (see Section~\ref{sec: chains}) we are able to say even more: just as periodic orbits, each node is compact and entirely contained in the global attractor.


%
\begin{proposition}
    \label{prop: streams}
    Let $S$ be a stream of $\cF$. Then: 
    \begin{enumerate}
    \item $\cR^{ow}_S$ and each of its nodes are forward-invariant under $\Phi_\cF$;
    \item $\cR_S$, $\cR^{ow}_S$ and each of their nodes are closed;
    \item $\cR_\cF\subset\cR_\cS$ and $\cR^{ow}_\cF\subset\cR^{ow}_\cS$;
    \item if $\cF$ has compact dynamics, then $\Om(x)\cap\cR^{ow}_S\neq\emptyset$ for every $x\in X$;
    \item for every $x\in X$, $$\Down_S(x) \supset \cO_\cF(x)\bigcup\operatorname{Down}_S(\Om(x)).$$
    \end{enumerate}
\end{proposition}
%

\begin{proof}
\noindent
(1) Let $x\in \cR^{ow}_S$.
To prove the forward-invariance of $\cR^{ow}_S$ and $\Node^{ow}_S(x)$, it is enough to prove that $\cO_\cF(x)\subset \Node^{ow}_S(x)$.
By definition, $x$ is $S$-equivalent to every point on its orbit.
By transitivity, this means that every point of $\cO_\cF(x)$ is equivalent to every other point on that orbit.
Hence, $\cO_\cF(x)\subset \Node^{ow}_S(x)$.

\medskip
\noindent
    (2) We prove the claim in case of orbitwise $S$-recurrence, the reader can adapt the argument to the other case.
    Let $x_n\in \cR^{ow}_S$ converge to some $x\in X$.
    Since each $x_n$ is orbitwise $S$-recurrent, for every $I\in\cI^m$ we have that $x_n\Seq\P(x_n)$.
    Since $S$ is closed, taking the limit we get that $x\Seq\P(x)$ for every $I\in\cI^m$, namely $x$ is $S$-equivalent to each point on its orbit.
    Hence, $x\in\cR^{ow}_\cS$, namely $\cR^{ow}_\cS$ is closed.

    Moreover, assume that the $x_n$ all belong to the same orbitwise node.
    Then $x_n\Seq x_m$ for every $n,m=1,2,\dots.$
    Hence, taking the limit for $n\to\infty$ and leaving $m$ fixed, we get that $x\Seq x_m$ for all $m$, so $x$ belongs to the same node as well.
    Therefore, each orbitwise node is closed.


\medskip\noindent
    (3) We proceed as in (2).
    Let $x\in \cR^{ow}_\cF$. 
    We need to prove that all points of $\cO_\cF(x)$ are $S$-equivalent to each other.

    Recall first of all that $\cO_\cF\subset S$, so $x\Sto y$ for all $y\in\cO_\cF(x)$.
    Since $x$ is orbitwise recurrent under $\cF$, we also know that, for every $y\in\cO_\cF(x)$, $x\in\Omega_\cF(y)$, namely there are $I_n\in\cI^m$ with $|I_n|\to\infty$ and $\Pn(y)\to x$.
    Since $y\Sto\Pn(y)$ for every $n$ and $S$ is closed, we get that $y\Sto x$.
    Hence, $y\Seq x$.
    Since $y$ was arbitrary, this shows that the orbit of $x$ is an $S$-equivalent set, so $x\in\cR^{ow}_S$.

\medskip\noindent
    (4) By Lemma~\ref{lemma: Om}, $\Om(x)\cap \cR^{ow}_\cF\neq\varnothing$ for every $x\in X$.
    By (3), $\cR^{ow}_\cF\subset \cR^{ow}_S$,so also $\Om(x)\cap \cR^{ow}_S\neq\varnothing$.


\medskip
\noindent
(5)
This is an immediate consequence of transitivity and closedness of $S.$
\end{proof}






\medskip\noindent
{\bf Closed quasi-orders and Lyapunov functions.}
A function $L\in C^0(X)$ is a {\bf Lyapunov function} for a quasi-order $D$ if $x\Dto y$ implies $L(x)\geq L(y)$.
We denote by $\text{\BF$\cL_D$}\subset C^0(X)$ the set of all Lyapunov functions of a given quasi-order $D$.

The set of Lyapunov functions of the quasi-order $\cO_\cF$ coincides with the set of Lyapunov functions of $\cF$, namely continuous functions $L$ such that $L(x)\geq L(\P(x))$ for all $I\in\cI^m$.
We denote this important set simply by {\BF$\cL_\cF$}.

One can also run things backward and use collections of functions to define quasi-orders in the following way.
Given a collection $\cS\subset C^0(X)$, we set
\beq
   \text{\BF$\cD_\cS$} = \{(x,y) : L(x)\geq L(y)\text{ for all }L\in \cS\}.
\eeq

The existence of non-trivial Lyapunov functions for a closed quasi-order is granted by the following fundamental result by Joseph Auslander.
\begin{lemmA}[Auslander, 1963~\cite{Aus63}, Thm.~4 (see also Lemma A in~\cite{DLY25})]
    \label{lem: A}
    Let $X$ be separable and locally compact and let $D$ be a closed quasi-order on $X$. 
    Then, if $x\not\Dto y$ (equivalently, if $(x,y)\not\in D$), there exists a function $L\in\cL_D$ such that $L(x)<L(y)$.
\end{lemmA}
This lemma leads to the following fundamental result.
\begin{thmX}[Auslander, 1963~\cite{Aus63} (see Theorem A in~\cite{DLY25})]
    \label{thm: Joe}
    Let $X$ be separable and locally-compact and let $D$ be a closed quasi-order on $X$.
    Then 
    $
    D=\cD_{\cL_D}.
    $
\end{thmX}
In other words, on each separable locally compact space, each quasi-order has enough Lyapunov functions to completely characterize it.

\medskip\noindent
{\bf The smallest stream.}
The reader can verify that the intersection of any collection of streams of $\cF$ is a stream of $\cF$.
Since each IFS $\cF$ has a non-empty largest stream, the relation $X\times X$, it also has a non-empty smallest stream, which is the intersection of all of its streams.

We denote the smallest stream of $\cF$ by $\cA_\cF$ since the smallest stream of a semiflow was studied first in literature by Joseph Auslander~\cite{Aus63}.
This set has the following characterization in terms of Lyapunov functions.
\begin{proposition}
    \label{prop: A_F via Lyap functs}
    $\cA_\cF$ is the closed quasi-order generated by the set of Lyapunov functions of $\cF$: 
$$
\cA_\cF=\cD_{\cL_\cF}.
$$    
\end{proposition}
\begin{proof}
    The reader can verify that $\cD_{\cL_\cF}$ is a closed quasi-order containing $\cO_\cF$, hence a stream.
    If $S$ is any stream, then $\cO_\cF \subset S$ implies that $\cL_\cS \subset \cL_\cF$, hence $\cD_{\cL_\cF} \subset \cD_{\cL_S} = S$ by Theorem~\ref{thm: Joe}.
\end{proof}
In the next section we will show that, under mild conditions, the chain stream is the smallest stream of an IFS.

\paragraph{Strict Lyapunov functions.}
Let \(D\) be a quasi-order on \(X\). We call a continuous function
\(L\colon X\to\mathbb R\) a {\bf strict Lyapunov function} for \(D\) if
\[
x\succcurlyeq_D y \Longrightarrow L(x)\geq L(y)
\]
and
\[
x\succcurlyeq_D y,\qquad y\not\succcurlyeq_D x
\Longrightarrow L(x)>L(y).
\]

\begin{proposition}[Existence of a strict Lyapunov function]
    \label{prop:strict-lyapunov}
    Let \(D\) be a closed quasi-order on a separable locally compact metrizable space \(X\). 
    Then \(D\) admits a bounded strict Lyapunov function
\[
L_D\colon X\longrightarrow [0,1].
\]
    In particular, \(L_D\) is constant on every \(D\)-equivalence class.
\end{proposition}

\begin{proof}
    For every \((x,y)\notin D\), by Auslander's Lemma~\ref{lem: A} there exists a function \(\ell_{x,y}\in \cL_D\) such that
\[
\ell_{x,y}(x)<\ell_{x,y}(y).
\]
    For \(\ell\in \cL_D\), set
\[
U_\ell=\{(x,y)\in X\times X:\ell(x)<\ell(y)\}.
\]
    Each \(U_\ell\) is open, and the family
\[
\{U_\ell:\ell\in \cL_D\}
\]
covers \((X\times X)\setminus D\).

Since \(X\) is separable and metrizable, \(X\times X\) is second
countable. Hence the open subspace \((X\times X)\setminus D\) is
Lindelöf. We may therefore choose a sequence $\ell_1,\ell_2,\ldots\in \cL_D$ such that
\[
(X\times X)\setminus D
=
\bigcup_{n\geq1}U_{\ell_n}.
\]

Let
\[
\vartheta(t)=\frac12+\frac1\pi\arctan t
\]
and set
\[
u_n=\vartheta\circ\ell_n.
\]
    Since \(\vartheta\) is strictly increasing, each \(u_n\) belongs to \(\cL_D\), and \(0<u_n<1\). 
    Define
\[
L_D(x)=\sum_{n=1}^{\infty}2^{-n}u_n(x).
\]
    The series converges uniformly, so \(L_D\) is continuous and takes values in \([0,1]\).

    Suppose that \(x\succcurlyeq_D y\). Since each \(u_n\) is a Lyapunov function for \(D\),
\[
u_n(x)\geq u_n(y)
\qquad\text{for every }n,
\]
    and therefore $L_D(x)\geq L_D(y)$.
    Assume moreover that \(y\not\succcurlyeq_D x\). 
    Then $(y,x)\notin D$.
    By the choice of the sequence \((\ell_n)\), there exists \(n\) such that $\ell_n(y)<\ell_n(x)$, so that $u_n(y)<u_n(x)$.
    Since all the other differences \(u_k(x)-u_k(y)\) are nonnegative, we obtain
\[
L_D(x)-L_D(y)
=
\sum_{k=1}^{\infty}2^{-k}
\bigl(u_k(x)-u_k(y)\bigr)>0.
\]
    Thus, \(L_D(x)>L_D(y)\).

    Finally, if \(x\) and \(y\) are \(D\)-equivalent, then both \(x\succcurlyeq_D y\) and \(y\succcurlyeq_D x\), so $L_D(x)=L_D(y)$.
\end{proof}


\medskip\noindent
{\BF$\Omega$streams.} An immediate consequence of transitivity and closedness of $S$ is that 
$$
\Down_S(x) \supset \cO_\cF(x) \cup \Down_S(\Om(x)).
$$
This means that, given a point $x$, every point in every node that intersects $\Om(x)$ is downstream from $x$. 

We say that a stream $S$ is an {\BF$\Omega$stream} if we actually have equality:
     $$
     \Down_S(x) = \cO_\cF(x) \cup \Down_S(\Om(x))\text{ for every }x\in X.
     $$
This is a strong condition: it says that, if $x$ is not in a node, the nodes immediately downstream from $x$ are those that intersect $\Om(x)$.
In other words, between $x$ and any other node downstream from $x$, there is some node that contains limit points of $x$.
Hence, $\Omega$streams are the streams more strictly related to the dynamics of $\cF$.
\begin{lemma}
    \label{lemma: x owR => OM(x) in owR}
    Let \(S\) be a stream of \(\cF\). 
    If \(x\in \cR_S^{\mathrm{ow}}\), then $\Om(x)\subset\Node^{ow}_S(x)$.
\end{lemma}
\begin{proof}
    Let \(y\in\Om(x)\) and choose multiindices \(I_n\) such that $|I_n|\to\infty$ and $\Pn(x)\to y$.
    Since \(x\in \cR_S^{\mathrm{ow}}\), $\Pn(x)\Seq x$ for every \(n\). 
    By closedness of \(S\), $y\Seq x$.

    Now, let \(z\in \cO_\cF(y)\).
    Since \(\Om(x)\) is forward-invariant, $z\in\Om(x)$.
    Applying the preceding argument to \(z\), we obtain $z\Seq x\Seq y$.
    Thus, \(y\in \cR_S^{\mathrm{ow}}\). Moreover, \(y\Seq x\), so \(x\) and \(y\) belong to the same orbitwise node.
\end{proof}
\begin{proposition}
\label{prop: O-stream}
Assume that $\cF$ has compact dynamics and let $S$ be an $\Omega$-stream of $\cF$. 
Then: 
\begin{enumerate}
    \item for every $x\in\cR_S$, there is $y\in\Om(x)\subset\Gl$ such that $y\Seq x$, i.e. $\Node_S(x)=\Node_S(y)$.
    \item for every $x\in X$, $\Down_S(x)\cap \cR_S=\Down_S(\Om(x))\cap \cR_S$ ;
    \item $N\cap\Gl\neq\emptyset$ for every node $N$ of $S$.
\end{enumerate}
The same properties hold by replacing nodes and recurrent set with the corresponding orbitwise sets. 
\end{proposition}
\begin{proof}
    (1) Let \(x\in \cR_S\). 
    By definition, there exists a generator $f$ of $\cF$ such that $f(x)\Seq x$.
    If \(x\in \cO_\cF(f(x))\), then \(x\) belongs to a positive-length periodic orbit of the IFS. 
    Hence, $x\in\Om(x)$
    and the assertion follows by taking \(y=x\).
    Assume now that \(x\notin \cO_\cF(f(x))\). 
    Since \(S\) is an \(\Omega\)-stream, there exists $z\in\Om(f(x))$ such that $z\Seq x$.
    On the other hand, $x\Sto f(x)\Sto z$ and, consequently, \(z\Seq x\).

    Consider the family \(\cK\) of all nonempty compact forward-invariant sets \(K\subset\Om(f(x))\) such that
\[
K\cap\operatorname{Up}_S(x)\neq\varnothing.
\]
    This family is nonempty because \(\Om(f(x))\in\cK\).
    Moreover, every decreasing chain in \(\cK\) has an intersection belonging to \(\cK\). 
    Indeed, the sets are compact subsets of the compact set \(\Om(f(x))\), and their intersections with the closed set \(\operatorname{Up}_S(x)\) are nonempty nested compact sets. Hence, by Zorn's lemma, \(\cK\) has a minimal element \(K\).
    We claim that $\Om(K)\in\cK$.

    Indeed, the set \(\Om(K)\) is nonempty, compact, forward-invariant, and contained in \(K\). 
    It remains to show that it intersects \(\operatorname{Up}_S(x)\).

    Choose a $u\in K\cap\operatorname{Up}_S(x)$.
    If \(x\notin \cO_\cF(u)\), the \(\Omega\)-stream identity applied to \(u\Sto x\) gives a point $v\in\Om(u)\subset\Om(K)$ such that \(v\Sto x\).
    If \(x\in \cO_\cF(u)\), then \(x\in K\), because \(K\) is forward-invariant.
    Hence \(f(x)\in K\). 
    Since \(x\notin \cO_\cF(f(x))\), applying the \(\Omega\)-stream identity to \(f(x)\Sto x\) gives a $v\in\Om(f(x))\subset\Om(K)$ such that \(v\Sto x\).

    Thus, in either case, $\Om(K)\cap\operatorname{Up}_S(x)\neq\varnothing$, and therefore \(\Om(K)\in\cK\). 
    By minimality, $K=\Om(K)$.
    In particular, \(K\) is \(H\)-invariant.
    Hence, there exist \(w\in K\) and a generator \(f_j\) such that $f_j(w)=u$, so $w\Sto u\Sto x$.
    Moreover, since \(w\in K\subset\Om(f(x))\), we have also that $x\Sto f(x)\Sto w$.
    Thus, $w\Seq x$.
    It also follows that $u\Sto x\Sto w$, while \(w\Sto u\), and so $f_j(w)=u\Seq w$.
    Therefore, \(w\in \cR_S\). Finally, $w\in K\subset\Om(f(x))\subset\Om(x)$.
    Taking \(y=w\), this proves~(1).

    \medskip\noindent
    (2). The inclusion
\[
\operatorname{Down}_S\bigl(\Om(x)\bigr)\cap R_S
\subset
\operatorname{Down}_S(x)\cap R_S
\]
    follows from the closedness and transitivity of \(S\).
    For the reverse inclusion, let $y\in\operatorname{Down}_S(x)\cap \cR_S$.
    If \(y\notin \cO_\cF(x)\), the \(\Omega\)-stream identity gives directly $y\in\operatorname{Down}_S\bigl(\Om(x)\bigr)$.
    If \(y\in \cO_\cF(x)\), by part~(1), applied to \(y\), there is a point $z\in\Om(y)\cap \cR_S$ such that \(z\Seq y\). 
    Since $\Om(y)\subset\Om(x)$, we again obtain
\[
y\in\operatorname{Down}_S\bigl(\Om(x)\bigr).
\]

\medskip\noindent
    (3) Let \(N\) be an \(S\)-node and choose \(x\in N\). 
    By~(1), there exists $y\in\Om(x)\cap \cR_S$ such that \(y\Seq x\). Hence \(y\in N\). 
    Since \(\cF\) has compact dynamics, $\Om(x)\subset \Gl$, and therefore \(y\in N\cap \Gl\).

    \medskip\noindent
    With the help of Lemma~\ref{lemma: x owR => OM(x) in owR}, the argument above applies to the orbitwise case.
\end{proof}

\medskip\noindent
{\bf The graph of a stream.}
As mentioned at the beginning of the section, the qualitative properties of a stream of an IFS can be encoded in a graph.
%
Let $S$ be a stream of $\cF$. 
Given two nodes $A,B$ of $S$, we write $A\Sto B$ if and only if there are $x\in A$ and $y\in B$ such that $x\Sto y$.
The graph of $S$, denoted by {\bf\BF$\Gamma_S$}, is the directed graph whose nodes are the nodes of $\sR_S$ and such that there is an edge from a node $N_1$ to a node $N_2\neq N_1$ if $N_1\Sto N_2$.
    We say that a node is a {\bf bottom node} if all its edges are incoming and a {\bf top node} if all its edges are outgoing.

\medskip
Notice that the definition of edge is well posed since, by transitivity, if $x\Sto y$ for $x\in N_1$ and $y\in N_2$, then the same relation holds for every pair $x'\in N_1$ and $y'\in N_2$.
%

\medskip
We denote by {\BF$\Gamma^{ow}_S$} the graph built out of the orbitwise nodes.
Since every orbitwise node is contained in a unique node, there is a
natural node map $\Gamma_S^{\mathrm{ow}}\longrightarrow\Gamma_S$.
This map preserves downstream relations, although it need not be
injective. 
As shown below, every bottom node of \(\Gamma_S\) is itself
an orbitwise node.
The two graphs provide different information about the qualitative dynamics of the IFS:
$\Gamma_S$ about itineraries of points, $\Gamma^{ow}_S$ about the dynamics of the IFS as a semigroup.
As the following proposition shows, there is a type of nodes that are always common between the two graphs.    

\begin{proposition}
    Assume that \(\cF\) has compact dynamics and let \(S\) be a stream of \(\cF\).
    Then every bottom \(S\)-node is an orbitwise \(S\)-node.
\end{proposition}
\begin{proof}
    Let \(N\) be a bottom \(S\)-node and let \(x\in N\). We claim that \(x\) is orbitwise \(S\)-recurrent.

    Indeed, fix \(z\in \cO_\cF(x)\). 
    By Lemma~\ref{lemma: Om}, there exists $p\in\Om(z)\cap \cR_\cF^{\mathrm{ow}}$.
    Since \(p\in \cR_\cF^{\mathrm{ow}}\), we also have that $p\in \cR_S^{\mathrm{ow}}\subset \cR_S$.
    Indeed, if \(w\in \cO_\cF(p)\), then \(p\in\Om(w)\), and therefore closedness of \(S\) gives \(w\Sto  p\); the reverse relation \(p\Sto  w\) follows from \(\cO_\cF\subset S\).
    Since \(z\in \cO_\cF(x)\), we have $x\Sto  z$, so that $x\Sto  p$.
    
    Let \(P=\operatorname{Node}_S(p)\). 
    If \(P\neq N\), the relation \(x\Sto  p\) would determine an outgoing edge $N\to P$, contrary to the fact that \(N\) is a bottom node. 
    Hence \(P=N\). 
    In particular, $p\Sto  x$.
    It follows that $z\Sto  p\Sto  x$.
    Since \(x\Sto  z\), we conclude that $x\Seq z$.

    As \(z\in \cO_\cF(x)\) was arbitrary, \(x\in \cR_S^{\mathrm{ow}}\).
    Thus, \(N\subset \cR_S^{\mathrm{ow}}\). 
    Since \(R_S^{\mathrm{ow}}\subset R_S\), \(N=\Node^{ow}(x)\). 
\end{proof}

\medskip\noindent
{\bf Connectedness of the graph of a stream.}
Let $\cN_1,\cN_2$ be two non-empty mutually disjoint collections of nodes of $S$ such that each node of $S$ belongs to either $\cN_1$ or $\cN_2$ and denote by $N_1$ and $N_2$ the unions of all points in, respectively, the nodes  $\cN_1$ and $\cN_2$.
We say that $\cN_1,\cN_2$ are a {\bf nodes partition} of $\Gamma_{S}$ if $N_1,N_2$ are closed and disjoint.
We say that the graph $\Gamma_{S}$  is {\bf connected} if, for each  nodes partition $\cN_1,\cN_2$, there is an edge from a node of $\cN_1$ to a node of $\cN_2$ or viceversa.
Notice that, given any partition $\cN_1,\cN_2$, we have that $N_1\cup N_2=\cR_S$.


One of the main results of streams theory is that, when the global attractor of a semiflow is connected, the graph of any of its streams is connected as well.
We start here proving that the same result holds for the graph of the stream of an IFS.
We will conclude the proof in next section using streams of chains.
In next section we will also show with examples that this does not hold for the orbitwise graph.

\begin{lemma}
    Assume that \(\cF\) has compact dynamics and let \(S\) be a stream of \(\cF\).
    Then every \(S\)-node is upstream from an \(S\)-node containing an \(\cA_\cF\)-node.
    More precisely, for every \(S\)-node \(N\), there exist an \(S\)-node
    \(P\) and an \(\cA_\cF\)-node \(M\) such that
\[
M\subset P
\qquad\text{and}\qquad
N\Sto P.
\]
\end{lemma}

\begin{proof}
    Let \(N\) be an \(S\)-node and choose \(x\in N\). By Lemma~\ref{lemma: Om}, there exists $y\in\Om(x)\cap \owR$.
    By Proposition~\ref{prop: streams}(3), applied to the smallest stream \(\cA_\cF\), $\owR\subset \cR_{\cA_\cF}^{\mathrm{ow}}\subset \cR_{\cA_\cF}$.
    Hence, \(y\) belongs to an \(\cA_\cF\)-node, which we denote by
\[
M=\operatorname{Node}_{\cA_\cF}(y).
\]

    Since \(\cA_\cF\subset S\), we have \(\cR_{\cA_\cF}\subset \cR_S\), and every \(\cA_\cF\)-node is contained in a unique \(S\)-node. 
    Let $P=\operatorname{Node}_S(y)$.
    Then $M\subset P$.

    Moreover, since \(y\in\Omega_F(x)\), there exist multiindices \(I_n\), with \(\lvert I_n\rvert\to\infty\), such that $\Pn(x)\to y$.
    Since \(S\) contains the orbit relation,
\[
x\Sto\Phi_\cF^{I_n}(x)
\]
    for every \(n\). 
    By the closedness of \(S\), passing to the limit gives $x\Sto y$.
    Therefore, $N\Sto P$.
\end{proof}
\begin{proposition}
    Assume that \(\cF\) has compact dynamics. 
    If the graph \(\Gamma_{\cA_\cF}\) of the smallest stream is connected, then the graph \(\Gamma_S\) is connected for every stream \(S\) of \(\cF\).
\end{proposition}
\begin{proof}
    Let \(S\) be a stream of \(\cF\). Since \(\cA_\cF\) is the smallest stream, $\cR_{\cA_\cF}\subset \cR_S$ and every \(\cA_\cF\)-node is contained in a unique \(S\)-node.

    Suppose, by contradiction, that \(\Gamma_S\) is disconnected. 
    Then the \(S\)-nodes admit a partition into two nonempty families $\mathcal N_1,\mathcal N_2$ such that there is no edge between a node in \(\mathcal N_1\) and a node in \(\mathcal N_2\). 
    Set
\[
N_i=\bigcup_{P\in\mathcal N_i}P,
\qquad i=1,2.
\]
By the definition of a node partition, \(N_1\) and \(N_2\) are closed
and disjoint.

For \(i=1,2\), let
\[
\mathcal M_i
=
\left\{
M:
\begin{array}{l}
M\text{ is an }\cA_\cF\text{-node and}\\
M\subset P\text{ for some }P\in\mathcal N_i
\end{array}
\right\}.
\]

    We first show that both families are nonempty. 
    Fix \(i\in\{1,2\}\) and choose \(P\in\mathcal N_i\). 
    By the preceding lemma, there exist an \(S\)-node \(Q\) and an \(A_F\)-node \(M\subset Q\) such that $P\Sto  Q$.
    The node \(Q\) must belong to \(\mathcal N_i\). 
    Indeed, if it belonged to the other family, then, unless \(P=Q\), the relation \(P\Sto  Q\) would give an edge between the two classes, contrary to the choice of the partition. 
    Hence $M\in\mathcal M_i$
    and, therefore, \(\mathcal M_i\neq\varnothing\).

    Since every \(\cA_\cF\)-node is contained in a unique \(S\)-node, the families \(\mathcal M_1\) and \(\mathcal M_2\) are disjoint and together contain all \(\cA_\cF\)-nodes. 
    Let
\[
M_i=\bigcup_{M\in\mathcal M_i}M,
\qquad i=1,2.
\]
    Then $M_i=N_i\cap \cR_{\cA_\cF}$.
    Indeed, an \(\cA_\cF\)-recurrent point belongs to \(M_i\) precisely when the \(S\)-node containing it belongs to \(\mathcal N_i\).

    By Proposition~\ref{prop: streams}(2), \(\cR_{\cA_\cF}\) is closed.
    Hence, each \(M_i\) is closed, being the intersection of the closed sets \(N_i\) with
    \(\cR_{\cA_\cF}\). 
    Thus \(\mathcal M_1,\mathcal M_2\) form a node partition of \(\Gamma_{\cA_\cF}\).

    We finally show that there is no edge of \(\Gamma_{\cA_\cF}\) between these two families. 
    Suppose that there were \(\cA_\cF\)-nodes $M\in\mathcal M_1$ and $L\in\mathcal M_2$ and points \(u\in M\), \(v\in L\) such that $u\Ato v$.
    Since \(A_F\subset S\), it follows that $u\Sto v$.

    Let \(P\in\mathcal N_1\) and \(Q\in\mathcal N_2\) be the \(S\)-nodes containing \(M\) and \(L\), respectively. Then the relation \(u\Sto  v\) gives an edge $P\to Q$ in \(\Gamma_S\), contradicting the choice of \(\mathcal N_1,\mathcal N_2\).
    Therefore \(\Gamma_{A_F}\) is disconnected, contrary to the hypothesis.
    Hence, \(\Gamma_S\) is connected.
\end{proof}
The corresponding orbitwise statements also hold. More precisely, every orbitwise \(S\)-node contains an orbitwise \(\cA_\cF\)-node.
Consequently, if the orbitwise graph \(\Gamma^{\mathrm{ow}}_{\cA_\cF}\) is connected, then \(\Gamma^{\mathrm{ow}}_S\) is connected for every stream \(S\) of \(\cF\).
The first assertion follows from Lemma~\ref{lemma: Om} and Lemma~\ref{lemma: x owR => OM(x) in owR}, the second follows by the same partition argument as in the preceding theorem.
\section{The chain streams}
\label{sec: chains}
The concept of chain was introduced in the theory of (continuous-time) dynamical systems by Charles Conley in 1972~\cite{Con72b} as a way to simplify the construction of generalized recurrence introduced by Joseph Auslander a decade earlier~\cite{Aus63}.
This concept was then extended to discrete-time dynamical systems by Bowen~\cite{Bow75}.
In~\cite{DLY25}, we used this concept to define chain streams for semiflows with compact dynamics on metric spaces.
Here, we use the extension of the $\eps$-chains concept to IFSs by Barnsley and Vince~\cite{BV11} to define analogous chain streams on IFSs.
We will show in this section that the infinitesimal chain stream $\cC_\cF$, that we refer to simply as the chain stream, of an IFS $\cF$ with compact dynamics is strongly related to the dynamics of the IFS: 1.  $\cC_\cF$ is topological; 2. for every $x\in X$, each minimal compact forward-invariant subset of $\Om(x)$ is chain-equivalent and lies in some orbitwise chain node; 3. the chain-recurrent set and each of its nodes are compact and contained in $\Gl$; 4. the orbitwise chain-recurrent set of $\cC_\cF$ and each of its nodes are forward-invariant, compact and contained in $\Gl$; 5. for each orbitwise chain node $N$, the set $\Om(N)$ is nonempty, compact, H-invariant, and contained in N; 6. the plain and orbitwise graphs of $\cC_\cF$ coincide with those of the chain stream of the restriction of $\cF$ to $\Gl$.

\medskip\noindent
{\BF Chain streams.}
Let $\varepsilon>0$.  
Following Barnsley and Vince~\cite{BV11}, we say that an {\BF$\varepsilon$-chain} from $x$ to $y$
is a finite sequence of points of $X$
\[
x_0=x,x_1,\dots,x_k=y
\]
together with a sequence of indices $i_0,\dots,i_{k-1}$ such that
\[
d(f_{i_n}(x_n),x_{n+1})<\varepsilon
\qquad (n=0,\dots,k-1).
\]
We set
$$
\text{{\BF$\cC_{\cF,d,\eta}$}}=\{(x,y) : \text{ for every $\eps>\eta$, there is an $\eps$-chain from $x$ to $y$}\}.
$$

Note that $\cC_{\cF,d,\eta}\subset\cC_{\cF,d,\eta'}$ for $\eta\leq\eta'$.
 
\begin{proposition}
\label{prop: C is a stream}
    The relation $\cC_{\cF,d,\eta}$ is a stream for every $\eta\geq0$.
\end{proposition}
\begin{proof}
We must prove that $\cC_{\cF,d,\eta}$ is transitive, closed, and contains the orbit relation $\cO_\cF$.

\smallskip\noindent
{\BF 1. $\cO_\cF\subset \cC_{\cF,d,\eta}$.} 
Let $(x,y) \in \cO_\cF$. 
Then $y=\P(x)$ for some multiindex $I$.
The corresponding exact orbit segment is an $\varepsilon$-chain from $x$ to $y$ for every $\varepsilon>0$.
Therefore $\cO_\cF \subset \cC_{\cF,d,\eta}$.

\smallskip\noindent
{\BF 2. $\cC_{\cF,d,\eta}$ is transitive.}
Suppose that $(x,y)\in \cC_{\cF,d,\eta}$ and $(y,z)\in \cC_{\cF,d,\eta}$.
Fix an $\varepsilon>\eta$. 
Then there exists an $\varepsilon$-chain from $x$ to $y$ and an $\varepsilon$-chain from $y$ to $z$.
Concatenating the two chains gives an $\varepsilon$-chain from $x$ to $z$. 
Since $\varepsilon>\eta$ was arbitrary, it follows that $(x,z)\in \cC_{\cF,d,\eta}$, namely $\cC_{\cF,d,\eta}$ is transitive.

\smallskip\noindent
{\BF 3. $\cC_{\cF,d,\eta}$ is closed.}
Let $(x_n,y_n)\in \cC_{\cF,d,\eta}$ be such that $(x_n,y_n)\to (x,y)$.
We show that $(x,y)\in \cC_{\cF,d,\eta}$.
If $x=y$, then $(x,y)\in \cO_\cF\subset \cC_{\cF,d,\eta}$, so there is
nothing to prove. 
Assume therefore that $x\neq y$ and fix an
$\varepsilon>\eta$. 
Choose a $\rho>0$ such that $\eta+3\rho<\varepsilon$.
Since $\cF$ is finite and every $f_i$ is continuous at any $x$, for each $x$ there exists $\delta>0$ such that $d\bigl(f_i(x),f_i(u)\bigr)<\rho$ for every $u\in X$ with $d(x,u)<\delta$ and $i=1,\ldots,m$.

Choose $n$ sufficiently large that $d(x_n,x)<\delta$, $d(y_n,y)<\rho$ and $x_n\neq y_n$.
Since $(x_n,y_n)\in \cC_{\cF,d,\eta}$ and $\eta+\rho>\eta$, there exists an $(\eta+\rho)$-chain $x_n=u_0,u_1,\ldots,u_k=y_n$ from $x_n$ to $y_n$, with associated indices $i_0,\ldots,i_{k-1}$. 
Since $x_n\neq y_n$, we have $k\geq 1$.

Suppose first that $k=1$. 
Then $d\bigl(f_{i_0}(x_n),y_n\bigr)<\eta+\rho$, so that
$$
d\bigl(f_{i_0}(x),y\bigr)\leq d\bigl(f_{i_0}(x),f_{i_0}(x_n)\bigr) + d\bigl(f_{i_0}(x_n),y_n\bigr) +d(y_n,y) <\rho+(\eta+\rho)+\rho<\eps.
$$
Thus, $x,y$ is an $\varepsilon$-chain from $x$ to $y$.

Suppose now that $k\geq 2$ and consider the sequence $x,u_1,\ldots,u_{k-1},y$.
For its first step,
$$
d\bigl(f_{i_0}(x),u_1\bigr)\leq d\bigl(f_{i_0}(x),f_{i_0}(x_n)\bigr)+d\bigl(f_{i_0}(x_n),u_1\bigr)<\rho+(\eta+\rho)<\eps.
$$
Every intermediate step is unchanged from the original chain, and
therefore has error smaller than $\eta+\rho<\varepsilon$.
For the last step,
$$
d\bigl(f_{i_{k-1}}(u_{k-1}),y\bigr)\leq d\bigl(f_{i_{k-1}}(u_{k-1}),y_n\bigr)+d(y_n,y)<(\eta+\rho)+\rho=\eta+2\rho<\varepsilon.
$$
Hence, this sequence is an $\varepsilon$-chain from $x$ to $y$.
Since $\varepsilon>\eta$ was arbitrary, we conclude that $(x,y)\in \cC_{\cF,d,\eta}$.
Therefore, $\cC_{\cF,d,\eta}$ is closed.
\end{proof}

\medskip\noindent
{\bf Infinitesimal chains.} We often denote simply by {\BF$\cC_{\cF,d}$} the stream $\cC_{\cF,d,0}$.
We call this the {\bf stream of infinitesimal chains} or, more often, simply the {\bf chain stream}.
A point $y$ is $\cC_{\cF,d,\eta}$-downstream from $x$ when it is possible to get from $x$ to $y$ under $\Phi_\cF$ by steering at every step by no more than some $\eps$ arbitrarily close from above to $\eta$.
A point $y$ is $\cC_{\cF,d}$-downstream from $x$ when it is possible to get from $x$ to $y$ under $\Phi_\cF$ by steering at every step by some arbitrarily small $\eps>0$.

\medskip\noindent
{\bf Chain-recurrence.}
Let $\eta\geq0$.
Since $\cC_{\cF,d,\eta}$ is a stream, we have a concept of chain-recurrence.
Two points $x,y\in X$ are {\BF$\eta$-chain-equivalent}, and we write $$x\Cdeqeta y,$$ if, for every $\eps>\eta$, there is an $\eps$-chain from $x$ to $y$ and an $\eps$-chain from $y$ to $x$.
The set of {\BF$\eta$-chain-recurrent} points is then defined as
$$
\text{{\BF$\cR_{\cC_{\cF,d,\eta}}$}} = \{x\in X\;:\;x\Cdeqeta f(x)\text{ for some }f\in\cF\}
$$
and the set of all {\BF orbitwise $\eta$-chain-recurrent} points as
$$
\text{{\BF$\cR^{ow}_{\cC_{\cF,d,\eta}}$}} = \{x\in X\;:\;x\Cdeqeta z\text{ for all }z\in\cO_\cF(x)\}.
$$
When $\eta=0$, we just use the terms {\bf chain-equivalent} and {\bf chain-recurrent}/{\bf orbitwise chain-recurrent}.

\medskip\noindent
{\bf Chain-recurrence and limit sets.}
When $m=1$ and $\cF$ has compact dynamics, the limit set of each point is a (non-empty) chain-equivalent set, namely it is entirely contained in a single chain node.
For $m>1$, $\Om(x)$ is not necessarily a subset of $\cR_{\cC_{\cF,d}}$ and $\Om(x)\cap\cR_{\cC_{\cF,d}}$ is not necessarily chain-equivalent.
Nevertheless, as next proposition shows, 
it always contains at least one orbitwise chain-recurrent point. 

\begin{proposition}
    Assume that $\cF$ has compact dynamics.
    Then, for each $x\in X$, $\Om(x)$ has at least a non-empty minimal compact forward-invariant subset and each such minimal set is contained in some orbitwise chain node.
\end{proposition}
\begin{proof}
    The existence of a non-empty compact forward-invariant $M\subset\Om(x)$ is shown in the proof of Lemma~\ref{lemma: Om}.
    Because of the forward invariance, for any $y\in M$ we have that $\Om(y)\subset M$.
    Since $\Om(y)$ is forward-invariant and $M$ is minimal, $\Om(y)=M$.
    Hence, for any distinct points $y,z\in M$, we have that $y\in\Om(z)$ and $z\in\Om(y)$.
    Hence, $y\Cto z$ and $z\Cto y$, namely $y\Ceq z.$

    Since $M$ is forward-invariant, $\cO_\cF(y)\subset M$ for every $y\in M$.
    Thus, every point of the orbit of $y$ is chain-equivalent to $y$, so $y$ is orbitwise chain-recurrent. Consequently, $M$ is contained in an orbitwise chain node.
\end{proof}
\medskip\noindent
{\bf Example: the $\eta$-chain-recurrent set of the Sierpi\'nski IFS.} 
In Section~\ref{sec: global attractor} we defined the Sierpi\'nski IFS $\cS$ on $\bR^2$ and discussed its trapping regions and global attractor $\cG_\cS$, the Sierpi\'nski gasket. 
Moreover, by Proposition~\ref{prop: R_F=K}, $\cR_\cS=\cR^{ow}_\cS=\cG_\cS$ and we will show in
Proposition~\ref{prop: R_C=K} also that $\cR_{\cC_{\cS,d}}=\cR^{ow}_{\cC_{\cS,d}}=\cG_\cS$.
Hence, 
each point of $\cR_{\cC_{\cS,d}}$ is orbitwise chain-equivalent to every other point in that set, namely $\cG_\cS$ is the unique orbitwise node of $\cC_{S,d}$. 

In the proposition below we study the case $\eta>0$.

\begin{proposition}
    \label{prop: R_eta subset N_2eta(R)}
    For every $\eta>0$, 
    $$
    \cR_{\cC_{\cS,d,\eta}}=\cR^{ow}_{\cC_{\cS,d,\eta}}=\overline{N_{2\eta}(\cG_\cS)}
    $$ 
    and the whole set is a single orbitwise $\cC_{\cS,d,\eta}$-node.
\end{proposition}

\begin{proof}
Since every generator $s_i$ has Lipschitz constant $1/2$ and maps $\cG_\cS$ into itself, 
$$
d(s_i(x),\cG_\cS)\leq \frac12 d(x,\cG_\cS).
$$
Consequently, along any $\varepsilon$-chain
$x_0,\ldots,x_n$,
\[
d(x_{j+1},\cG_\cS)<\varepsilon+\frac12d(x_j,\cG_\cS).
\]

Let $x$ satisfy $d(x,\cG_\cS)>2\eta$ and choose an $\eps\in(\eta,\frac12d(x,\cG_\cS))$.
If an $\varepsilon$-chain starts at $x$, then
\[
d(x_1,\cG_\cS)<\varepsilon+\frac12d(x,\cG_\cS)<d(x,\cG_\cS),
\]
and induction gives $d(x_j,\cG_\cS)<d(x,\cG_\cS)$ for every $j\geq1$.
Thus, there is no positive-length $\varepsilon$-chain from $x$ to
itself.

    Choose a generator $s_i$ such that $s_i(x)\neq x$; such a generator exists because the generators of $\cS$ have no common fixed point.
    If $x$ were $\cC_{\cS,d,\eta}$-recurrent, there would be an $\varepsilon$-chain from $s_i(x)$ to $x$. 
    Prepending the exact step $x\mapsto s_i(x)$ would give a positive-length $\varepsilon$-chain from $x$ to itself, a contradiction. 
    Hence $\cR_{\cC_{\cS,d,\eta}}\subset \overline{N_{2\eta}(\cG_\cS)}$.

    Now, let $p\in\mathbb{R}^2$ and $y\in\overline{N_{2\eta}(\cG_\cS)}$ and fix an $\varepsilon>\eta$. 
    Choose $q\in \cG_\cS$ such that
    \[
     d(y,q)=d(y,\cG_\cS)\leq2\eta<2\varepsilon.
    \]
    Since $H_\cS^n(\cG_\cS)=\cG_\cS$, for every $n\geq1$ there exist a $I_n\in\cI^m$ with $|I_n|=n$ and a point $q_n\in \cG_\cS$ such that $q=\Phi_S^{I_n}(q_n)$.
    Hence
\[
d\bigl(\Phi_\cS^{I_n}(p),q\bigr)
=
d\bigl(\Phi_\cS^{I_n}(p),\Phi_S^{I_n}(q_n)\bigr)
\leq 2^{-n}d(p,q_n)\to0,
\]
    because $\cG_\cS$ is compact.
    Therefore, for some sufficiently large $n$,
\[
d\bigl(y,\Phi_\cS^{I_n}(p)\bigr)
<
2\varepsilon(1-2^{-n}).
\]
    Set
\[
e=
\frac{y-\Phi_\cS^{I_n}(p)}
     {2(1-2^{-n})},
\]
    so that $\lVert e\rVert<\varepsilon$, and define
\[
x_0=p,
\qquad
x_{j+1}=s_{i_{j+1}}(x_j)+e,
\qquad j=0,\ldots,n-1.
\]
    Then
\[
d\bigl(s_{i_{j+1}}(x_j),x_{j+1}\bigr)=\lVert e\rVert<\varepsilon,
\]
    and, since every $s_i$ has linear part $\frac12\operatorname{id}$,
\[
x_n
=
\Phi_\cS^{I_n}(p)
+
\left(\sum_{k=0}^{n-1}2^{-k}\right)e
=
\Phi_\cS^{I_n}(p)+2(1-2^{-n})e
=
y.
\]
    Thus, there is an $\varepsilon$-chain from $p$ to $y$. 
    Since $\varepsilon>\eta$ was arbitrary, $p\succcurlyeq_{\cC_{\cS,d,\eta}}y$.

    Now, let $x\in\overline{N_{2\eta}(\cG_\cS)}$ and let $z\in \cO_\cS(x)$.
    Since the chain relation contains the orbit relation, $x\succcurlyeq_{\cC_{\cS,d,\eta}}z$, while the preceding claim, applied with $p=z$ and $y=x$, gives $z\succcurlyeq_{\cC_{\cS,d,\eta}}x$.
    Hence $x$ is orbitwise $\cC_{\cS,d,\eta}$-recurrent, namely $\overline{N_{2\eta}(\cG_\cS)}\subset \cR^{ow}_{\cC_{\cS,d,\eta}}$.

    Finally, the same argument shows that any two points of $\overline{N_{2\eta}(\cG_\cS)}$ are mutually $\cC_{\cS,d,\eta}$-equivalent, namely this set is a single node.
\end{proof}

Notice that every point in $T$ lies at distance not larger than $\sqrt{3}/12$ from the gasket (the maximum is achieved at the barycenter of the triangle).
Hence, for $\eta\geq\sqrt{3}/24\simeq0.072$ we have that $\cR_{\cC_{\cS,d,\eta}}=\cR^{ow}_{\cC_{\cS,d,\eta}}\supset T$.

\medskip\noindent
{\bf Alternative definitions of chain-recurrence.}
For the chain stream, holds the following alternate definition of recurrence.
\begin{proposition}
    \label{prop: alternative weak node}
    For \(x\in X\), the following are equivalent:
\begin{enumerate}
    \item \(x\) is chain-recurrent, namely there is a generator $f_i$ such that $f_i(x)\Ceq x$;
    \item \(f_i(x)\Cto x\) for some generator \(f_i\);
    \item either \(f_i(x)=x\) for some generator \(f_i\), or
    \(x\) is chain-equivalent to a point distinct from itself;
    \item for every $\eps>0$ there is a positive-length $\eps$-chain from $x$ to $x$.
\end{enumerate}
\end{proposition}
%
\begin{proof}
    (1) $\iff$ (2).
    By definition, there is a $f_i$ such that $f_i(x)\Seq x$ , and so in particular $f_i(x)\Cto x$.

    (2) $\implies$ (3).
    Since the chain stream contains the orbit relation, we have that $x\Cto f_i(x)$ and so $x$ is $\cC_\cF$-equivalent to $f_i(x)$.
    Assume that $x$ is not fixed under any $f_j$.
    Then $f_i(x)\neq x$ and so $x$ is chain-equivalent to a point distinct from itself. 

    (3) $\implies$ (4).
    If $x$ is fixed under some generator of $\cF$, the claim is obvious.
    If there is a $y\neq x$ with $y\Cto x$ and $x\Cto y$, then for every $\eps>0$ there are $\eps$-chains from $y$ to $x$ and viceversa.
    By concatenating those chains, we get an $\eps$-look at $x$ for every $\eps>0$.

    (4) $\implies$ (2).
    Let $\varepsilon_n\downarrow 0$. 
    For every $n$, choose an $\varepsilon_n$-loop at $x$
\[
x^{(n)}_0=x,x^{(n)}_1,\ldots,x^{(n)}_{k_n}=x.
\]

    Let $i^{(n)}_0,\ldots,i^{(n)}_{k_n-1}$ be the generators associated with this loop. 
    Assume first that $k_n=1$ for infinitely many $n$.
    Then $d\bigl(f_i(x),x\bigr)<\varepsilon_n$.
    Letting $n\to\infty$, we obtain $f_i(x)=x$, and therefore $f_i(x)\Cto x$.
    
    Assume now that $k_n>1$ for $n$ large enough. 
    Then, since $\cF$ is finite, after passing to a subsequence we may assume that $i^{(n)}_0=i$, $i^{(n)}_1=j$ are independent of $n$. 
    From the first jump of the chain, $d\bigl(f_i(x),x^{(n)}_1\bigr)<\varepsilon_n$, and hence $x^{(n)}_1\to f_i(x)$.

    We claim that $f_i(x)\Cto x$.
    Indeed, fix an $\varepsilon>0$. 
    By continuity of $f_j$ at $f_i(x)$, for all sufficiently large $n$ we have
\[
d\bigl(f_j(f_i(x)),f_j(x^{(n)}_1)\bigr)<\frac{\varepsilon}{2}.
\]
    Choose such an $n$ also satisfying $\varepsilon_n<\frac{\varepsilon}{2}$.
    Since the second jump of the original loop satisfies $d\bigl(f_j(x^{(n)}_1),x^{(n)}_2\bigr)<\varepsilon_n$, we obtain
$$
d\bigl(f_j(f_i(x)),x^{(n)}_2\bigr)
\leq
d\bigl(f_j(f_i(x)),f_j(x^{(n)}_1)\bigr)
+d\bigl(f_j(x^{(n)}_1),x^{(n)}_2\bigr) 
<
\frac{\varepsilon}{2}+\varepsilon_n
<\varepsilon.
$$

    All the remaining jumps of $f_i(x),x^{(n)}_2,\ldots,x^{(n)}_{k_n}=x$ have error smaller than $\varepsilon_n<\varepsilon$. 
    This sequence is therefore an $\varepsilon$-chain from $f_i(x)$ to $x$.
    Since $\varepsilon>0$ was arbitrary, this means that $f_i(x)\Cto x$.
\end{proof}





\medskip\noindent
{\BF The chain stream is an $\Omega$stream.} 
With the help of some technical lemma, we prove below that the chain stream is an $\Omega$stream.

\begin{lemma}
\label{lem:local-shadowing-fixed-x}
Let $p\geq0$.
Then, for every $x\in X$ and $\varepsilon>0$, there exists $\delta>0$ such that, for every $\delta$-chain $x_0=x,x_1,\dots,x_n$ with indices $i_0,\dots,i_{n-1}\in\{1,\dots,m\}$, if we set $I_0=0$ and $I_j=i_0\cdots i_{j-1}$ for $j=1,\dots,n$, we have that
$$
d\!\left(\Phi_\cF^{I_j}(x),x_j\right)<\varepsilon\quad\text{for every $j=0,\dots,\min\{p,n\}$}.
$$
If $n>p$, moreover,
the tail $\Phi_\cF^{I_p}(x),\,x_{p+1},\,x_{p+2},\,\dots,\,x_n$ is an $\varepsilon$-chain from $\Phi_\cF^{I_p}(x)$ to $x_n$.
\end{lemma}


\begin{proof}
Set
$\eta_{p+1}:=\varepsilon$.
We recursively choose positive numbers $\eta_p,\eta_{p-1},\ldots,\eta_0$ as follows. 
For each $j=p,p-1,\ldots,0$, choose
$\eta_j\in(0,\varepsilon)$ such that
\[
d(u,v)<\eta_j
\quad\Longrightarrow\quad
d\bigl(f_i(u),f_i(v)\bigr)
<
\frac{\eta_{j+1}}{2}
\]
for every
\[
u\in H_\cF^j(\{x\}),\qquad v\in X,
\qquad i\in\{1,\ldots,m\}.
\]
Such a choice is possible because $H_\cF^j(\{x\})$ is finite, $\cF$ has finitely many generators and every generator is continuous.

Now, set $\delta=\min\{\eta_1,\ldots,\eta_{p+1}\}/2$.
Since $\eta_{p+1}=\varepsilon$, in particular  $\delta\leq\varepsilon/2$.
Let $x_0=x,x_1,\ldots,x_n$ be a $\delta$-chain with associated indices $i_0,\ldots,i_{n-1}$. 
For $j=0$, since $I_0=0$ and $x_0=x$, we have that
\[
d\bigl(\Phi_\cF^{I_0}(x),x_0\bigr)
=
d(x,x)=0<\eta_0.
\]

    Assume now that the assertion holds for some $j<p$. 
    Since
\[
\Phi_\cF^{I_j}(x)\in H_\cF^j(\{x\}),
\]
the choice of $\eta_j$ gives
\[
d\left(
f_{i_j}\bigl(\Phi_\cF^{I_j}(x)\bigr),
f_{i_j}(x_j)
\right)
<
\frac{\eta_{j+1}}{2}.
\]
Moreover, since the given sequence is a $\delta$-chain,
\[
d\bigl(f_{i_j}(x_j),x_{j+1}\bigr)<\delta
\leq\frac{\eta_{j+1}}{2}.
\]
    Using $I_{j+1}=I_j i_j$ and $\Phi_\cF^{I_{j+1}}=f_{i_j}\circ\Phi_\cF^{I_j}$ we obtain
$$
d\bigl(\Phi_\cF^{I_{j+1}}(x),x_{j+1}\bigr)
\leq
d\left(
f_{i_j}\bigl(\Phi_\cF^{I_j}(x)\bigr),
f_{i_j}(x_j)
\right) +d\bigl(f_{i_j}(x_j),x_{j+1}\bigr)<
\frac{\eta_{j+1}}{2}
+
\frac{\eta_{j+1}}{2}
=
\eta_{j+1}.
$$
    Since each $\eta_j<\varepsilon$ for $j=0,\ldots,p$, it follows that
\[
d\bigl(\Phi_\cF^{I_j}(x),x_j\bigr)<\varepsilon
\]
for every $j=0,\ldots,p$.

It remains to prove the assertion about the tail. 
Since $d\bigl(\Phi_\cF^{I_p}(x),x_p\bigr)<\eta_p$ and $\Phi_\cF^{I_p}(x)\in H_\cF^p(\{x\})$, the choice of $\eta_p$ gives
\[
d\left(
f_{i_p}\bigl(\Phi_\cF^{I_p}(x)\bigr),
f_{i_p}(x_p)
\right)
<
\frac{\eta_{p+1}}{2}
=
\frac{\varepsilon}{2}.
\]
Also,
\[
d\bigl(f_{i_p}(x_p),x_{p+1}\bigr)
<
\delta
\leq
\frac{\varepsilon}{2}.
\]
Consequently,
$$d\left(
f_{i_p}\bigl(\Phi_\cF^{I_p}(x)\bigr),
x_{p+1}
\right)
\leq
d\left(
f_{i_p}\bigl(\Phi_\cF^{I_p}(x)\bigr),
f_{i_p}(x_p)
\right) +d\bigl(f_{i_p}(x_p),x_{p+1}\bigr)<\varepsilon.
$$
All subsequent jumps $x_{p+1}\to x_{p+2}\to\cdots\to x_n$ are unchanged from the original $\delta$-chain and therefore have error smaller than $\delta\leq\varepsilon/2<\varepsilon$.
Hence $\Phi_\cF^{I_p}(x),x_{p+1},\ldots,x_n$ is an $\varepsilon$-chain from $\Phi_\cF^{I_p}(x)$ to $x_n$.
\end{proof}

\begin{lemma}
\label{lem:lengths-go-to-infinity}
Assume that $x\Cto y$ with $y\not\in\cO_\cF(x)$.
Then, for every integer $p\ge 0$, there exists $\delta_p>0$ such that every
$\delta_p$-chain from $x$ to $y$ has length strictly greater than $p$.
\end{lemma}

\begin{proof}
Fix $p\ge 0$. Since $\bigcup_{j=0}^p H^j_\cF(\{x\})$
is a finite set and $y\notin \cO_\cF(x)$, the number
\[
\rho_p=\min\left\{ d\!\left(y,\P(x)\right) : |I|\le p \right\}
\]
is well-defined and strictly positive.

Apply Lemma \ref{lem:local-shadowing-fixed-x} with this $p$ and with $\varepsilon=\rho_p/2$, let $\delta_p>0$ be the corresponding number and let $x_0=x,x_1,\dots,x_n=y$ be a $\delta_p$-chain from $x$ to $y$ with $n\le p$. 
We may apply the same inductive argument
used in the proof of Lemma \ref{lem:local-shadowing-fixed-x} up to time $n$ and obtain
\[
d\!\left(\Phi_\cF^{I_n}(x),y\right)<\rho_p/2,
\]
where $I_n=i_0\cdots i_{n-1}$ is the multiindex associated with the chain. Since $n\le p$, this contradicts
the definition of $\rho_p$.
Therefore no $\delta_p$-chain from $x$ to $y$ can have length $\le p$.
\end{proof}

\begin{proposition}
\label{prop:chains-omega-stream}
Assume that $\cF$ has compact dynamics. 
Then $\cC_{\cF,d}$ is an $\Omega$stream;
equivalently, if $x\Cto y$ then either $y\in\cO_\cF(x)$ or there is a $z\in\Om(x)$ such that $z\Cto y$.
\end{proposition}

\begin{proof}
Let $y\in \Down_{\cC_{\cF,d}}(x)\setminus \cO_\cF(x)$.
We must prove that $y\in \Down_{\cC_{\cF,d}}(\Om(x))$.

Fix an integer $p\ge 0$ and an $\varepsilon>0$.
Let $\delta_p$ be given by Lemma \ref{lem:lengths-go-to-infinity} and 
$\delta_{p,\varepsilon}$ by Lemma \ref{lem:local-shadowing-fixed-x} and set
$\delta=\min\{\delta_p,\delta_{p,\varepsilon}\}$.
Since $y\in \Down_{\cC_{\cF,d}}(x)$, there exists a $\delta$-chain $x_0=x,x_1,\dots,x_n=y$ from $x$ to $y$. 
By Lemma \ref{lem:lengths-go-to-infinity}, its length satisfies $n>p$.
Let $I_p=i_0\cdots i_{p-1}$ be the multiindex determined by the first $p$ jumps of the chain.
By Lemma \ref{lem:local-shadowing-fixed-x}, the tail $\Phi_\cF^{I_p}(x),x_{p+1},\dots,x_n=y$ is an $\varepsilon$-chain from $\Phi_\cF^{I_p}(x)$ to $y$.

Notice that $I_p$ depends on $\eps$.
Now, choose a sequence $\eps_k\to0$.
For each $k$, there is a word $I^{(k)}_p$ such that there is an $\varepsilon_k$-chain from $\Phi_\cF^{I^{(k)}_p}(x)$ to $y$.
Since there are only finitely many words in $m$ symbols with a given length, we can extract a subsequence $k_i$ so that, for every $i=1,2,\dots$, $I^{(k_i)}_p$ is equal to a given word $J_p$ and there is an $\varepsilon_{k_i}$-chain from $\Phi_\cF^{J_p}(x)$ to $y$.
Hence, 
$\Phi_\cF^{J_p}(x)\Cto y.$

Consider now the sequence $\Phi_\cF^{J_p}(x)$ and choose a compact neighborhood $K$ of $\Gl$. 
Since $\Gl$ attracts $\{x\}$, all sufficiently large $\Phi_\cF^{J_p}(x)$ belong to $K$.
Since $|J_p|=p\to\infty$, the sequence has a subsequence converging to some point $z\in \Om(x)$.
Since $\cC_\cF$ is closed, passing to the limit yields $z\Cto y$.
\end{proof}

\black

\medskip\noindent
{\bf Trapping regions of a stream.}
We say that a closed set $Q$ is a {\BF trapping region} of a stream $S$ if $\Down_S(Q)\subset Q$.

\begin{lemma}[Robust trapping of small chains]
    \label{lemma: robust trapping region}
    Assume that $\cF$ has compact dynamics. 
    Then, for every $\varepsilon>0$, there exist $\eta\in(0,\varepsilon)$ and $\delta>0$ such that
    every $\delta$-chain starting at a point of $\overline{N_\eta(\Gl)}$ is entirely contained in $N_\varepsilon(\Gl)$.
\end{lemma}

\begin{proof}
Choose $\eta_0\in(0,\varepsilon)$ sufficiently small that $\overline{N_{\eta_0}(\Gl)}$ is compact. 
Since $\Gl$ attracts every compact set, there exists $N_0\geq1$ such that 
\beqn
\label{eq: 5.9-1}
\P(\overline{N_{\eta_0}(\Gl)})\subset N_{\eta_0/2}(\Gl)
\qquad\text{ whenever } |I|\geq N_0.
\eeqn
Since $\Gl$ is $H$-invariant, for every $I\in\cI^m$ with $|I|<N_0$ we have that $\P(\Gl)\subset \Gl$.
Since there are only finitely many such multiindices, continuity and compactness of $\Gl$ imply that there exists $\eta_1\in(0,\eta_0)$ such that $\P\bigl(\overline{N_{\eta_1}(\Gl)}\bigr)\subset N_{\eta_0}(\Gl)$ whenever $|I|<N_0$.
For $|I|\geq N_0$, the same inclusion follows from~\eqref{eq: 5.9-1} since $\eta_1<\eta_0$.
Consequently,
\beqn
\label{eq: 5.9-2}
\P\bigl(\overline{N_{\eta_1}(\Gl)}\bigr)
\subset N_{\eta_0}(\Gl)
\qquad\text{for every }I\in \cI^m.
\eeqn

Now, choose an $\eta\in(0,\eta_1)$ and notice that there exists an $N\geq1$ such that
\beqn
\label{eq: 5.9-2b}
\P\bigl(\overline{N_{\eta}(\Gl)}\bigr)\subset N_{\eta/2}(\Gl)
\qquad\text{whenever } |I|\geq N.
\eeqn
By~\eqref{eq: 5.9-2}, we also have that
\beqn
\label{eq: 5.9-3}
\P\bigl(\overline{N_{\eta}(\Gl)}\bigr)\subset N_{\eta_0}(\Gl)
\qquad\text{for every }I\in \cI^m.
\eeqn

    Now, choose a $\rho\in(0,\min\left\{\varepsilon-\eta_0,\eta/2\right\})$ and a $\delta>0$ so that every $\delta$-chain of length at most $N$, starting in $\overline{N_{\eta}(\Gl)}$, stays within $\rho$ from the corresponding exact orbit segment.    
    Set $\rho_N:=\rho$ and recursively, for $j=N-1,N-2,\ldots,0$, choose $\rho_j\in(0,\rho)$ such that
\[
d(u,v)<\rho_j
\quad\Longrightarrow\quad
d\bigl(f_i(u),f_i(v)\bigr)<\frac{\rho_{j+1}}2
\]
    for every
\[
u\in H_\cF^j(\overline{N_{\eta}(\Gl)}),\qquad v\in X,\qquad i=1,\ldots,m.
\]
    Such a choice is possible because every $H_\cF^j(\overline{N_{\eta}(\Gl)})$ is compact and the family of generators is finite. 
    Set $\delta=\min\{\rho_1,\ldots,\rho_N\}/2$ and let $x_0,x_1,\ldots,x_r$, $r\leq N$, be a $\delta$-chain with $x_0\in \overline{N_{\eta}(\Gl)}$, and let $I_0=0$, $I_j=i_0\cdots i_{j-1}$.
    We claim that
\beqn
\label{eq: 5.9-4}
d\bigl(\Phi_\cF^{I_j}(x_0),x_j\bigr)<\rho_j
\qquad (j=0,\ldots,r).
\eeqn
The assertion is immediate for $j=0$. If it holds for some $j<r$,
then
$$
d\bigl(\Phi_\cF^{I_{j+1}}(x_0),x_{j+1}\bigr)
\leq
d\bigl(
f_{i_j}(\Phi_\cF^{I_j}(x_0)),f_{i_j}(x_j)
\bigr)
+
d\bigl(f_{i_j}(x_j),x_{j+1}\bigr)
<
\frac{\rho_{j+1}}2+\delta
\leq \rho_{j+1}.
$$
    Thus~\eqref{eq: 5.9-4} follows by induction. 
    In particular,
\beqn
\label{eq: 5.9-5}
d\bigl(\Phi_\cF^{I_j}(x_0),x_j\bigr)<\rho
\qquad (j=0,\ldots,r).
\eeqn

    Now, let $x_0,x_1,\ldots,x_\ell$ be an arbitrary $\delta$-chain with $x_0\in \overline{N_{\eta}(\Gl)}$.
    Consider first its first block of at most $N$ steps. 
    By~\eqref{eq: 5.9-3} and~\eqref{eq: 5.9-5}, $d(x_j,\Gl)<\eta_0+\rho<\varepsilon$, so every point in this block belongs to $N_\varepsilon(\Gl)$.
    If the block has length exactly $N$ then, by~\eqref{eq: 5.9-2b} and \eqref{eq: 5.9-5}, $d(x_N,\Gl)<\frac{\eta}{2}+\rho<\eta$.
    Hence,
\[
x_N\in N_\eta(\Gl)\subset \overline{N_{\eta}(\Gl)}.
\]
    We may therefore apply the same argument to the next block of $N$ steps. 
    Repeating this procedure, and applying the same finite-time estimate to the final block if it has fewer than $N$ steps, shows that $x_j\in N_\varepsilon(\Gl)$ for $j=0,\ldots,\ell$.

    Thus, every $\delta$-chain starting in $\overline{N_\eta(\Gl)}$ is entirely contained in $N_\varepsilon(\Gl)$.
\end{proof}

\begin{proposition}
    \label{prop: GF-trapping-for-CF}
    The global attractor $\Gl$ is a trapping region for the chain stream $\cC_{\cF,d}$. 
    More precisely,
\[
\Down_{\cC_{\cF,d}}(\Gl)\subset \Gl .
\]
    In particular,
\[
\cR_{\cC_{\cF,d}}\subset \Gl
\]
    and so also
\[
\cR^{ow}_{\cC_{\cF,d}}\subset \Gl.
\]
\end{proposition}

\begin{proof}
    Let $x\in \Gl$ and $y\notin \Gl$ and set $\varepsilon=d(y,\Gl)/2>0$.
    By Lemma~\ref{lemma: robust trapping region}, there exist an $\eta\in(0,\varepsilon)$ and a $\delta>0$ such that every $\delta$-chain starting in $\overline{N_\eta(\Gl)}$ is entirely contained in $N_\varepsilon(\Gl)$.

    Since $x\in \Gl\subset\overline{N_\eta(\Gl)}$, every $\delta$-chain starting at $x$ is contained in $N_\varepsilon(\Gl)$. 
    On the other hand, $d(y,\Gl)=2\varepsilon$, so that $y\notin N_\varepsilon(\Gl)$.
    Hence, there is no $\delta$-chain from $x$ to $y$, and so $x\not\succcurlyeq_{\cC_{\cF,d}}y$.
    Since this holds for every $x\in \Gl$ and every $y\notin \Gl$, we conclude that
$
\operatorname{Down}_{\cC_{\cF,d}}(\Gl)\subset \Gl$.

    Now, let $x\in \cR_{\cC_{\cF,d}}$. 
    Since $\cC_\cF$ is an $\Omega$-stream by Proposition~\ref{prop:chains-omega-stream}, by Proposition~\ref{prop: O-stream}(1) there is a $z\in\Om(x)$ such that $z\Ceq x$.
    By Proposition~\ref{prop: omega(K) subset G}, $\Om(x)\subset \Gl$, so $z\in \Gl$. 
    Since $\Gl$ is a trapping region for $\cC_\cF$ and $z\Cto x$, 
    it follows that $x\in \Gl$. Therefore, $\cR_{\cC_{\cF,d}}\subset \Gl$.
\end{proof}

\medskip\noindent
{\bf The chain stream is topological.}
We show below that, when $\cF$ has compact dynamics, $\cC_{\cF,d}$ does not actually depend on $d$ but only on the topology induced by $d$ on $X$.
Hence, from now on, we will denote the chain stream of $\cF$ simply by {\BF$\cC_\cF$}.
The first proposition of this kind in literature is due, to our knowledge, to Mike Hurley, that proved in~\cite{Hur91} that, when $X$ is compact, the chain-recurrent set of any semiflow is topological.
We recently proved that the same is true for the chain stream of any semiflow with compact dynamics on a locally compact phase space~\cite{DLY26a} and of any semiflow with strong compact dynamics on a separable phase space~\cite{DLY26b}.
\begin{proposition}
    \label{prop: C_F is topological}
    Assume that $\cF$ has compact dynamics.
    Then the chain relation $\cC_{\cF,d}$ is topological, i.e. it does not depend on the particular compatible metric on $X$.
\end{proposition}

\begin{proof}
    Let $d$ and $d'$ be two metrics on $X$ compatible with the topology of $X$. We must prove that
\[
\cC_{\cF,d}=\cC_{\cF,d'}.
\]

    Let $x,y\in X$, and assume first that $(x,y)\in \cC_{\cF,d}$. 
    We show that $(x,y)\in \cC_{\cF,d'}$.

    We sort the proof in four cases: 
    (1) $x,y\not\in\Gl$; 
    (2) $x\in\Gl,y\not\in\Gl$; 
    (3) $x,y\in\Gl$; 
    (4) $x\not\in\Gl, y\in\Gl$.

    (1) Since the chain stream is an $\Omega$stream and $\Om(x)\subset\Gl$ for every $x\in X$, $(x,y)\in\cC_{\cF,d}$ with $x,y\not\in\Gl$ implies that $y\in\cO_\cF(x)$, so that we have also that  $(x,y)\in\cC_{\cF,d'}$.

    (2) This case is not possible because, by Proposition~\ref{prop: GF-trapping-for-CF}, $\Gl$ is a trapping region for the chain stream.

    (3)
    Fix an $\varepsilon>0$ and choose $a>0$ such that $\overline{N_a(\Gl)}$ is compact, and set
\[
B=\overline{N_a(\Gl)}\bigcup H_\cF(\overline{N_a(\Gl)}).
\]
    Then $B$ is compact.
    Since $d$ and $d'$ induce the same topology, the identity $\operatorname{id}\colon(B,d)\longrightarrow(B,d')$ is uniformly continuous. 
    Hence, there exists $\gamma>0$ such that
\[
u,v\in B,\quad d(u,v)<\gamma
\quad\Longrightarrow\quad
d'(u,v)<\varepsilon.
\]

    By Lemma~\ref{lemma: robust trapping region}, there exists $\delta_0>0$ such that every $d$-$\delta_0$-chain starting in $\Gl$ is contained in $N_a(\Gl)\subset B$. Set $\delta:=\min\{\delta_0,\gamma\}$
    and choose a $d$-$\delta$-chain $x_0=x,x_1,\ldots,x_n=y$.
    Every $x_j$ belongs to $B$, and therefore both $f_{i_j}(x_j)$ and $x_{j+1}$ belong to $B$. 
    Since
\[
d\bigl(f_{i_j}(x_j),x_{j+1}\bigr)<\delta\leq\gamma,
\]
    we have
\[
d'\bigl(f_{i_j}(x_j),x_{j+1}\bigr)<\varepsilon.
\]
    Thus, the same sequence is a $d'$-$\varepsilon$-chain from $x$ to $y$.

    (4) In this case, either $y\in\cO_\cF(x)$, and so $(x,y)\in\cC_{\cF,d'}$, or, since the chain stream is an $\Omega$stream, there is a $z\in\Om(x)\subset\Gl$ with $(z,y)\in\cC_{\cF,d}$.
    By (3), $(z,y)\in\cC_{\cF,d'}$.
    Since $\cC_{\cF,d'}$ is closed and contains $\cO_\cF$, we also have that $x\succcurlyeq_{C_{F,d'}}z$ and so, by transitivity, $x\succcurlyeq_{C_{F,d'}}y$.
    


    Hence, $\cC_{\cF,d}\subset \cC_{\cF,d'}$;
    by symmetry, exchanging the roles of $d$ and $d'$, we ultimately get that
$\cC_{\cF,d}=\cC_{\cF,d'}$.
\end{proof}

\medskip\noindent
{\bf Chain nodes and the global attractor.}
Proposition~\ref{prop: GF-trapping-for-CF} shows that the chain-recurrent set lies in $\Gl$. 
Together with Proposition~\ref{prop: streams}, this implies immediately the following result.
\begin{corollary}
    \label{cor: chain nodes are compact}
    Assume that $\cF$ has compact dynamics.
    Then $\cR_{\cC_{\cF}}$, $\cR^{ow}_{\cC_{\cF}}$ and all of their nodes are compact.
\end{corollary}
The relation between $\Gl$ and the chain-recurrent set is even deeper: in the result below we show that, if two points are chain-equivalent in $X$ under $\cF$, then they are chain-equivalent even in $\Gl$, so that
the qualitative dynamics of $\cF$ can be entirely read from the restriction of $\cF$ to $\Gl$, which is a compact set. 
In other words, an IFS with compact dynamics has the same qualitative dynamics of an IFS on a compact set.
\begin{theorem}
\label{prop: C = C glob}
    Assume that $\cF$ has compact dynamics. 
    Then
\[
\cC_\cF|_{\Gl\times \Gl}=\cC_{\cF|_{\Gl}}.
\]
    In particular, both the plain and orbitwise chain-recurrent sets, nodes, and graphs of \(\cF\) coincide with those of \(\cF|_{\Gl}\).
\end{theorem}

\begin{proof}
Since every $\varepsilon$-chain for $\cF|_{\Gl}$ is also an $\varepsilon$-chain
for $\cF$, we have
\[
\cC_{\cF|_{\Gl}}\subset \cC_\cF|_{\Gl\times \Gl}.
\]
We now prove the opposite inclusion.
Let $x,y\in \Gl$ and assume that $x\to y$.
Fix an $\varepsilon>0$ and choose an $a>0$ such that $\overline{N_a(\Gl)}$
is compact.
By uniform continuity of the finitely many maps on $\overline{N_a(\Gl)}$ there is a $\rho\in(0,\min\{a,\varepsilon/3\})$ such that
\[
u,v\in \overline{N_a(\Gl)},\quad d(u,v)<\rho
\quad\Longrightarrow\quad
d\bigl(f_i(u),f_i(v)\bigr)<\varepsilon/3
\]
for every $i=1,\ldots,m$.

    Apply Lemma~\ref{lemma: robust trapping region} with outer radius $\rho$. 
    There exists $\delta_0>0$ such that every $\delta_0$-chain starting in $\Gl$ is contained in $N_\rho(\Gl)$. 
    Set $\delta=\min\{\delta_0,\varepsilon/3\}$.
    Choose a $\delta$-chain $x_0=x,x_1,\ldots,x_n=y$.
    Then $x_j\in N_\rho(\Gl)$ for every $j$. 
    Choose $z_j\in \Gl$ such that $d(x_j,z_j)<\rho$, with $z_0=x$ and $z_n=y$.

    For every $j$,
$$
d\bigl(f_{i_j}(z_j),z_{j+1}\bigr)
\leq
d\bigl(f_{i_j}(z_j),f_{i_j}(x_j)\bigr)
+d\bigl(f_{i_j}(x_j),x_{j+1}\bigr)
+d(x_{j+1},z_{j+1})\\
<\frac{\varepsilon}{3}
+\frac{\varepsilon}{3}
+\frac{\varepsilon}{3}
=\varepsilon.
$$
    Thus $z_0,\ldots,z_n$ is an $\varepsilon$-chain contained in $\Gl$ and so $(x,y)\in \cC_{\cF|_{\Gl}}$.
\end{proof}

\medskip\noindent
{\BF Orbitwise chain-recurrence, forward-invariance and $H$-invariance.}
When $m=1$, the chain-recurrent set of a semiflow with compact dynamics and all of its nodes are invariant~\cite{DLY25,DLY26a}.
By Proposition~\ref{prop: streams} we know that forward-invariance still holds for the orbitwise chain-recurrent set and each orbitwise chain node when $m>1$.
These sets are not necessarily $H$-invariant but we have the following weaker property.

\begin{theorem}
    \label{thm: N is H-invariant}
    Assume that $\cF$ has compact dynamics and let $N$ be one of its orbitwise chain nodes.
    Then $N$ is forward-invariant and $\Om(N)$ is the greatest non-empty compact $H$-invariant subset of $N$.
\end{theorem}
\begin{proof}
    We know that $N$ is compact and forward-invariant.
    Hence, by Lemma~\ref{lemma: O(A)=A}, $\Om(N)\subset N$ and so $\Om(N)$ is chain-equivalent.
    Moreover, by Lemma~\ref{lemma: Om} and Proposition~\ref{prop: omega(K) subset G}, $\Om(N)$ is $H$-invariant and non-empty.
    
    Finally, let \(K\subset N\) be a nonempty compact \(H\)-invariant set and fix \(y\in K\). 
    Since \(H_\cF^n(K)=K\) for every \(n\geq1\), there exist \(y_n\in K\) and \(I_n\in \cI^m\), with \(|I_n|=n\), such that $y=\Pn(y_n)$.
    Thus, \(y\in\Om(N)\) and so
\[
K\subset\Om(N),
\]
    which proves that \(\Om(N)\) is in fact the greatest compact \(H\)-invariant subset of \(N\).
\end{proof}

\medskip\noindent
{\bf Bitrajectories.}
A {\bf bitrajectory} of $\cF$ is a pair of a sequence of points $b_n\in X$ and a sequence of indices $i_n\in\{1,\dots,m\}$ defined for all $n\in\bZ$ and such that
\[
b_{n+1}=f_{i_n}(b_n)\qquad \forall n\in\mathbb Z.
\]
We denote by $\alpha(b)$ and $\omega(b)$, respectively, the $\alpha$- and
$\omega$-limit sets of $b$.
We say that $b$ is a {\BF bitrajectory through $x$} if $x\in b$.

\medskip
Notice that, since the global attractor is $H$-invariant, every $x\in\Gl$ has at least a bitrajectory of $\cF|_{\Gl}$ passing through it.
\begin{lemma}
    \label{lemma: limit sets of single itineraries}
    Let $S$ be a stream of $\cF$ and let $b=(b_n)_{n\geq0}$ be a trajectory of $\cF$ contained in a compact set.
    Then $\omega(b)$ is $S$-equivalent and is contained in a node of $S$. 
    If $b=(b_n)_{n\in\mathbb Z}$ is a bitrajectory contained in a compact set, then the same holds for $\alpha(b)$.
\end{lemma}

\begin{proof}
    We prove the claim for $\omega(b)$; the reader can adapt the proof to the case $\alpha(b)$.
    
    Let $b=(b_n)_{n\geq0}$ and assume first that $\omega(b)=\{p\}$.
    Since the trajectory is contained in a compact set and has a unique limit point, we have $b_n\to p$.
    Since $\cF$ is finite, some index $i\in\{1,\ldots,m\}$ occurs infinitely often in the itinerary.
    Choose $n_k\to\infty$ such that $i_{n_k}=i$. 
    Then $b_{n_k}\to p$, $b_{n_k+1}=f_i(b_{n_k})\to p$.
    By continuity of $f_i$, $f_i(p)=p$.
    Hence, $\{p\}$ is contained in a node of $S$.

    Assume now that $\omega(b)$ is not a singleton and let $p,q\in\omega(b)$ be distinct. 
    Since both $p$ and $q$ are limit points of the trajectory, we can choose integers $n_k<r_k$, $n_k\to\infty$, such that $b_{n_k}\to p$ and $b_{r_k}\to q$.
    Since $b_{r_k}$ is an exact orbit point of $b_{n_k}$, we have that $b_{n_k}\succcurlyeq_S b_{r_k}$.
    The closedness of $S$ therefore gives $p\succcurlyeq_S q$.
    Interchanging the roles of $p$ and $q$ gives $q\succcurlyeq_S p$.
    Thus, all points of $\omega(b)$ are mutually $S$-equivalent.
    Moreover, using the same argument above, one can prove that there is a generator $f_i$ such that $f_i(p)\in\omega(b)$.
    Hence, $f_i(p)\Seq p$ and so $p$ is $S$-recurrent and, by transitivity, $\omega(b)\subset\Node_S(p)$.
\end{proof}
%
  
The following result shows that a point belongs to a given chain node if and only if there is a bitrajectory through it whose $\alpha$ and $\omega$ limits belong to that same node.
\begin{lemma}[Bitrajectories inside chain nodes]
\label{lemma: characterization of weak recurrence}
    Assume that $\cF$ has compact dynamics. 
    Let $N$ be a chain node of $\cF$ and let $x\in \Gl$. 
    Then $x\in N$ if and only if there exists a bitrajectory $b=(b_n)_{n\in\mathbb Z}$ through $x$ entirely contained in $\Gl$ such that $\omega(b)\cup\alpha(b)\subset N$.
    Moreover, if $x\in N$, the bitrajectory may be chosen entirely inside $N$.
\end{lemma}
\begin{proof}
    We first show that every point $z\in N$ has both a successor and a predecessor in $N$.

    By Proposition~\ref{prop: alternative weak node}, there exists a generator $f_i$ such that $f_i(z)\Cto z$.
    Since the chain stream contains the orbit relation, we also have $z\Cto f_i(z)$.
    Thus, $z\Ceq f_i(z)$ and therefore $f_i(z)\in N$. 
    Hence, every point of $N$ has a successor in $N$.

    Now, if $f_i(z)=z$, then $z$ itself is a predecessor of $z$, and there is nothing to prove. 
    Assume, therefore, that $f_i(z)\neq z$ and let $\varepsilon_n\downarrow0$. 
    Since $f_i(z)\Cto z$ and both points belong to $\Gl$, Theorem~\ref{prop: C = C glob} allows us to choose, for every $n$, an $\varepsilon_n$-chain contained in $\Gl$,
\[
u^{(n)}_0=f_i(z),\,
u^{(n)}_1,\ldots,
u^{(n)}_{k_n}=z.
\]
    Since $f_i(z)\neq z$, we have $k_n\geq1$. 
    Let $y_n=u^{(n)}_{k_n-1}$ and let $j_n$ be the index associated with the last jump of the chain. 
    Thus
\beqn
\label{eq: 5.13-1}
d\bigl(f_{j_n}(y_n),z\bigr)<\varepsilon_n.
\eeqn

    Since $\Gl$ is compact and the alphabet is finite, after passing to a subsequence we may assume that $y_n\to y\in \Gl$, $j_n=j$, for some fixed $j\in\{1,\ldots,m\}$. Passing to the limit in \eqref{eq: 5.13-1} gives
\[
f_j(y)=z.
\]

    We claim that $z\Cto y$.
    Indeed, prepend the exact step $z\to f_i(z)$ to the $\varepsilon_n$-chain above and omit its final point $z$.
    This gives an $\varepsilon_n$-chain from $z$ to $y_n$. 
    Since $y_n\to y$ and $\varepsilon_n\to0$, replacing its final point $y_n$ by $y$ shows that, for every $\varepsilon>0$ and all sufficiently large $n$, there is an $\varepsilon$-chain from $z$ to $y$. 
    Hence, $z\Cto y$.

    On the other hand, since $f_j(y)=z$, we get that
    $y\Cto z$ and so $y\Ceq z$.
    Moreover, either $y=z$, in which case $f_j(y)=y$, or $y\neq z$ and $y$ is chain-equivalent to the distinct point $z$. 
    Thus, by Proposition~\ref{prop: alternative weak node}, $y$ is chain-recurrent, and consequently $y\in N$. 

    We now construct the required bitrajectory.
    Set $b_0=x$. 
    Recursively choose, for every $n\geq0$, a point $b_{n+1}\in N$ and an index $i_n$ such that $b_{n+1}=f_{i_n}(b_n)$.
    Using the predecessor property, recursively choose, for every $n\leq-1$, a point $b_n\in N$ and an index $i_n$ such that $b_{n+1}=f_{i_n}(b_n)$.
    Then $b=(b_n)_{n\in\mathbb Z}$ is a bitrajectory through $x$ contained entirely in $N$.

    Finally, $N$ is compact by Corollary~\ref{cor: chain nodes are compact}.
    Hence, both $\alpha(b)$ and $\omega(b)$ are nonempty. 
    Since $N$ is closed and $b_n\in N$ for every $n$, we also have $\alpha(b)\subset N$ and $\omega(b)\subset N$.

    \medskip
    Conversely, suppose that $b$ is a bitrajectory through $x$ and that $\alpha(b)\cup\omega(b)\subset N$.
    After reindexing, assume that $b_0=x$. Choose $p\in\alpha(b)$ and $q\in\omega(b)$.
    Since $\cC_\cF$ is closed, $p\Cto x\Cto q$.
    Since $p,q\in N$, we also have $q\Cto p$, so that $p\Cto x\Cto q\Cto p$.
    Hence, $x$ is chain-equivalent to $p$, so that $x$ is chain-recurrent and belongs to $N$.
\end{proof}

\medskip\noindent
{\bf Chain graph and orbitwise chain graph.}
As mentioned earlier in this article, one of the most general results of the theory of streams is that, whenever the global attractor of a semiflow $F$ is connected, the chain graph of $F$ is connected as well.
We prove below that the same result holds for IFSs as well.
The argument is an elementary adaptation of the one we used in case of semiflows and we include it here for completeness sake.
On the other side, the relevance of this result for IFS with $m>1$ is weaker than for semiflows (i.e. for $m=1$) since chain nodes for $m>1$ are not necessarily forward-invariant and so they do not represent, in general, elementary blocks of the dynamics of the IFS as a whole.
That is where the orbitwise chain graph will come into play.


\begin{theorem}
    \label{thm: weak chain graph is connected}
    Assume that $\cF$ has compact dynamics and that $\Gl$ is connected. 
    Then $\Gamma_{\cC_\cF}$
    is a connected graph.
\end{theorem}
\begin{proof}
    Suppose, by contradiction, that $\Gamma_{\cC_\cF}$ is disconnected.
    Then its chain nodes admit a partition into two nonempty families $\mathcal N_1,\mathcal N_2$ such that there is no edge between a node in $\mathcal N_1$ and a node in $\mathcal N_2$.
    For $i=1,2$, let $N_i=\bigcup_{M\in\mathcal N_i}M$.
    By the definition of a node partition, $N_1$ and $N_2$ are closed and disjoint and, by Corollary~\ref{cor: chain nodes are compact}, they lie in $\Gl$; in particular, they are compact.


    Now, choose a bitrajectory through $x\Gl$ fully contained in $\Gl$, whose existence follows from H-invariance.
    By Lemma~\ref{lemma: limit sets of single itineraries}, $\alpha(b)$ and $\omega(b)$ are contained in chain nodes. 
    Let $p\in\alpha(b)$, $q\in\omega(b)$.
    There are sequences $n_k,r_k\to\infty$ such that $b_{-n_k}\to p$, $b_{r_k}\to q$.   
    Since $b_{-n_k}\Fto x$ and $x\Fto b_{r_k}$, closedness of $\cC_\cF$ gives $p\Cto x\Cto q$.
    By transitivity, $p\Cto q$.
    
    If $p$ and $q$ belonged to different families, this would give an edge between the two families, contradicting the initial assumption.
    Hence, the chain nodes containing $p$ and $q$ must belong to the same family. 
    Now, for $i=1,2$, define
\[
A_i=
\left\{
x\in \Gl:
\text{there exist }p,q\in N_i
\text{ such that }
p\Cto x\Cto q
\right\}.
\]
    The preceding argument shows that $\Gl=A_1\cup A_2$.

    We claim that $A_1\cap A_2=\varnothing$.
    Indeed, if $x\in A_1\cap A_2$, then there would exist $p_1,q_1\in N_1$ and $p_2,q_2\in N_2$ such that $p_1\Cto x\Cto q_1$ and $p_2\Cto x\Cto q_2$.
    Hence, $p_1\Cto q_2$.
    This gives an edge from a chain node in $\mathcal N_1$ to a chain node in $\mathcal N_2$, a contradiction.

    Now, both $A_1$ and $A_2$ are nonempty since $N_i\subset A_i$ for $i=1,2$. 
    Moreover, $A_1$ and $A_2$ are closed. 
    We prove this for $A_1$. 
    Let $x_n\in A_1$, $x_n\to x\in \Gl$.
    For every $n$, choose $p_n,q_n\in N_1$ such that
\[
p_n\Cto x_n\Cto q_n.
\]
    Since $N_1$ is compact, after passing to a subsequence we may assume that $p_n\longrightarrow p\in N_1$ and $q_n\longrightarrow q\in N_1$.
    Since $\cC_\cF$ is closed, $p\Cto x\Cto q$.
    Therefore $x\in A_1$. The same argument proves that $A_2$ is closed.

    We have thus written the connected set $\Gl$ as the disjoint union of two nonempty closed subsets, against the assumption that $\Gl$ is connected.
    Hence, $\Gamma_{\cC_\cF}$ is connected.
\end{proof}

\medskip\noindent
{\bf IFSs with countably many chain nodes.}
We show below that, just as in the $m=1$ case, when $\cF$ has countably many chain nodes its smallest stream is the chain stream.

\begin{proposition}[Lyapunov levels and chain nodes]
    \label{prop: Lyapunov levels and chain nodes}
    Assume that $\cF$ has compact dynamics. 
    Let $L\in L_\cF$ and let
    $u,v\in \Gl$ satisfy $u\Cto  v$ and $L(u)<L(v)$.
    Then $\cA_\cF$ has uncountably many nodes.
\end{proposition}

\begin{proof}
    For $r\geq 0$, let
\[
\vartheta_L(r)
=
\sup\left\{
|L(a)-L(b)|:
a,b\in \Gl,\ d(a,b)\leq r
\right\}.
\]
    Since $\Gl$ is compact and $L$ is continuous, then $\vartheta_L(r)\to 0$ as $r\longrightarrow 0$.

    Choose a sequence $\varepsilon_n\downarrow 0$. 
    Since $u\Cto v$ and the chain relation of $\cF$ agrees on $\Gl$ with the chain relation of $\cF|_{\Gl}$, for every $n$ there is an $\varepsilon_n$-chain contained in $\Gl$,
\[
x^{(n)}_0=u,x^{(n)}_1,\ldots,x^{(n)}_{k_n}=v,
\]
    with associated indices $i^{(n)}_0,\ldots,i^{(n)}_{k_n-1}$.
    Set $\delta_n=\vartheta_L(\varepsilon_n)$.
    Notice that $\delta_n>0$ since, otherwise, $L$ would be constant over the whole chain, against the fact that $L(u)<L(v)$.
    For every $j=0,\ldots,k_n-1$,
\[
d\left(
f_{i^{(n)}_j}(x^{(n)}_j),x^{(n)}_{j+1}
\right)<\varepsilon_n.
\]
    Since $L$ is a Lyapunov function for $\cF$,
\[
L\left(f_{i^{(n)}_j}(x^{(n)}_j)\right)
\leq L(x^{(n)}_j),
\]
    and therefore
\begin{equation}
\label{eq: 5.17-1}
L(x^{(n)}_{j+1})
\leq L(x^{(n)}_j)+\delta_n.
\end{equation}

    Fix a $t\in\bigl(L(u),L(v)\bigr)$ and define $r_n=1+\max\left\{j:L(x^{(n)}_j)<t \right\}$.
    This is well-defined because $L(x^{(n)}_0)=L(u)<t<L(v)=L(x^{(n)}_{k_n})$.
    By the definition of $r_n$,
\begin{equation}
\label{eq: 5.17-2}
L(x^{(n)}_j)\geq t
\qquad\text{for every }j\geq r_n.
\end{equation}
    Moreover,
\[
L(x^{(n)}_{r_n-1})<t,
\]
    so~\eqref{eq: 5.17-1} gives
\begin{equation}
\label{eq: 5.17-3}
t\leq L(x^{(n)}_{r_n})<t+\delta_n.
\end{equation}

    We claim that $k_n-r_n\to\infty$.
    Indeed, iterating~\eqref{eq: 5.17-1} from $r_n$ to $k_n$ and using~\eqref{eq: 5.17-3}, we obtain
\[
L(v)
\leq
L(x^{(n)}_{r_n})+(k_n-r_n)\delta_n
<
t+\delta_n+(k_n-r_n)\delta_n.
\]
Hence
\[
k_n-r_n
>
\frac{L(v)-t}{\delta_n}-1,
\]
    and the right-hand side tends to infinity because $\delta_n\to 0$.

    Fix now an integer $\ell\geq 0$. For all sufficiently large $n$, $r_n+\ell\leq k_n$. From~\eqref{eq: 5.17-1}--\eqref{eq: 5.17-3},
\begin{equation}
\label{eq: 5.17-4}
t
\leq L(x^{(n)}_{r_n+\ell})
<
t+(\ell+1)\delta_n.
\end{equation}
    Consequently,
\begin{equation}
\label{eq: 5.17-5}
L(x^{(n)}_{r_n+\ell})
\longrightarrow t
\qquad\text{for every fixed }\ell\geq 0.
\end{equation}

    Using compactness of $\Gl$, finiteness of the alphabet and a diagonal argument, we may pass to a subsequence such that, for every fixed $\ell\geq 0$, $x^{(n)}_{r_n+\ell}\to c^t_\ell$ and $i^{(n)}_{r_n+\ell}=i^t_\ell$ for all sufficiently large $n$, for some $i^t_\ell\in\{1,\ldots,m\}$.

Since
\[
d\left(
f_{i^{(n)}_{r_n+\ell}}
(x^{(n)}_{r_n+\ell}),
x^{(n)}_{r_n+\ell+1}
\right)
<\varepsilon_n,
\]
    passing to the limit gives $c^t_{\ell+1}=f_{i^t_\ell}(c^t_\ell)$.
    Thus, $\boldsymbol{c}^t=(c^t_\ell)_{\ell\geq 0}$ is a genuine one-sided trajectory in $\Gl$. 
    By~\eqref{eq: 5.17-5}, $L(c^t_\ell)=t$ for every $\ell\geq 0$.
    It follows that its limit set $W_t=\omega(\boldsymbol{c}^t)$ is nonempty and satisfies
\begin{equation}
\label{eq: 5.17-6}
L|_{W_t}\equiv t.
\end{equation}

    By Lemma~\ref{lemma: limit sets of single itineraries}, $W_t$ is contained in a chain node $P_t$ of the smallest stream $\cA_\cF$.
    Recall that $\cA_\cF=\cD_{L_\cF}$.
    Hence, every $L\in \cL_\cF$ is constant on each $\cA_\cF$-equivalence class:
    if $p=_{\cA_\cF}q$, then $L(p)\geq L(q)$ and $L(q)\geq L(p)$, and therefore $L(p)=L(q)$.
    It follows that distinct values of $t$ determine distinct chain nodes.
    Indeed, suppose that $s\neq t$ and $P_s=P_t$ and choose $p\in W_s$ and $q\in W_t$.
    Then $p\Aeq q$, and therefore $s=L(p)=L(q)=t$, a contradiction.
    Thus, the map $t\longmapsto P_t$ is injective on the nonempty interval $\bigl(L(u),L(v)\bigr)$.
    Consequently, $\cA_\cF$ has uncountably many nodes.
\end{proof}
\begin{corollary}
    Assume that $\cF$ has compact dynamics and that $\cA_\cF$ has at most countably many nodes. 
    Then, for every $L\in \cL_\cF$ and every $u,v\in \Gl$,
\[
u\Cto v
\quad\Longrightarrow\quad
L(u)\geq L(v).
\]
\end{corollary}
\begin{proof}
    If there were $L\in \cL_\cF$ and $u,v\in \Gl$ such that $u\Cto v$ and $L(u)<L(v)$, the preceding proposition would imply that $\cA_\cF$ has uncountably many nodes.
\end{proof}
\begin{theorem}
\label{thm: Omega-stream}
    Assume that $F$ has compact dynamics and that $\cA_\cF$ has at most countably many nodes. 
    Then
\[
\cA_\cF=\cC_\cF.
\]
\end{theorem}
\begin{proof}
    Since $\cA_\cF$ is the smallest stream of $\cF$ and $\cC_\cF$ is a stream, $\cA_\cF\subset \cC_\cF$.
    It remains to prove the opposite inclusion.

    Let $(x,y)\in \cC_\cF$.
    We prove that $L(x)\geq L(y)$ for every $L\in \cL_\cF$.
    Since $\cA_\cF=D_{\cL_\cF}$, this will imply $(x,y)\in \cA_\cF$.
    If $y\in \cO_\cF(x)$, then the desired inequality follows immediately from the definition of a Lyapunov function.

    Assume now that $y\notin \cO_\cF(x)$.
    Since $\cC_\cF$ is an $\Omega$-stream, there exists $z\in\Om(x)$ such that $z\Cto y$.
    Since $\cF$ has compact dynamics, $z\in \Gl$.
    Moreover, $\Gl$ is a trapping region for $\cC_\cF$. 
    Therefore, $z\Cto y$ implies $y\in \Gl$.
    Fix $L\in \cL_\cF$. 
    Since $z\in\Om(x)$, there are multiindices $I_n$ with $|I_n|\to\infty$ such that $\Pn(x)\to z$. 
    For every $n$, $L(x)\geq L\left(\Pn(x)\right)$.
    Passing to the limit and using continuity of $L$, we obtain $L(x)\geq L(z)$.
    Since $z,y\in \Gl$ and $z\Cto y$, the preceding corollary gives $L(z)\geq L(y)$.
    Consequently, $L(x)\geq L(z)\geq L(y)$.
    Since this holds for every $L\in \cL_\cF$, $(x,y)\in D_{\cL_\cF}=\cA_\cF$.
    Thus, $\cC_\cF\subset \cA_\cF$.
    Together with the opposite inclusion, this proves $\cA_\cF=\cC_\cF$.
\end{proof}

\begin{corollary}
    Assume that $\cF$ has compact dynamics, that $\Gl$ is connected and that $\cA_\cF$ has at most countably many nodes.
    Then the graph of any stream of $\cF$ is connected.
\end{corollary}

\medskip\noindent
{\bf Decomposition into recurrent and gradient dynamics of an IFS.}
For $m=1$, Conley proved in~\cite{Con78} that, when $X$ is compact, every point $x\in X$ either belongs to a chain node or every bitrajectory $b$ through it has the property that $\alpha(b)$ is a subset of some chain node $M$ while $\omega(b)$ is a subset of a different chain node $N$ (so, in particular, $M\Cto N$).
In the first case Conley says that the dynamics of $x$ is of ``recurrent type'', in the second of ``gradient type''.
In our setting, when $X$ is locally compact, we can extend this result to semiflows with compact dynamics, so that every point of $\Gl$ is either of recurrent or of gradient type as explained above.


For $m>1$, bitrajectories through a given point do not need to have the same $\omega$-limits and so the dynamics of $x$ can be of three distinct types.
Denote by $\cB_x$ the set of all bitrajectories through $x$ lying entirely in $\Gl$.
We have the following trichotomy.
\begin{theorem}[Bitrajectory characterization and trichotomy]
    \label{thm: trichotomy}
    Assume that \(\cF\) has compact dynamics and let $x\in\Gl$. 
    Then exactly one of the following three cases occurs.

\begin{enumerate}
    \item {\bf Orbitwise recurrent dynamics.}
    The point \(x\) is orbitwise chain-recurrent.
    Equivalently, there exists a chain node \(N\) such that
\[
\omega(b)\subset N
\qquad\text{for every }b\in\mathcal B_x,
\]
and
\[
\alpha(b_0)\subset N
\qquad\text{for some }b_0\in\mathcal B_x.
\]

\item {\bf Mixed recurrent-gradient dynamics.}
    The point \(x\) is chain-recurrent but not orbitwise chain-recurrent.
    Equivalently, denoting by \(N=\Node_{\cC_\cF}(x)\) its chain node, there exist bitrajectories \(b_{\mathrm r},b_{\mathrm g}\in\mathcal B_x\) and a chain node \(P\neq N\) such that
\[
\alpha(b_{\mathrm r})\cup\omega(b_{\mathrm r})\subset N,
\]
whereas
\[
\alpha(b_{\mathrm g})\subset N,
\qquad
\omega(b_{\mathrm g})\subset P.
\]

\item {\bf Gradient dynamics.}
    The point \(x\) is not chain-recurrent.
    Equivalently, for every \(b\in\mathcal B_x\), the sets \(\alpha(b)\) and \(\omega(b)\) are contained in distinct chain nodes. 
    In particular, every such bitrajectory determines an edge
\[
N_{\alpha(b)}\longrightarrow N_{\omega(b)}
\]
    between distinct nodes of the chain graph.
\end{enumerate}
\end{theorem}

\begin{proof}
    By Lemma~\ref{lemma: characterization of weak recurrence}, a point \(x\in \Gl\) is chain-recurrent if and only if there exists a bitrajectory \(b\in\mathcal B_x\) whose \(\alpha\)- and \(\omega\)-limit sets are contained in the same chain node. 
    This immediately proves the characterization in item~(3).

    We next prove item~(1). 
    Assume first that \(x\) is orbitwise chain-recurrent, and let \(N=\Node_{\cC_\cF}(x)\). 
    If \(b=(b_n)_{n\in\mathbb Z}\in\mathcal B_x\), after reindexing we may assume that \(b_0=x\). 
    For every \(n\geq 0\), one has \(b_n\in \cO_\cF(x)\). 
    Since \(x\) is orbitwise chain-recurrent, $b_n\Ceq x$, and so \(b_n\in N\). 
    Since \(N\) is closed, it follows that $\omega(b)\subset N$.
    Moreover, \(x\) is chain-recurrent, so that, by Lemma~\ref{lemma: characterization of weak recurrence}, there is a bitrajectory \(b_0\in\mathcal B_x\) entirely contained in \(N\).
    In particular, \(\alpha(b_0)\subset N\).

    Conversely, assume that there exists a chain node \(N\) such that \(\omega(b)\subset N\) for every \(b\in\mathcal B_x\), and that \(\alpha(b_0)\subset N\) for some \(b_0\in\mathcal B_x\). 
    Let \(z\in \cO_\cF(x)\) and choose a finite exact orbit segment from \(x\) to \(z\).
    Append to it an arbitrary forward trajectory starting at \(z\) and prepend the negative half of \(b_0\). 
    This produces a bitrajectory \(b\in\mathcal B_x\) passing through \(z\), with $\alpha(b)\subset N$.
    By hypothesis, also \(\omega(b)\subset N\). 
    Applying Lemma~\ref{lemma: characterization of weak recurrence} to every point of \(b\), we obtain \(x,z\in N\). 
    Thus, $x\Ceq z$.
    Since \(z\in \cO_\cF(x)\) was arbitrary, \(x\) is orbitwise chain-recurrent.

    It remains to prove item~(2).
    Suppose first that \(x\) is chain-recurrent but not orbitwise chain-recurrent, and let $N=\Node_{\cC_\cF}(x)$.
    By Lemma~\ref{lemma: characterization of weak recurrence}, there exists \(b_{\mathrm r}\in\mathcal B_x\), entirely contained in \(N\), and therefore
\[
\alpha(b_{\mathrm r})\cup\omega(b_{\mathrm r})\subset N.
\]

    Since \(x\) is not orbitwise chain-recurrent, there exists \(z\in \cO_\cF(x)\) such that \(z\notin N\). 
    Prepend the negative half of \(b_{\mathrm r}\) to a finite exact orbit segment from \(x\) to \(z\), and then append an arbitrary forward trajectory starting at \(z\). 
    The resulting bitrajectory \(b_{\mathrm g}\) satisfies $\alpha(b_{\mathrm g})\subset N$.
    By Lemma~\ref{lemma: characterization of weak recurrence}, its \(\omega\)-limit set is contained in some chain node \(P\). 
    Necessarily \(P\neq N\): otherwise
\[
\alpha(b_{\mathrm g})\cup\omega(b_{\mathrm g})\subset N,
\]
    and Lemma~\ref{lemma: characterization of weak recurrence}, applied to the point \(z\) of \(b_{\mathrm g}\), would imply \(z\in N\), a contradiction.

    Conversely, the existence of \(b_{\mathrm r}\) implies, by Lemma~\ref{lemma: characterization of weak recurrence}, that \(x\) is chain-recurrent and belongs to \(N\). 
    If \(x\) were orbitwise chain-recurrent, every point on the positive half of \(b_{\mathrm g}\), and therefore every point of \(\omega(b_{\mathrm g})\), would belong to \(N\). 
    This contradicts the fact that \(\omega(b_{\mathrm g})\subset P\) with \(P\neq N\). 
    Hence, \(x\) is not orbitwise chain-recurrent.

    The three alternatives are mutually exclusive and exhaust all points of \(\Gl\).
\end{proof}
The three alternatives above can be also reformulated in terms of Lyapunov functions as follows.
\begin{corollary}[Lyapunov characterization of the trichotomy]
    Assume that $\cF$ has compact dynamics and let
\[
L:X\longrightarrow[0,1]
\]
    be a strict Lyapunov function for the chain stream $\cC_\cF$.
    Then, for every $x\in \Gl$, exactly one of the following three alternatives occurs.
    \begin{enumerate}
    \item \emph{Orbitwise recurrent dynamics.}
    The point $x$ is orbitwise chain-recurrent if and only if
\[
L(\P(x))=L(x)
\qquad\text{for every } I\in \cI^m .
\]

\item \emph{Mixed recurrent--gradient dynamics.}
The point $x$ is chain-recurrent but not orbitwise chain-recurrent
if and only if
\[
L(f_i(x))=L(x)
\]
for some $i\in\{1,\ldots,m\}$, while
\[
L(\P(x))<L(x)
\]
for some $I\in \cI^m$.

\item \emph{Gradient dynamics.}
    The point $x$ is not chain-recurrent if and only if
\[
L(f_i(x))<L(x)
\qquad\text{for every }i=1,\ldots,m .
\]
\end{enumerate}
\end{corollary}

\begin{proof}
    A strict Lyapunov function $L$ for $\cF$ exists by Proposition~\ref{prop:strict-lyapunov}.
    Since the chain stream contains the orbit relation, for every $I\in \cI^m$ we have
\[
x\Cto\P(x).
\]
    As $L$ is strict,
\[
L(\P(x))=L(x)
\quad\Longleftrightarrow\quad
\P(x)\Cto x,
\]
and so
\[
L(\P(x))=L(x)
\quad\Longleftrightarrow\quad
\P(x)\Ceq x.
\]
    The first assertion now follows directly from the definition of orbitwise chain recurrence.

    By Proposition~\ref{prop: alternative weak node}, $x$ is chain-recurrent if and only if $f_i(x)\Cto x$ for some generator $f_i$. 
    By the preceding observation, this is equivalent to
\[
L(f_i(x))=L(x)
\]
for some $i$. Thus $x$ is not chain-recurrent precisely when
\[
L(f_i(x))<L(x)
\]
for every $i$.

    Finally, if $x$ is chain-recurrent but not orbitwise chain-recurrent, equality holds for at least one generator, while there exists $I\in \cI^m$ such that $\P(x)\not=_{\cC_\cF}x$. 
    Since $x\succeq_{\cC_\cF}\P(x)$, strictness of $L$ gives
\[
L(\P(x))<L(x).
\]
    The converse is immediate.
\end{proof}


\medskip\noindent
{\bf When the orbitwise chain graph of an IFS is connected.}
As we already pointed out, the graph that represents the semigroup dynamics of the IFS is the orbitwise one.
We will show at the end of the section with a fully worked out example that, unlike the ``plain'' chain graph, the orbitwise one is not necessarily connected.
Now, we provide a general condition that grants the connectedness of the orbitwise chain graph.



\begin{definition}
We say that $\cF$ has \emph{backward recurrent access} if, for every
$x\in \Gl$, there exists a backward trajectory
\[
b^-=\{x_0=x,x_{-1},x_{-2},\dots\}\subset \Gl
\]
such that
\[
\alpha(b^-)\cap \cR^{ow}_{\cC_\cF}\neq\varnothing.
\]
\end{definition}



\begin{theorem}
\label{thm: G conn => graph conn}
    Assume that $\cF$ has compact dynamics, that $\Gl$ is connected and that $\cF$ has backward recurrent access.
    Then the orbitwise graph $\Gamma^{ow}_{\cC_\cF}$ of the chain stream is connected.
\end{theorem}
\begin{proof}
    Suppose, by contradiction, that $\Gamma^{ow}_{\cC_\cF}$ is disconnected.
    Then there exists a node partition $\mathcal N_1,\mathcal N_2$ such that there is no edge between a node in $\mathcal N_1$ and a node in $\mathcal N_2$. 
    Denote by $N_i=\bigcup_{M\in\mathcal N_i}M$, $i=1,2$, the corresponding unions of node points. 
    By the definition of a node partition, $N_1$ and $N_2$ are nonempty, closed, and disjoint.

    We first associate to every point $x\in \Gl$ one of the two classes $\mathcal N_1,\mathcal N_2$.
    By backward recurrent access, there exists a backward trajectory $b^-=\{b_0=x,b_{-1},b_{-2},\ldots\}\subset \Gl$
    and a point $p\in\alpha(b^-)\cap \cR^{ow}_{\cC_\cF}$.
    Choose a sequence $n_k\to\infty$ such that $b_{-n_k}\to p$.
    Since $b_{-n_k}\succcurlyeq_\cF x$ for every $k$, we have $b_{-n_k}\Cto x$. 
    Since $\cC_\cF$ is closed, $p\Cto x$.
    On the other hand, Proposition~\ref{prop: streams}(4), applied to $\cC_\cF$, gives a point $q\in\Om(x)\cap \cR^{ow}_{\cC_\cF}$.
    
    There are multiindices $I_k$, with $|I_k|\to\infty$, such that $\Phi_\cF^{I_k}(x)\to q$. 
    Since $x\Cto \Phi_\cF^{I_k}(x)$ for every $k$, closedness gives $x\Cto q$.
    Consequently,
\beqn
\label{eq: 8-1}
p\Cto x\Cto q.
\eeqn

    Both $p$ and $q$ are orbitwise chain-recurrent, and hence belong to orbitwise chain nodes. 
    These two nodes must belong to the same class $\mathcal N_i$. 
    Indeed, if the node containing $p$ belonged to $\mathcal N_1$ and the node containing $q$ belonged to $\mathcal N_2$, or vice versa, then by ~\eqref{eq: 8-1} there would be an edge between the two classes, contrary to the choice of the partition.

    For $i=1,2$, set 
    \[
A_i
=
\left\{
x\in \Gl:
\text{there exist }p,q\in N_i
\text{ such that }
p\Cto x\Cto q
\right\}.
\]
    The preceding argument shows that $\Gl=A_1\cup A_2$.    
    Moreover, the union is disjoint. 
    Indeed, if $x\in A_1\cap A_2$, then there would exist $p_1,q_1\in N_1$, $p_2,q_2\in N_2$ such that
\[
p_1\to x\to q_1,
\qquad
p_2\to x\to q_2.
\]
In particular,
\[
p_1\to x\to q_2,
\]
    which gives an edge from a node in $\mathcal N_1$ to a node in $\mathcal N_2$, a contradiction. 

    The sets $A_1$ and $A_2$ are nonempty. 
    Indeed, if $x\in N_i$, then, by reflexivity,
\[
x\to x\to x,
\]
    and therefore $x\in A_i$. 
    Thus, $N_i\subset A_i$.

    It remains to show that $A_1$ and $A_2$ are closed. 
    We prove this for $A_1$; the proof for $A_2$ is identical.
    Let $x_n\in A_1$ with $x_n\to x\in \Gl$.
    For every $n$, choose $p_n,q_n\in N_1$ such that $p_n\to x_n\to q_n$.
    By Corollary~\ref{cor: chain nodes are compact}, $\cR^{ow}_{\cC_\cF}$ is compact. 
    Since $N_1$ is a closed subset of $\cR^{ow}_{\cC_\cF}$, it is compact as well.
    Passing to a subsequence, we may assume that $p_n\to p\in N_1$, $q_n\to q\in N_1$.
    Since $\cC_\cF$ is closed, $p\to x\to q$.
    Hence $x\in A_1$, proving that $A_1$ is closed. 

    Since we assumed that $\Gl$ is connected, this is not possible.
    Consequently, every node partition admits an edge between its two classes, namely $\Gamma^{ow}_{\cC_\cF}$ is connected.
\end{proof}


\medskip\noindent
{\bf An example with not connected orbitwise chain graph.}
Set $X=[0,3]$ and $\cE=\{e_1,e_2\}$, where $e_1$ is the time-1 map of the ODE
$$
x'=-x(x-1)(x-3)
$$
and $e_2$ the time-1 map of
$$
x'=-x(x-2)(x-3).
$$
Then $\cE$ has two fixed points: $x=0$ and $x=3$.
We show below that $\Gamma_{\cC_\cE}$ consists of the three nodes $\{0\}$, $\{3\}$ and $\Lambda$, where $\Lambda\subset[1,2]$ is a Cantor set, and the two edges $\Lambda\to\{0\}$ and $\Lambda\to\{3\}$. 
%

\medskip
    {\BF The node $\Lambda$.}
    Set $v_1(x)=-x(x-1)(x-3)$ and $v_2(x)=-x(x-2)(x-3)$.
    On the interval $(1,3)$, $v_1(x)>0$ and so $e_1$ moves points of $(1,3)$ to the right whereas, similarly, $e_2$ moves points of $(0,2)$ to the left.
    There is a unique point $a\in(1,2)$ such that $e_1(a)=2$ and, in fact, $a<3/2$. 
    Indeed, on $[3/2,2]$, $v_1(x)\geq \frac{9}{8}$ and so
    the time required by the $v_1$-trajectory starting at $3/2$ to reach $2$ is at most
\[
\int_{3/2}^{2}\frac{ds}{v_1(s)}
\leq
\frac{2-3/2}{9/8}
=\frac49<1.
\]
    Hence, $e_1(3/2)>2$ and, consequently, $a<3/2$.

    The reflection $\rho(x)=3-x$ conjugates the two maps: $e_2=\rho\circ e_1\circ\rho$.
    Thus, setting $b=3-a$, we have that $\frac32<b<2$ and $e_2(b)=1$.

    The restrictions
\[
e_1|_{[1,a]}:[1,a]\longrightarrow[1,2],
\qquad
e_2|_{[b,2]}:[b,2]\longrightarrow[1,2]
\]
    are increasing homeomorphisms. 
    Let
\[
g_1=\bigl(e_1|_{[1,a]}\bigr)^{-1},
\qquad
g_2=\bigl(e_2|_{[b,2]}\bigr)^{-1}.
\]
Then
\[
g_1([1,2])=[1,a],
\qquad
g_2([1,2])=[b,2],
\]
and these two intervals are disjoint.

    We claim that $g_1$ and $g_2$ are contractions. 
    Let $\varphi_i^t$ denote the flow of $v_i$. The derivative of a time-one map is
\[
e_i'(x)
=
\exp\left(
\int_0^1 v_i'\bigl(\varphi_i^t(x)\bigr)\,dt
\right).
\]
    Now
\[
v_1'(x)=-3x^2+8x-3\geq1
\qquad\text{ for } x\in[1,2],
\]
and
\[
v_2'(x)=-3x^2+10x-6\geq1
\qquad\text{ for } x\in[1,2].
\]
    The $v_1$-orbit of every point of $[1,a]$ and the $v_2$-orbit of every point of $[b,2]$ remain in $[1,2]$ for times $t\in[0,1]$.
    Therefore, $e_1'(x)\geq e$ for $x\in[1,a]$ and $e_2'(x)\geq e$ for $x\in[b,2]$.
    It follows that $\operatorname{Lip}(g_1)\leq e^{-1}$ and $\operatorname{Lip}(g_2)\leq e^{-1}$.
    The contractive IFS $\{g_1,g_2\}$ therefore has a unique nonempty compact invariant set
\[
\Lambda=g_1(\Lambda)\cup g_2(\Lambda).
\]
    Since the images $g_1([1,2])$ and $g_2([1,2])$ are disjoint, the usual coding map
\[
\pi:\{1,2\}^{\mathbb N_0}\longrightarrow\Lambda
\]
    is a homeomorphism. 
    Hence $\Lambda$ is a Cantor set.
    Equivalently, $\Lambda$ is the survivor set
\[
\Lambda
=
\left\{
x\in[1,2]:
\begin{array}{l}
\text{there exist indices }i_0,i_1,\ldots
\text{ such that, setting}\\
x_0=x,\quad x_{n+1}=e_{i_n}(x_n),
\text{ one has }x_n\in[1,2]\text{ for every }n
\end{array}
\right\}.
\]

\medskip
\textbf{Chain equivalence on $\Lambda$.}
Define the branch map $T:\Lambda\to\Lambda$ by
\[
T(x)=
\begin{cases}
e_1(x),&x\in g_1(\Lambda),\\
e_2(x),&x\in g_2(\Lambda).
\end{cases}
\]
    The two domains are disjoint, so $T$ is well defined, and the coding map conjugates $T$ to the one-sided full shift:
\[
T\circ\pi=\pi\circ\sigma.
\]

    We show directly that any two points of $\Lambda$ are chain-equivalent under $\cE$. 
    Let
\[
x=\pi(\omega),
\qquad
y=\pi(\eta),
\]
where
\[
\omega=(\omega_0,\omega_1,\ldots),
\qquad
\eta=(\eta_0,\eta_1,\ldots).
\]
    Fix $\varepsilon>0$. Since the inverse branches are contractions, there exists $N\geq1$ such that every cylinder of length $N$ has diameter smaller than $\varepsilon$.

    Consider the sequence
\[
\zeta=
(\omega_1,\ldots,\omega_N,\eta_0,\eta_1,\ldots).
\]
    The sequences $\sigma(\omega)$ and $\zeta$ have the same first $N$ symbols, and hence $d\bigl(T(x),\pi(\zeta)\bigr)<\varepsilon$.
    Moreover, $T^N(\pi(\zeta))=y$.
    Therefore
\[
x,\ \pi(\zeta),\ T(\pi(\zeta)),\ldots,T^N(\pi(\zeta))=y
\]
    is an $\varepsilon$-chain for $F$ from $x$ to $y$: the first jump has error smaller than $\varepsilon$, and all subsequent jumps are exact applications of one of the generators.
    Interchanging $x$ and $y$ gives an $\varepsilon$-chain from $y$ to $x$. 
    Since $\varepsilon>0$ was arbitrary, $x$ is $\cC_\cE$-equivalent to $y$ for every $x,y\in\Lambda$.
    Thus, every point of $\Lambda$ is chain-recurrent.

\medskip
\textbf{Maximality of $\Lambda$.}
    We prove that no point outside $\Lambda$ is chain-equivalent to a point of $\Lambda$.
    For $n\geq0$, set
\[
Q_n
=
\bigcup_{i_0,\ldots,i_{n-1}\in\{1,2\}}
g_{i_0}\circ\cdots\circ g_{i_{n-1}}([1,2]).
\]
    Then
\[
Q_{n+1}\subset Q_n
\qquad\text{and}\qquad
\Lambda=\bigcap_{n\geq0}Q_n.
\]
    The set $Q_n$ consists precisely of the points from which some forward itinerary remains in $[1,2]$ for at least $n$ steps.
    Let $z\in[1,2]\setminus\Lambda$. 
    Then $z\notin Q_N$ for some $N$.
    Consequently, for every word $I$ of length $N$,
\[
\Phi_\cE^I(z)\in[0,1)\cup(2,3].
\]
    None of these finitely many images equals $1$ or $2$, since from either of these points there is an infinite itinerary contained in $[1,2]$. 
    Hence there is a $\delta>0$ such that
\beqn
\label{eq: 5.15-1}
\Phi_\cE^I(z)\in[0,1-\delta]\cup[2+\delta,3]
\qquad\text{for every }|I|=N.
\eeqn

    On $(0,1)$ both generators move every point to the left. 
    Define
\[
h(x)=\max\{e_1(x),e_2(x)\},\qquad x\in[0,1].
\]
    Then $h$ is increasing and $h(x)<x$ for $x\in(0,1)$, and therefore $h^n(1-\delta)\to0$.
    It follows that all sufficiently long continuations of any point in $[0,1-\delta]$ converge uniformly to $0$.

    Similarly, on $(2,3)$ both generators move every point to the right. 
    Defining
\[
k(x)=\min\{e_1(x),e_2(x)\},
\qquad x\in[2,3],
\]
    one has $k(x)>x$ for $x\in(2,3)$ and $k^n(2+\delta)\to3$.
    Thus, all sufficiently long continuations of a point in $[2+\delta,3]$ converge uniformly to $3$.

    Together with~\eqref{eq: 5.15-1}, this proves that
\beqn
\label{eq: 5.15-5}
\Omega_\cE(z)\subset\{0,3\}
\qquad\text{for every }z\notin\Lambda.
\eeqn

    We also observe that neither $0$ nor $3$ is chain-upstream from $\Lambda$. 
    Choose $r\in(0,1)$. 
    Since $e_i(r)<r$, $i=1,2$, the number
\[
\varepsilon_0
=
r-\max\{e_1(r),e_2(r)\}
\]
    is positive. 
    Every $\varepsilon$-chain starting at $0$, with $0<\varepsilon<\varepsilon_0$, remains in $[0,r]$. 
    Hence
\beqn
\label{eq: 5.15-6}
0\not\CEto x
\qquad\text{for every }x\in\Lambda.
\eeqn
    The symmetric argument near $3$ gives
\beqn
\label{eq: 5.15-7}
3\not\CEto x
\qquad\text{for every }x\in\Lambda.
\eeqn

    Suppose now that some $z\notin\Lambda$ were chain-equivalent to some $x\in\Lambda$. 
    No exact orbit of $z$ can meet $\Lambda$.
    Indeed, if $\Phi_\cE^I(z)\in\Lambda$,    then, working backwards along the word $I$ and using
\[
e_i^{-1}(\Lambda)=g_i(\Lambda)\subset\Lambda,
\]
    one would obtain $z\in\Lambda$.
    Thus, $x\notin \cO_\cE(z)$. 
    Since the chain stream is an $\Omega$-stream and $z\CEto x$, there exists a point $p\in\Omega_\cE(z)$ such that $p\CEto x$.
    By~\eqref{eq: 5.15-5}, $p\in\{0,3\}$, contradicting~\eqref{eq: 5.15-6} and~\eqref{eq: 5.15-7}.

    Therefore, no point of $X\setminus\Lambda$ is chain-equivalent to a point of $\Lambda$.
    Together with the chain-equivalence of $\Lambda$ proved above, this shows that $\Lambda$ is a maximal set of mutually chain-equivalent chain-recurrent points.
    Hence, $\Lambda$ is a chain node.

\medskip
{\bf There are no other chain-recurrent points.}
    We claim that $\cR_{\cC_\cE}=\{0,3\}\cup\Lambda$.
    The inclusion from right to left is already clear: the points $0$ and $3$ are fixed by both generators, while every point of $\Lambda$ has been shown to be chain-recurrent.

    Let now $z\in[0,3]\setminus\bigl(\Lambda\cup\{0,3\}\bigr)$ and suppose, by contradiction, that $z$ is chain-recurrent.
    By Proposition~\ref{prop: alternative weak node}, there exists $i\in\{1,2\}$ such that
\beqn
\label{eq: 5.15-11}
e_i(z)\CEto z.
\eeqn

    Assume first that $z\in \cO_\cE(e_i(z))$.
    Then there exists a word $I$ such that $\Phi_\cE^I(e_i(z))=z$.
    Thus, $z$ belongs to a positive-length periodic orbit of the semigroup action.

    Such an orbit cannot meet $(0,1)$, since both generators satisfy $e_j(x)<x$ for $x\in(0,1)$, $j=1,2$, and therefore every subsequent point of the orbit is strictly smaller.
    Similarly, it cannot meet $(2,3)$, since $e_j(x)>x$ for $x\in(2,3)$,  $j=1,2$.
    Consequently, the whole periodic orbit is contained in $[1,2]$.
    Repeating its periodic itinerary gives an infinite forward itinerary of $z$ entirely contained in $[1,2]$. 
    By the survivor-set characterization of $\Lambda$ above, this implies that $z\in\Lambda$, a contradiction.

    Assume now that $z\notin \cO_\cE(e_i(z))$.
    Since $\cC_\cE$ is an $\Omega$-stream, relation~\eqref{eq: 5.15-11} implies that there exists $p\in\Omega_\cE(e_i(z))$ such that
\beqn
\label{eq: 5.15-12}
p\CEto z.
\eeqn
    No point of the orbit of $z$ belongs to $\Lambda$. 
    Indeed, if $\Phi_\cE^J(z)\in\Lambda$ for some word $J$, then, working backward along $J$ and using
\[
e_j^{-1}(\Lambda)=g_j(\Lambda)\subset\Lambda,
\qquad j=1,2,
\]
    we would obtain $z\in\Lambda$. 
    In particular, $e_i(z)\notin\Lambda$.
    It follows from~\eqref{eq: 5.15-5} that $\Omega_\cE(e_i(z))\subset\{0,3\}$ and so $p\in\{0,3\}$.

    We now observe that neither endpoint is chain-upstream from any different point. 
    Fix $y>0$ and choose $r\in(0,y)$. 
    Since $e_j(r)<r$, $j=1,2$, the number
\[
\varepsilon_r
=
r-\max\{e_1(r),e_2(r)\}
\]
    is positive. 
    Every $\varepsilon$-chain starting at $0$, with $0<\varepsilon<\varepsilon_r$, remains in $[0,r]$. 
    Therefore
\beqn
\label{eq: 5.15-13}
0\not\CEto y.
\eeqn
    By symmetry,
\beqn
\label{eq: 5.15-14}
3\not\CEto y
\qquad\text{for every }y<3.
\eeqn
    Since $z\neq0,3$, relations~\eqref{eq: 5.15-13}--\eqref{eq: 5.15-14} contradict~\eqref{eq: 5.15-12}. 
    Thus, $z$ is not chain-recurrent.
    Hence, 
\[
\cR_{\cC_\cE}=\{0,3\}\cup\Lambda.
\]
    Since neither $0$ nor $3$ is chain-equivalent to any different point, the chain nodes are precisely
\[
\{0\},\qquad \Lambda,\qquad \{3\}.
\]

\medskip
{\bf The edges of the chain graph.}
    The points $1$ and $2$ belong to $\Lambda$. Moreover, $e_2^n(1)\to0$ and $e_1^n(2)\to3$.
    Since the chain stream contains the orbit relation and is closed, it follows that $1\CEto 0$ and $2\CEto 3$.
    Hence the chain graph contains the edges
\[
\Lambda\longrightarrow\{0\},
\qquad
\Lambda\longrightarrow\{3\}.
\]

    Relations~\eqref{eq: 5.15-13}--\eqref{eq: 5.15-14} show that neither $\{0\}$ nor $\{3\}$ has an outgoing edge to a different chain node. 
    Since the chain graph has only the three nodes identified above, there are no further edges.

Hence, $\Gamma_{\cC_\cE}$ is the following graph:
\[
\{0\}
\longleftarrow
\Lambda
\longrightarrow
\{3\}.
\]

\medskip
{\bf The orbitwise chain graph.}
No point of $\Lambda$ is orbitwise chain recurrent.

Indeed, let \(x\in\Lambda\). 
Since
\[
\Lambda=g_1(\Lambda)\,\dot\cup\,g_2(\Lambda),
\]
the point \(x\) belongs to exactly one of the two sets. 
If \(x\in g_1(\Lambda)\), then
\[
e_2(x)\notin\Lambda,
\]
because
\[
e_2^{-1}(\Lambda)\cap[1,2]=g_2(\Lambda).
\]
Likewise, if \(x\in g_2(\Lambda)\), then \(e_1(x)\notin\Lambda\).
By the maximality argument above, no point outside \(\Lambda\) is
chain-equivalent to a point of \(\Lambda\). 
Thus \(x\) is not chain-equivalent to every point of its orbit and so 
\[
x\notin \cR_{\cC_\cE}^{\mathrm{ow}}.
\]

Hence, the orbitwise chain graph consists of the two bottom nodes $\{0\},\{3\}$ and no edge.
In particular, it is not a connected graph.

\section{The Hutchinson map}
\label{sec: Hutch}

In~\cite{Hut81}, Hutchinson showed that it is possible to study the dynamics of a contractive IFS $\cF$ by studying the dynamics of the map that now takes his name on the space of all nonempty compact subspaces.
In this section, we go over this point of view and illustrate some of its relations with the IFS point of view illustrated in Sections~\ref{sec: global attractor} and~\ref{sec: streams}.

\medskip\noindent{\bf The space of all compact subspaces.}
We denote by {\BF$\KX$} the space of all non-empty compact subspaces of $X$.
We endow this space with the 
 {\bf Hausdorff metric}
\[
\text{\BF$\dH(A,B)$}=\max\Bigl\{\sup_{a\in A} d(a,B),\ \sup_{b\in B} d(b,A)\Bigr\},\quad A,B\in\KX.
\]
When $(X,d)$ is compact/locally compact/separable/metrizable, so is $(\KX,d_H)$.
This space is often referred in literature as the {\bf hyperspace} of $X$.

\medskip\noindent
{\bf The Hutchinson map.}
By {\bf Hutchinson map} of $F$ we mean the map $H_\cF:\KX\to\KX$ defined by
\[
\text{\BF$H_\cF(K)$}=\bigcup_{i=1}^m f_i(K),\quad K\in\KX.
\]
We recall the following important elementary property of $H_\cF$:
\[
H_\cF(A\cup B)=H_\cF(A)\cup H_\cF(B),\quad A,B\in\KX.
\]


\medskip\noindent
{\BF Trapping regions and global attractor of $H_\cF$.}
The map $H_\cF$ defines a discrete-time semiflow on $\KX$.
The whole theory developed in the previous sections applies therefore to $H_\cF$, with the advantage that, for $m=1$, plain and orbitwise chain nodes coincide and so do also invariance and $H$-invariance.
We list below the results relevant for this section, which were also proven in the semiflows setting by James Yorke and the present author in~\cite{DLY25,DLY26a}. 

\begin{thmX}[Yorke and De Leo, 2025]
    \label{thm:B}
    Let $\Phi$ be a discrete semiflow with compact dynamics on $X$.
    Then:
    \begin{enumerate}
        \item The global attractor of $\Phi$ is the maximal invariant compact subset of $X$ that attracts each compact subset of $X$ under $\Phi$.
        \item If the smallest stream of $\Phi$ has only countably many nodes, then its chain stream is its smallest stream.
        \item The chain-recurrent set of $\Phi$ and all of its nodes are compact invariant subsets of its global attractor $\cG_\Phi$.
        \item If $\cG_\Phi$ is connected, the chain graph of $\Phi$ is connected.
    \end{enumerate}
\end{thmX}

We leave to the reader the verification of the elementary identity 
\beqn
\label{lemma: H^n}
H_\cF^n(A) = \bigcup_{|I|=n} \P(A)
\eeqn
for every $A \in \KX$ and $n \ge 1$.









\begin{lemma}
\label{lemma: F has tr reg => H has tr reg}
Assume that $\cF$ has a compact global trapping region $Q \subset X$.
Then $\cK(Q)$ is a compact global trapping region for the Hutchinson map $H_\cF$.
\end{lemma}

\begin{proof}
Since $Q$ is compact, the hyperspace $\cK(Q)$ of nonempty compact subsets
of $Q$ is compact in $\KX$.

We first prove that $\cK(Q)$ is forward-invariant under $H_\cF$.
Let $A \in \cK(Q)$. Then $A \subset Q$. Since $Q$ is forward-invariant under
$F$, we have $f_i(Q) \subset Q$ for every $i=1,\dots,m$. Hence
\[
H_\cF(A)
  = \bigcup_{i=1}^m f_i(A)
  \subset \bigcup_{i=1}^m f_i(Q)
  \subset Q.
\]
Therefore $H_\cF(A) \in \cK(Q)$ and $\cK(Q)$ is forward-invariant.

We now prove that $\cK(Q)$ absorbs every compact subset of $\KX$.
Let $\mathcal C \subset \KX$ be compact and set
\[
U = \bigcup_{A \in \mathcal C} A .
\]
Since $\mathcal C$ is compact in $\KX$, the set $U$ is compact in $X$ and so, since $Q$ is a global trapping region, there exists $\tau \ge 0$ such that
\[
\P(U) \subset Q
\qquad
\text{for every } I \in \cI^m \text{ with } |I| \ge \tau .
\]

Let $A \in \mathcal C$. Since $A \subset U$, for every multiindex $I$
with $|I| \ge \tau$ we have
\[
\P(A) \subset \P(U) \subset Q .
\]

Using the identity
\[
H_\cF^n(A) = \bigcup_{|I|=n} \P(A),
\]
it follows that for every $n \ge \tau$
\[
H_\cF^n(A) \subset Q.
\]

Hence $H_\cF^n(A) \in \cK(Q)$ for every $A \in \mathcal C$ and every
$n \ge \tau$, that is,
\[
H_\cF^n(\mathcal C) \subset \cK(Q)
\qquad
\text{for every } n \ge \tau .
\]

Therefore $\cK(Q)$ absorbs every compact subset of $\KX$.
Thus $\cK(Q)$ is a compact global trapping region for $H_\cF$.
\end{proof}

\begin{lemma}
\label{lemma: H has tr reg => F has tr reg}
Assume that $H_\cF$ has a compact global trapping region $\cQ \subset \KX$.
Then
\[
Q = \bigcup_{A \in \cQ} A \subset X
\]
is a compact global trapping region for $\cF$.
\end{lemma}

\begin{proof}
We first prove that $Q$ is compact.

Let $\{x_n\}$ be a sequence in $Q$. For each $n$ there exists $A_n \in \cQ$
such that $x_n \in A_n$. Since $\cQ$ is compact in $\KX$, passing to a
subsequence if necessary, we may assume that
\[
A_n \to A \in \cQ
\]
in the Hausdorff metric. Since $x_n \in A_n$ for every $n$ and
$d_H(A_n,A)\to 0$, we have
\[
d(x_n,A) \to 0.
\]
For each $n$ choose $a_n \in A$ such that
\[
d(x_n,a_n) = d(x_n,A).
\]
Since $A$ is compact, passing to a further subsequence we may assume that
\[
a_n \to a \in A.
\]
Then
\[
d(x_n,a) \le d(x_n,a_n)+d(a_n,a)\to 0,
\]
so $x_n \to a \in A \subset Q$. Hence every sequence in $Q$ has a convergent
subsequence with limit in $Q$, and therefore $Q$ is compact.

We now prove that $Q$ is forward-invariant under $F$. Let $x \in Q$.
Then there exists $A \in \cQ$ such that $x \in A$. Since $\cQ$ is forward-invariant
under $H_\cF$, we have
\[
H_\cF(A) \in \cQ.
\]
For every $i=1,\dots,m$,
\[
f_i(x) \in \bigcup_{j=1}^m f_j(A) = H_\cF(A) \subset Q.
\]
Thus
\[
f_i(Q) \subset Q \qquad (i=1,\dots,m),
\]
so $Q$ is forward-invariant.

It remains to show that $Q$ absorbs every compact subset of $X$.
Let $B \subset X$ be compact. Then $\{B\}$ is a compact subset of $\KX$.
Since $\cQ$ is a global trapping region for $H_\cF$, there exists $\tau \ge 0$
such that
\[
H_\cF^n(B) \in \cQ
\qquad
\text{for every } n \ge \tau.
\]
By~\eqref{lemma: H^n}, 
it follows that for every multiindex $I$ with $|I|=n \ge \tau$,
\[
\P(B) \subset H_\cF^n(B) \subset Q.
\]
Hence
\[
\P(B) \subset Q
\qquad
\text{for every } I \in \cI^m \text{ with } |I| \ge \tau.
\]
Therefore $Q$ absorbs $B$.

Since $B$ was arbitrary, $Q$ absorbs every compact subset of $X$.
Thus $Q$ is a compact global trapping region for $\cF$.
\end{proof}

\begin{lemma}
\label{lemma: d_H}
For every $A \in \KX$ and every compact set $G \subset X$,
\[
d_H(A,\cK(G))=\sup_{x\in A} d(x,G).
\]
\end{lemma}

\begin{proof}
    Let $B\subset G$ and set $\delta=\sup_{x\in A} d(x,G)$.
    Then $d(x,B)\ge d(x,G)$ for every $x\in A$, so that 
\[
d_H(A,B)\ge\sup_{x\in A} d(x,B)\ge \delta.
\]
    Since this holds for every $B\in\cK(G)$, we get that $d_H(A,\cK(G))\ge \delta$.
    Now, set 
\[
B=\{y\in G : d(y,A)\le\delta\}.
\]
    Since $G$ is compact and nonempty, $B$ is a nonempty compact subset of $G$ and so $B\in \cK(G)$.
    Now, for every $x\in A$, by definition there exists $y\in G$ such that $d(x,y)\le \sup_{x\in A} d(x,G)$, and such a point $y$ necessarily belongs to $B$.
    Therefore, $\sup_{x\in A} d(x,B)\le \delta$.
    Moreover, since $B\subset G$ and each point of $B$ lies within distance at most $\delta$ of $A$, we have $\sup_{y\in B} d(y,A)\le \delta$ and so $d_H(A,B)\le \delta$ and $d_H(A,\cK(G))\le \delta$.    
    Therefore, ultimately, $d_H(A,\cK(G))=\delta$.
\end{proof}

\begin{lemma}
    \label{lemma: H -> K}
    Assume that $\cF$ has compact dynamics and let $\cC \subset\KX$ be compact. 
    Then
$$
\sup_{A\in\cC}d_H(H_\cF^n(A),\cK(\Gl))\to0.
$$
    In particular,
\[
\Omega_{H_\cF}(\cC) \subset \cK(\Gl).
\]
\end{lemma}

\begin{proof}
    Set $U=\bigcup_{A\in\cC}A$.
    The set $U$ is compact. 
    Since $\Gl$ attracts $U$, for every $\eps>0$ there exists $N\geq1$ such that $\P(U)\subset N_\eps(\Gl)$ whenever $|I|\geq N$.
    Hence, for $A\in \cC$ and $n\geq N$, 
    $$
    H_\cF^n(A) = \bigcup_{|I|=n}\P(A) \subset\bigcup_{|I|=n}\P(U)\subset N_\eps(\Gl).
    $$
    Then, by Lemma~\ref{lemma: d_H},
    $$
    d_H(H_\cF^n(A),\cK(\Gl))=e(H_\cF^n(A),\Gl)\leq\eps.
    $$
    The estimate is uniform in $A\in\cC$, so the assertion about the limit set follows.

    Since $\Gl$ is closed, this implies 
%
    that $\Omega_{H_\cF}(\cC) \subset \cK(\Gl)$.
\end{proof}



 \begin{theorem}
    \label{thm: F has comp dyn <=> H has comp dyn}
    $\cF$ has compact dynamics if and only if $H_\cF$ has compact dynamics.
    Moreover, if $\cH_\cF$ denotes the global attractor of $H_\cF$, then
    $$
    \Gl\in\cH_\cF \subset \cK(\Gl).
    $$
    In particular,
    $$
    \Gl=\bigcup_{A\in{\cH_\cF}}A\subset X.
    $$ 
    Finally, if $\cG_\cF$ is connected, then so is $\cH_\cF$. 
\end{theorem}
%


\begin{proof}
    Assume first that $\cF$ has compact dynamics. 
    By Theorem~\ref{thm: compact dynamics = compact gl tr reg}, there exists a compact global trapping region $Q\subset X$ for $\cF$. 
    By Lemma~\ref{lemma: F has tr reg => H has tr reg}, the hyperspace $\cK(Q)\subset\c\KX$ is a compact global trapping region for $H_\cF$. 
    Therefore, $H_\cF$ has compact dynamics.
    
    Now, since $\Gl$ is compact in $X$ and forward-invariant under $\cF$, it follows that $\cK(\Gl)$ is compact in $\c\KX$ and forward-invariant under $H_\cF$.
    Moreover, by Lemma~\ref{lemma: H -> K}, $\cK(\Gl)$ attracts every compact subset of $\KX$.
    Hence, by Proposition~\ref{prop: G_F minimal}, $\cH_\cF\subset\cK(\Gl)$.

    Since $\Gl$ is $H$-invariant, $\Gl$ is a fixed point of the map $H_\cF$ and, since $\mathcal \cH_\cF$ attracts the singleton $\{\Gl\}$,
\[
d_H\bigl(H_\cF^n(\Gl),\mathcal \cH_\cF\bigr)
=
d_H\bigl(\Gl,\mathcal \cH_\cF\bigr)
\longrightarrow0.
\]
    As $\mathcal \cH_\cF$ is closed, this gives $\Gl\in\mathcal \cH_\cF$.

    \medskip
    Conversely, assume that $H_\cF$ has compact dynamics. By Theorem~\ref{thm: compact dynamics = compact gl tr reg} in the semiflow setting, there exists a compact global trapping region $\cC\subset \KX$ for $H_\cF$. 
    By Lemma~\ref{lemma: H has tr reg => F has tr reg}, the set
\[
Q=\bigcup_{A\in \cC} A
\]
    is a compact global trapping region for $\cF$. 
    Hence, again by Theorem~\ref{thm: compact dynamics = compact gl tr reg}, $\cF$ has compact dynamics.

Finally, since $\Gl\in\mathcal \cH_\cF
\subset\mathcal K(\Gl)$, we have
\[
\Gl
\subset
\bigcup_{A\in\mathcal \cH_\cF}A
\subset
\Gl.
\]
so that 
\[
\Gl=\bigcup_{A\in\mathcal \cH_\cF}A.
\]
    Finally, assume that $\Gl$ is connected.
    Then $\cK(\Gl)$ is connected and so is
    $$
    \cH_\cF=\bigcap_{n\geq0}H^n_\cF(\cK(\Gl))
    $$
    since the sets $H^n_\cF(\cK(\Gl))$ are nested compact connected sets. 
\end{proof}

Notice that, being the global attractor of a semiflow, $\cH_\cF$ is invariant under $H_\cF$.

 \medskip\noindent
{\BF The chain stream of $H_\cF$.}
Let $\eps>0$. 
An $\eps$-chain from $K$ to $L$ is a finite sequence
\[
K=K_0,K_1,\dots,K_n=L
\]
such that
\[
\dH\bigl(H_\cF(K_j),K_{j+1}\bigr)<\eps
\quad\text{for all }j=0,\dots,n-1.
\]
We write that $K\geq L$ if, for every $\eps>0$, there is an $\eps$-chain from $K$ to $L$.
In this case, we say that $L$ is $H_\cF$-downstream from $K$.
We set
\[
\cC_{H_\cF} = \{(K,L)\in \KX\times \KX:\ K\Hto L\}.
\]
This relation is closed and transitive for any semiflow~\cite{DLY25,DLY26a} and we call it the {\bf chain stream} of $H_\cF$.


\medskip\noindent
{\bf Contractive IFSs.}
We recall the following seminal result of Hutchinson.
\begin{thmX}[Hutchinson, 1981~\cite{Hut81}]
\label{thm: Hutch}
Assume that $\cF$ is contractive and $X$ is complete.
Then $H_\cF$ is a contraction in $\KX$ and so, in particular, it has a unique fixed point $K$ and every point in $\KX$ asymptotes to $K$ under $H_\cF$. 
\end{thmX}
Below we rephrase Hutchinson's result
within our framework and then we extend it using the tools we developed so far.
%
\begin{corollary}
    Assume that $\cF$ is contractive and has compact dynamics.
    Then the global attractor of $H_\cF$ is a fixed point singleton $\{K\}$, which is also the unique node of $H_\cF$.
\end{corollary}
%
%
\begin{proposition}
\label{prop: R_C=K}
Assume that $\cF$ is contractive and let $K\in\KX$ be the unique fixed point of $H_\cF$.
Then $\Gl=K\subset X$ and $K$ is also the only chain node and orbitwise chain node of $\cF$, namely $\cR_{\cC_{\cF}}=\cR^{ow}_{\cC_{\cF}}=K$ and all points of $K$ are chain-equivalent to each other.
\end{proposition}

\begin{proof}
Since $\cF$ is contractive, there exists a $\lambda\in(0,1)$ such that
\[
d(f_i(u),f_i(v))\le \lambda\, d(u,v)
\qquad \text{ for all }u,v\in X\text{ and }  i=1,\dots,m.
\]
Hence, for every $I\in \cI^m$,
\[
d(\Phi_\cF^I(u),\Phi_\cF^I(v))\le \lambda^{|I|} d(u,v).
\]

Hence, for every compact $A\subset X$, 
$$
\sup_{x\in A}d(\P(x),K)\leq\lambda^{|I|}\sup_{x\in A}d(x,K).
$$
    Therefore $K$ attracts every compact subset of $X$. 
    Since it is $H$-invariant, it follows that $\Gl=K$.     
    
    Now, let $x,y\in K$ and $\varepsilon>0$ and choose $n\in \mathbb N$ such that $\lambda^n \operatorname{diam}(K)<\varepsilon$.
Since $K=H_\cF^n(K)$,
there exists an $I\in\cI^m$ of length $n$ such that $y\in \Phi_\cF^I(K)$.
Therefore
\[
d(\Phi_\cF^I(x),y)\le \operatorname{diam}\bigl(\Phi_\cF^I(K)\bigr)
\le \lambda^n \operatorname{diam}(K)<\varepsilon.
\]
Thus there is an $\varepsilon$-chain from $x$ to $y$. By symmetry, there is also an
$\varepsilon$-chain from $y$ to $x$. Since $\varepsilon>0$ was arbitrary, we conclude that
$x\Ceq y$. Hence $K$ is contained in a single chain node.

Moreover, no point of $X\setminus K$ is chain recurrent. 
To see this, set $L(x)=d(x,K)$.
Since each $f_i$ maps $K$ into itself, for every $x\in X$ and every $i=1,\dots,m$ we have that $L(f_i(x))\le \lambda L(x)$.
Indeed, for every $z\in K$ one has $f_i(z)\in K$, and therefore
\[
d(f_i(x),K)\le d(f_i(x),f_i(z))\le \lambda d(x,z).
\]
Taking the infimum over $z\in K$ gives the claim.

Let now $x\notin K$, and set $\delta=L(x)>0$ and choose an $\varepsilon>0$ such that $\varepsilon< (1-\lambda)\delta$.
Suppose by contradiction that there exists an $\varepsilon$-chain from $x$ to itself $x_0=x,\ x_1,\dots,x_n=x$.
By definition, for each $j=0,\dots,n-1$ there exists $i_j\in\{1,\dots,m\}$ such that $d(f_{i_j}(x_j),x_{j+1})<\varepsilon$.
Then
\[
L(x_{j+1})
\le d(x_{j+1},f_{i_j}(x_j)) + L(f_{i_j}(x_j))
< \varepsilon + \lambda L(x_j).
\]
We prove by induction that $L(x_j)<\delta$ for all $j\ge 1$.
For $j=0$ we have $L(x_0)=\delta$, so that
\[
L(x_1)<\varepsilon+\lambda\delta<\delta.
\]
If $L(x_j)<\delta$, then
\[
L(x_{j+1})<\varepsilon+\lambda L(x_j)
<\varepsilon+\lambda\delta<\delta.
\]
Thus indeed $L(x_j)<\delta$ for all $j\ge 1$.
In particular, $L(x_n)<\delta$.
But $x_n=x$, so $L(x_n)=L(x)=\delta$, a contradiction. Therefore there is no
$\varepsilon$-chain from $x$ to itself. It follows that $x\notin \cR_{\cC_\cF}$.

Hence, every chain recurrent point belongs to $K$. 
Moreover, since all points of $K$ are chain equivalent, $K$ is the unique node of the chain stream of $\cF$.
If \(x\in K\) and \(z\in \cO_\cF(x)\), then \(z\in K\), because \(K\) is forward-invariant. 
Since all points of \(K\) are chain-equivalent, \(x\Ceq z\). 
Hence, \(x\in \cR_{\cC_\cF}^{\mathrm{ow}}\).
\end{proof}

\medskip\noindent
{\BF Relations between $\cC_{\cF}$ and $\cC_{H_\cF}$.}
Recall that $\Gl$ is $H$-invariant under $\Phi_\cF$.
Hence, when $\cF$ has compact dynamics, the map $H_\cF$ has at least one fixed point: the global attractor $\Gl$.

Moreover, by Theorem~\ref{thm: N is H-invariant}, inside each orbitwise chain node of $\cF$ there is a nonempty $H$-invariant subset, namely a fixed point for the map $H_\cF$.
Hence, we have a map $$N\mapsto\Node_{\cC_{H_\cF}}(\Om(N))$$ 
from orbitwise nodes of $\cC_{\cF}$ to nodes of $\cC_{H_\cF}$.
We show in next proposition that this map is injective.
In particular, this shows that the chain graph of $H_\cF$ has at least as many nodes as the orbitwise chain graph of $\cF$.
\begin{proposition}
    \label{prop: N->Node(N) is injective}
    Assume that $\cF$ has compact dynamics and let $N_1,N_2$ be two distinct orbitwise chain nodes of $\cF$. 
    Then $\Node_{H_\cF}(\Om(N_1))\neq\Node_{H_\cF}(\Om(N_2))$.
\end{proposition}

\begin{proof}
    Let $N_1$ and $N_2$ be two orbitwise nodes of the chain stream of $\cF$, and assume that
\[
\Node_{{H_\cF}}(\Om(N_1))=\Node_{{H_\cF}}(\Om(N_2)).
\]
    We show that $N_1=N_2$.

    Since $\cF$ has compact dynamics, for each orbitwise chain node $N$ of $\cF$ the set $\Om(N)$ is compact and $H$-invariant. 
    Hence, regarded as points of $\KX$, both $\Om(N_1)$ and $\Om(N_2)$ belong to the chain recurrent set of $H_\cF$, and the assumption means that $\Om(N_1) \Hto \Om(N_2)$ and $\Om(N_2) \Hto \Om(N_1)$.

    We claim that if $K,L\in \KX$ satisfy $K \Hto L$, then, for every $y\in L$, there exists $x\in K$ such that $x \Cto y$.
    Indeed, fix an $\varepsilon>0$. 
    Since $K \Hto L$, there exist compact sets $K=K_0,K_1,\dots,K_r=L$ such that $d_H(H_\cF(K_j),K_{j+1})<\varepsilon$ for $j=0,\dots,r-1$.
    Fix $y_r=y\in L=K_r$. 
    Since $d_H(H_\cF(K_{r-1}),K_r)<\varepsilon$, there exist $y_{r-1}\in K_{r-1}$ and $i_{r-1}\in\{1,\dots,m\}$ such that $d\bigl(f_{i_{r-1}}(y_{r-1}),y_r\bigr)<\varepsilon$.
    Repeating the same argument backwards, we find points $y_j\in K_j$, $j=0,\dots,r$, and indices $i_j\in\{1,\dots,m\}$, $j=0,\dots,r-1$, such that
\[
d\bigl(f_{i_j}(y_j),y_{j+1}\bigr)<\varepsilon
\qquad\text{for }j=0,\dots,r-1.
\]
    Notice that this $\eps$-chain depends on $\eps$.
    Take any \(\varepsilon_n\downarrow0\). 
    For every \(n\), the preceding backtracking construction gives an \(\varepsilon_n\)-chain from some \(x_n\in K\) to \(y\). 
    By compactness of \(K\), after possibly passing to a subsequence, we may assume that \(x_n\to x\in K\). 
    Replacing the initial point \(x_n\) of these chains by \(x\), exactly as in the proof of Proposition~\ref{prop: C is a stream}, shows that for every \(\varepsilon>0\) there is an \(\varepsilon\)-chain from \(x\) to \(y\). 
    Hence, \(x\Cto y\).

    We now apply the claim with $K=\Om(N_1)$ and $L=\Om(N_2)$. 
    Fix $y\in \Om(N_2)$. 
    Since $\Om(N_1) \Hto \Om(N_2)$, there exists $x\in \Om(N_1)$ such that $x \Cto y$.
    Likewise, applying the claim with $K=\Om(N_2)$ and $L=\Om(N_1)$, for every $x'\in \Om(N_1)$ there exists $y'\in\Om(N_2)$ such that $y' \Cto x'$.

    Since $N_1$ and $N_2$ are nodes of the chain stream of $\cF$, any two points in the same node are mutually chain-related. 
    Therefore, from the existence of $x\in N_1$ and $y\in N_2$ with $x \Cto y$, it follows that $N_1 \Cto N_2$.
    Similarly, from the existence of $y'\in \Om(N_2)$ and $x'\in\Om(N_1)$ with $y' \Cto x'$, it follows that $N_2 \Cto N_1$.

    Hence, $N_1\Ceq N_2$, namely $N_1=N_2$.
\end{proof}

The following is the strongest general result we can prove on the relation between the chain graphs of $\cF$ and $H_\cF$.

\begin{theorem}
\label{thm: H one node => F one node}
    Assume that $\cF$ has compact dynamics.
    Then the chain graph of $H_\cF$ has at least as many nodes as the orbitwise chain graph of $\cF$.
    In particular, if $H_\cF$ has a unique chain node, then both the plain and orbitwise chain graphs of \(\cF\) have a unique node, and this node is \(\Gl\).
\end{theorem}
\begin{proof}


    Assume that \(H_\cF\) has a unique chain node. Proposition~\ref{prop: N->Node(N) is injective}, implies that \(\cF\) has a unique orbitwise chain node; denote it by \(N\).
    By Theorem~\ref{thm: N is H-invariant}, \(\Om(N)\) is a nonempty compact \(H\)-invariant subset of \(N\). 
    Hence, both \(\Om(N)\) and \(\Gl\) are fixed points of \(H_\cF\), and therefore they belong to the unique chain node of \(H_\cF\). 
    In particular, $\Om(N)\Hto \Gl$.

    Fix \(y\in \Gl\). 
    By the claim proved in Proposition~\ref{prop: N->Node(N) is injective}, there is \(x\in\Om(N)\subset N\) such that $x\Cto y$.

    Let \(z\in \cO_\cF(y)\). 
    By Lemma~\ref{lemma: Om}, there exists $q_z\in\Omega_F(z)\cap \cR_\cF^{\mathrm{ow}}$.
    By Proposition~\ref{prop: streams}, there is a $q_z\in \cR_{\cC_\cF}^{\mathrm{ow}}$ and therefore \(q_z\in N\). 
    Since \(q_z\in\Om(z)\), we have that $z\Cto q_z$ and, since \(q_z,x\in N\), we also have that $q_z\Cto x$.
    Consequently,
\[
z\Cto q_z\Cto x\Cto y\Cto z.
\]
    Hence, $x\Ceq y$ and $y\Ceq z$ for every $z\in \cO_\cF(y)$, so \(y\in N\subset  \cR_{\cC_\cF}^{\mathrm{ow}}\), which means that
    \(\Gl\subset N\). 
    
    Conversely,
\[
N\subset \cR_{\cC_\cF}\subset \Gl
\]
    by Proposition~\ref{prop: GF-trapping-for-CF} and so \(N=\Gl\).

    Finally, every plain chain-recurrent point belongs to \(\Gl=N\).
    Thus, the plain chain graph also has the single node \(\Gl\).
\end{proof}

The following elementary example shows that the inverse of Theorem~\ref{thm: H one node => F one node} does not hold, even for $m=1$.

\medskip\noindent
{\bf Example of a semiflow with a single chain node whose Hutchinson map has infinitely many chain nodes:  the odometer.}
Let \(X=\mathbb Z_2\) be the space of \(2\)-adic integers and let
\(f(x)=x+1\). The map \(f\) is minimal, hence its chain graph consists
of the single node \(X\).
The induced Hutchinson map on \(\mathcal K(\mathbb Z_2)\) is
\[
H(A)=A+1.
\]
For every \(n\ge1\), the compact-open subgroup $$A_n=2^n\mathbb Z_2$$ is a periodic point of \(H\), of period \(2^n\). 
Moreover, these periodic points lie in pairwise distinct chain nodes. 
Indeed, for \(k\ge n\), the canonical projection
\[
\pi_k:\mathbb Z_2\to\mathbb Z/2^k\mathbb Z
\]
induces a factor map on hyperspaces, and
\[
|\pi_k(A_n)|=2^{k-n}.
\]
If \(n\ne m\) and \(k\ge\max\{n,m\}\), the images of \(A_n\) and \(A_m\) have different cardinalities. Since translation on the finite hyperspace preserves cardinality, they belong to distinct chain nodes. 
Therefore, $H$ has infinitely many chain nodes.

\medskip\noindent
{\BF The chain graph of $H_\cE$.}
We conclude the article with a discussion on the graph of the chain graph of the Hutchinson map of the IFS $\cE$ on $X=[0,3]$ introduced below Theorem~\ref{thm: G conn => graph conn}.
Recall that the orbitwise chain graph of $E$ consists of the two isolated nodes $\{0\}$ and $\{3\}$, whereas its plain chain graph is
\[
\{0\}\longleftarrow\Lambda\longrightarrow\{3\},
\]
where $\Lambda\subset[1,2]$ is the survivor Cantor set.

We show below that the chain graph of $H_\cE$ is considerably richer than either the orbitwise or the plain chain graph of $E$.

\begin{proposition}
    \label{prop: six nodes}
    The chain graph of $H_\cE$ is connected and contains at least six distinct nodes.

    More precisely, define
\[
P_0:=\{0\},\qquad
P_3:=\{3\},\qquad
P_{03}:=\{0,3\},
\]
and
\[
K_L
:=
\overline{\bigcup_{n\geq0}H_\cE^n(\{1\})},
\qquad
K_R
:=
\overline{\bigcup_{n\geq0}H_\cE^n(\{2\})},
\qquad
K_B:=K_L\cup K_R.
\]
    Then all six compact sets
\[
P_0,\quad P_3,\quad P_{03},\quad K_L,\quad K_R,\quad K_B
\]
    are fixed points of $H_\cE$ and belong to pairwise distinct chain nodes.

    Writing
\[
\mathcal N_0=\Node_{C_{H_\cE}}(P_0),\qquad
\mathcal N_3=\Node_{C_{H_\cE}}(P_3),\qquad
\mathcal N_{03}=\Node_{C_{H_\cE}}(P_{03}),
\]
and
\[
\mathcal N_L=\Node_{C_{H_\cE}}(K_L),\qquad
\mathcal N_R=\Node_{C_{H_\cE}}(K_R),\qquad
\mathcal N_B=\Node_{C_{H_\cE}}(K_B),
\]
the graph contains the edges
\[
\mathcal N_L\longrightarrow\mathcal N_0,
\qquad
\mathcal N_R\longrightarrow\mathcal N_3,
\]
\[
\mathcal N_L\longrightarrow\mathcal N_B,
\qquad
\mathcal N_R\longrightarrow\mathcal N_B,
\]
and
\[
\mathcal N_B\longrightarrow\mathcal N_{03}.
\]
Consequently, by transitivity, it also contains
\[
\mathcal N_L\longrightarrow\mathcal N_{03},
\qquad
\mathcal N_R\longrightarrow\mathcal N_{03}.
\]
\end{proposition}

\begin{proof}
We divide the proof into several steps.

\medskip
\noindent
\emph{The six compact sets are fixed points.}
The sets $P_0$, $P_3$, and $P_{03}$ are fixed because both
generators fix $0$ and $3$:
\[
H_\cE(P_0)=P_0,\qquad
H_\cE(P_3)=P_3,\qquad
H_\cE(P_{03})=P_{03}.
\]

Since $e_1(1)=1$, we have
\[
\{1\}\subset H_\cE(\{1\}).
\]
By monotonicity of the Hutchinson map,
\[
H_\cE^n(\{1\})
\subset
H_\cE^{n+1}(\{1\})
\qquad(n\geq0).
\]
Therefore
\[
K_L
=
\overline{\bigcup_{n\geq0}H_\cE^n(\{1\})}
\]
is fixed by $H_\cE$.

Similarly, since $e_2(2)=2$, the sets
$H_\cE^n(\{2\})$ are increasing and
\[
H_\cE(K_R)=K_R.
\]
Finally, the Hutchinson map preserves unions, and hence
\[
H_\cE(K_B)
=
H_\cE(K_L\cup K_R)
=
H_\cE(K_L)\cup H_\cE(K_R)
=
K_L\cup K_R
=
K_B.
\]

Notice also that, if $u=e_2(1)\in(0,1)$, then
\[
K_L
=
\{0,1\}
\cup
\{\Phi_E^I(u):I\in\mathcal I^2\}.
\]
Thus $K_L$ is countable, $0$ is its only accumulation point, and
$1$ is isolated in $K_L$. By symmetry, $K_R$ is countable, $3$ is
its only accumulation point, and $2$ is isolated in $K_R$.

\medskip
\noindent
\emph{The three endpoint fixed points are distinct bottom nodes.}
Choose $r\in(0,1)$. On $(0,1)$ both generators move every point
strictly toward $0$. Consequently, the hyperspace neighborhood
\[
\mathcal U_0(r)
=
\{A\in\mathcal K([0,3]):A\subset[0,r]\}
\]
is forward-invariant for $H_\cE$, and every point of
$\mathcal U_0(r)$ converges to $P_0$ under iteration of
$H_\cE$.

Similarly,
\[
\mathcal U_3(r)
=
\{A\in\mathcal K([0,3]):A\subset[3-r,3]\}
\]
is a trapping neighborhood of $P_3$.

Finally, the set
\[
\begin{split}
\mathcal U_{03}(r)
=
\bigl\{
A\in\mathcal K([0,3]):\;&
A\subset[0,r]\cup[3-r,3],\\
&
A\cap[0,r]\neq\varnothing,\quad
A\cap[3-r,3]\neq\varnothing
\bigr\}
\end{split}
\]
is a trapping neighborhood of $P_{03}$, and every compact set in
this neighborhood converges to $\{0,3\}$.

It follows that $P_0$, $P_3$, and $P_{03}$ belong to three distinct
bottom chain nodes. In particular, none of them is chain-equivalent
to any of $K_L$, $K_R$, or $K_B$.

\medskip
\noindent
\emph{The sets $K_L$, $K_R$, and $K_B$ belong to distinct nodes.}
For $A\in\mathcal K([0,3])$, write
\[
m(A)=\min A,\qquad M(A)=\max A.
\]
The two generators are increasing and satisfy
\[
e_2(x)\leq e_1(x)
\qquad(x\in[0,3]).
\]
Hence
\beq
m(H_\cE(A))=e_2(m(A)),
\qquad
M(H_\cE(A))=e_1(M(A)).
\eeq

Choose $c\in(1,3)$. Since $e_1(c)>c$, sufficiently small
$H_\cE$-chains starting at a compact set $A$ with
$M(A)\geq c$ continue to satisfy $M(A)>c$. Since
\[
M(K_R)=M(K_B)=3,
\qquad
M(K_L)=1,
\]
there is no arbitrarily small chain from either $K_R$ or $K_B$ to
$K_L$.

Likewise, choose $d\in(0,2)$. Since $e_2(d)<d$, sufficiently small
chains starting at a compact set $A$ with $m(A)\leq d$ continue to
satisfy $m(A)<d$. Since
\[
m(K_B)=0,
\qquad
m(K_R)=2,
\]
there is no arbitrarily small chain from $K_B$ to $K_R$.

It follows that $K_L$, $K_R$, and $K_B$ are not pairwise
chain-equivalent. Together with the preceding step, this proves that
the six listed fixed points belong to six distinct chain nodes.

\medskip
\noindent
\emph{The edge $\mathcal N_L\to\mathcal N_0$.}
Since $1$ is isolated in $K_L$, for sufficiently small $\delta>0$
the compact set
\[
A_\delta
=
\bigl(K_L\setminus\{1\}\bigr)\cup\{1-\delta\}
\]
is arbitrarily close to $K_L$ in the Hausdorff metric.

Every point of $A_\delta$ lies strictly below $1$. Since both
generators move every point of $(0,1)$ toward $0$, uniformly on
compact subsets of $[0,1)$,
\[
H_\cE^n(A_\delta)\longrightarrow\{0\}.
\]
Thus an arbitrarily small first jump from $K_L$ to $A_\delta$,
followed by an exact $H_\cE$-orbit segment, gives arbitrarily
small chains from $K_L$ to $P_0$. Hence
\[
\mathcal N_L\longrightarrow\mathcal N_0.
\]

By reflection about $3/2$, the symmetric argument gives
\[
\mathcal N_R\longrightarrow\mathcal N_3.
\]

\medskip
\noindent
\emph{The edge $\mathcal N_B\to\mathcal N_{03}$.}
Perturb simultaneously the isolated points $1\in K_L$ and
$2\in K_R$. Namely, set
\[
B_\delta
=
\bigl(K_L\setminus\{1\}\bigr)\cup\{1-\delta\}
\cup
\bigl(K_R\setminus\{2\}\bigr)\cup\{2+\delta\}.
\]
Then
\[
B_\delta\longrightarrow K_B
\qquad\text{as }\delta\downarrow0.
\]
The left component of $B_\delta$ lies strictly below $1$ and hence
converges uniformly to $0$, while its right component lies strictly
above $2$ and converges uniformly to $3$. Therefore
\[
H_\cE^n(B_\delta)\longrightarrow\{0,3\}.
\]
This gives
\[
\mathcal N_B\longrightarrow\mathcal N_{03}.
\]

\medskip
\noindent
\emph{The edge $\mathcal N_L\to\mathcal N_B$.}
For every $n\geq1$, let $x_n\in(1,2)$ be the unique point satisfying
\[
e_1^n(x_n)=2
\]
and such that
\[
e_1^j(x_n)\in[1,2]
\qquad(0\leq j\leq n).
\]
Equivalently, $x_n$ is obtained by applying $n$ times the inverse
branch of $e_1$ fixing $1$. Hence
\[
x_n\longrightarrow1.
\]

Set
\[
A_n=K_L\cup\{x_n\}.
\]
Then
\beq
d_H(A_n,K_L)\longrightarrow0.
\eeq

Along the unique itinerary that remains in $[1,2]$, the point $x_n$ reaches $2$ after $n$ iterates. Any itinerary that departs from this
route before reaching $2$ enters $(0,1)$ and subsequently converges uniformly to $0$. 
After the point $2$ is reached, its full Hutchinson orbit generates $K_R$. 
Consequently,
\beqn
\label{eq: 6.8-3}
H_\cE^k(A_n)\longrightarrow K_L\cup K_R=K_B
\qquad\text{as }k\longrightarrow\infty.
\eeqn

Given $\varepsilon>0$, choose $n$ so large that $d_H(A_n,K_L)<\varepsilon$. 
By~\eqref{eq: 6.8-3}, an exact forward orbit of $A_n$ eventually enters the $\varepsilon$-neighborhood of $K_B$. 
Since $H_\cE(K_L)=K_L$, the sequence consisting of the first $\varepsilon$-jump
\[
K_L\longmapsto A_n
\]
followed by this exact orbit segment gives an $\varepsilon$-chain from $K_L$ to $K_B$. 
Hence
\[
\mathcal N_L\longrightarrow\mathcal N_B.
\]

The symmetric construction, using inverse iterates of $e_2$ that
converge to $2$, gives
\[
\mathcal N_R\longrightarrow\mathcal N_B.
\]

Transitivity now gives
\[
\mathcal N_L\longrightarrow\mathcal N_{03},
\qquad
\mathcal N_R\longrightarrow\mathcal N_{03}.
\]

\medskip
\noindent
\emph{Connectedness of the full chain graph.}
The hyperspace $\mathcal K([0,3])$ is connected; in fact, it is
contractible through the homotopy
\[
A\longmapsto (1-t)A,
\qquad 0\leq t\leq1.
\]
Since $H_\cE$ is continuous, every set
\[
H_\cE^n\bigl(\mathcal K([0,3])\bigr)
\]
is compact and connected. These sets form a decreasing sequence, and
the global attractor of $H_\cE$ is
\[
\cH_\cE
=
\bigcap_{n\geq0}
H_\cE^n\bigl(\mathcal K([0,3])\bigr).
\]
It is therefore connected. The connectedness theorem for the chain
graph of a discrete semiflow now implies that the full chain graph of
$H_\cE$ is connected.
\end{proof}

The six nodes above contain the following connected subgraph:
\[
\begin{array}{ccccc}
\mathcal N_0
&
\longleftarrow
&
\mathcal N_L
&
\longrightarrow
&
\mathcal N_B
\\[1mm]
&&&&\downarrow
\\[-1mm]
&&&&\mathcal N_{03}
\\[2mm]
\mathcal N_3
&
\longleftarrow
&
\mathcal N_R
&
\longrightarrow
&
\mathcal N_B .
\end{array}
\]
Equivalently,
\[
\mathcal N_0
\longleftarrow
\mathcal N_L
\longrightarrow
\mathcal N_B
\longleftarrow
\mathcal N_R
\longrightarrow
\mathcal N_3,
\qquad
\mathcal N_B\longrightarrow\mathcal N_{03}.
\]

\section{Old and new results from literature}
\label{sec: related}
In this section we discuss related literature and we import from it several results that are relevant for this work.

\subsection{Barnsley and Vince's Conley attractors for IFSs}
\label{sec: BV}
The closest antecedent to the present work is the theory of Conley attractors for IFSs developed by Barnsley and Vince in \cite{BV13}, building on the
$\varepsilon$-chains introduced in \cite{BV11} and adopted here.

\medskip\noindent
{\bf Pairs attractor-repeller.}
Barnsley and Vince work on a compact metric space $X$ and consider an IFS
$\F=\{f_1,\dots,f_m\}$ of continuous invertible maps. 
Notice that each such attractor is $H$-invariant.
Barnsley and Vince call a compact set $A$ an {\bf attractor} of $\F$ if it has an open neighbourhood $U$ such that $H_\F^n(U)\to A$ in the Hausdorff metric and define the {\BF basin $B(A)$} of $A$ the largest open neighborhood of $A$ that converges to $A$ in the Hausdorff metric.
They prove the following characterization for the basin:
\[
B(A)
=
\left\{
x :
\lim_{n\to\infty}
e\!\left(H_{\mathcal F}^{n}(\{x\}),A\right)=0
\right\}
=
\left\{
x :
\omega(\{x\})\subset A
\right\},
\]
where 
$$
\omega(\{x\}) = \bigcap_{K\geq1}\bigcup_{k\geq K}H^k_\cF({x}).
$$
The reader can verify that, in $\Gl$, $\omega({x})$ coincides with our $\Om(x)$.
Hence, we get the characterization
$$
B(A)
=
\left\{
x\in\Gl :
\Omega_{\cF}(x)\subset A
\right\}.
$$
Then, they associate to each attractor $A$ a {\BF dual repeller $A^*$}$=X\setminus B(A)$ and  characterize the chain recurrent set of $\F$ in terms of the resulting attractor--repeller pairs as follows.
\begin{thmX}[Barnsley and Vince 2013~\cite{BV13}, Theorem~5.5]
    \label{thm: BV}
    Let $\cF$ be an IFS with invertible generators over a compact metric space $X$ and denote by $\mathscr A$ the set of its attractors.
    Then 
    $$
    \cR_{\cC_\cF}=\bigcap_{A\in\mathscr A}(A\cup A^*).
    $$
\end{thmX}
\paragraph{Example: pairs attractor-repeller in the IFS $\cE$.}
Consider the IFS $\cE=\{e_1,e_2\}$ on $X=[0,3]$ introduced below Theorem~\ref{thm: G conn => graph conn}.
The reader can verify that the only attractors of $\cE$ in the sense of Barnsley and Vince are \{0\}, \{3\}, \{0,3\} and $X$. 
Now, we evaluate for each of them the repeller.
\begin{itemize}
    \item [\BF$\{0\}$] For $x\in[0,1)$, both generators move $x$ toward 0, uniformly with respect to all choices of generators. Thus, $B(\{0\})=[0,1)$ and $\{0\}^*=[1,3]$.
    \item [\BF$\{3\}$] Similarly,  $B(\{3\})=(2,3]$ and $\{3\}^*=[0,2]$.
    \item [\BF$\{0,3\}$] Every point outside of the Cantor set $\Lambda$ asymptotes to either 0 or 3, namely $\Om(x)\subset\{0,3\}$ for every $x\in X\setminus\Lambda$.
    Conversely, if $x\in\Lambda$, there is some trajectory that asymptotes to some point of $\Lambda$, so $\Om(x)\not\subset\{0,3\}$ and so $x\not\in B(\{0,3\})$.
    Hence, $B(\{0,3\})=X\setminus\Lambda$ and $\{0,3\}^*=\Lambda$.
    \item [\BF$X$] Since $X$ is $H$-invariant, $B(X)=X$ and $X^*=\varnothing$.
\end{itemize}
As an application of Theorem~\ref{thm: BV}, we obtain again that 
$$
\cR_{\cC_\cE}=(\{0\}\cup[1,3])\cap ([0,2]\cup\{3\})\cap(\{0,3\}\cup\Lambda)=\{0,3\}\cup\Lambda.
$$
Barnsley and Vince do not consider the sorting in nodes of the chain-recurrent set, which instead was studied by Akin and Akin and Wiseman, as we will illustrate in Subsection~\ref{subsec: Akin Wiseman}.
\paragraph{Attractor blocks.}
In the IFS setting, Barnsley and Vince also establish the existence of attractor blocks for Conley attractors.
In our notation, an {\bf attractor block} for an attractor $A$ can be defined as a trapping region $Q$ such that 
$$
A=\bigcap_{n\geq1}H^n_\cF(Q)\text{ and }H_\cF(Q)\subset int(Q).
$$
Hence, attractor blocks are a {\em local} (and stronger) analogue of our global trapping regions, whose attractor is the global attractor $\Gl$, since we showed that, when an IFS $\cF$ has compact dynamics, there are {\em global} trapping regions $Q$ such that $
\Gl=\cap_{n\geq1}H^n_\cF(Q).
$
%

\paragraph{Barnsley-Vince chain recurrence.}
The notion of chain recurrence introduced by Barnsley and Vince in~\cite[Definition~5.4]{BV13} coincides with the plain chain recurrence of Section~\ref{sec: chains}. 
Indeed, they call $x$ chain recurrent when for every $\varepsilon>0$ there is an $\varepsilon$-chain of positive length from $x$ to itself, and Proposition~\ref{prop: alternative weak node} identifies that condition with ours: a positive-length loop either has length $1$ for arbitrarily small $\varepsilon$, forcing $f_i(x)=x$ for some generator, or has length at least $2$, in which case $x$ is chain-equivalent to a point distinct from itself.

Notice that the requirement that these loops have \emph{positive} length is not cosmetic.
The relation $\C_{\F,d,\eta}$ was defined admitting chains of length $0$, since reflexivity is needed for it to be a quasi-order, so ``there is an $\varepsilon$-chain from $x$ to $x$'' holds vacuously for every $x\in X$. 
The loops of length exactly $1$ are the reason why Proposition~\ref{prop: alternative weak node}(3) lists fixed points of a single generator as a separate alternative: a length-$1$ loop of arbitrarily small error forces an exact fixed point rather than a nontrivial chain equivalence.

Although~\cite{BV13} assumes $X$ compact, Theorem~\ref{prop: C = C glob} makes its results available here whenever each generator restricts to a homeomorphism of $\Gl$ onto itself, so that $F|_{\Gl}$ is an invertible IFS on a compact metric space. 
Note that $H$-invariance of $\Gl$ alone is not enough: it gives $\bigcup_i f_i(\Gl)=\Gl$, not $f_i(\Gl)=\Gl$ for each $i$.
Without invertibility one may use Theorem~\ref{thm:akin} instead, which
requires only compactness.
The attractor--repeller characterization of~\cite[Thm.~5.5]{BV13} applies verbatim to $\F|_{\Gl}$ and, through Theorem~\ref{prop: C = C glob}, describes $\cR_{\C_\F}$. 
One caveat is that the attractors of \cite{BV13} are defined through the asymptotics of $H_\F$ on neighbourhoods, so the dual repeller of an attractor is a subset of $\Gl$ rather than of $X$; the decomposition obtained this way is a decomposition of the global attractor.

Our framework departs from \cite{BV13} in four respects. 
First, we do not assume $X$ compact, and Theorems~\ref{thm: Q => G} and~\ref{thm: compact dynamics = compact gl tr reg} identify compact dynamics with the existence of a compact global trapping
region. 
Second, we consider arbitrary streams and not only the chain stream:
Proposition~\ref{prop: A_F via Lyap functs} identifies the smallest stream through Lyapunov functions and Theorem~\ref{thm: Omega-stream} gives conditions under which the chain stream is the smallest one, so that the chain relation is not merely a convenient closure of the orbit relation but the canonical one. 
Third, we encode the resulting decomposition in a directed graph and study its connectedness (Theorems~\ref{thm: weak chain graph is connected} and~\ref{thm: G conn => graph conn}). 
Fourth, and most importantly for $m>1$, we distinguish plain from orbitwise chain recurrence.
The example of Section~\ref{sec: chains} shows that the distinction is not vacuous: there $\cR_{\C_\cE}$ contains a Cantor set $\Lambda$ no point of which is orbitwise chain recurrent. 
Since the nodes carrying the dynamics of the IFS as a semigroup action are the orbitwise ones, this distinction is what makes the graph a faithful record of the qualitative dynamics.

The chaos game results of Barnsley and Le\'sniak \cite{BL14} concern the same class of general IFSs and the same $\varepsilon$-chain apparatus, from the complementary point of view of which orbits realize the attractor; we do not pursue that direction here.

\subsection{Generalized recurrence and Lyapunov functions}

The smallest stream $\cA_\F$ of Section~\ref{sec: streams} is, by
Proposition~\ref{prop: A_F via Lyap functs}, the closed quasi-order generated by the continuous
Lyapunov functions of $\F$. 
For $m=1$ this object is classical: it is the relation introduced by Auslander in \cite{Aus63}, and its recurrent set is Auslander's \emph{generalized recurrent set}. 
The identification is immediate from the definitions.

\begin{proposition}
\label{prop:AFgenrec}
    Let $\F$ be an IFS on $X$. 
    Then
\[
\cR_{\cA_\F}
=\bigl\{x\in X:\ L(f_i(x))=L(x)\ \text{for all }L\in\cL_\F,
\ \text{for some }i\in\{1,\dots,m\}\bigr\},
\]
\[
\cR^{ow}_{\cA_\F}
=\bigl\{x\in X:\ L(z)=L(x)\ \text{for all }L\in\cL_\F
\ \text{and all }z\in\cO_\F(x)\bigr\}.
\]
    In particular, for $m=1$ the set $\cR_{\cA_\F}=\cR^{ow}_{\cA_\F}$ is the generalized recurrent set of the map $f_1$ in the sense of \cite{Aus63}.
\end{proposition}

\begin{proof}
    By definition $x\in\cR_{\cA_\F}$ if and only if $f_i(x)\succcurlyeq_{\cA_\F}x$ for some generator $f_i$. 
    Since $\cA_\F=D_{\cL_\F}$ by Proposition~\ref{prop: A_F via Lyap functs}, this holds if and only if $L(f_i(x))\ge L(x)$ for every $L\in\cL_\F$. 
    The reverse inequality holds for every $L\in\cL_\F$ by the definition of a Lyapunov function, whence the equality. The orbitwise case is identical, using $\cO_\F\subset\cA_\F$.
\end{proof}

Auslander's relation and Conley's chain relation have been compared systematically in recent years. 
Working with homeomorphisms of metric spaces, Fathi and Pageault \cite{FP15} developed an Aubry--Mather-style approach in which the generalized recurrent set appears as an Aubry set, and related it to the \emph{strong chain recurrent set} $SCR_d(f)$ introduced by Easton \cite{Eas77}: here one requires, for every $\varepsilon>0$, a chain
whose jumps satisfy
\[
\sum_{n=0}^{k-1} d\bigl(f(x_n),x_{n+1}\bigr)<\varepsilon
\]
    rather than the uniform bound used in Section~\ref{sec: chains}. 
    Their main result identifies Auslander's set with the intersection of the strong chain recurrent sets taken over all compatible metrics,
\[
GR(f)=\bigcap_{d} SCR_d(f),
\]
    so that the metric dependence of $SCR_d$ is exactly what separates the two notions. 
    This circle of ideas has been developed further by Wiseman~\cite{Wis18}, by Bernardi and Florio \cite{BF19a}, who obtain a Conley-type decomposition for the strong chain recurrent set, and by Bernardi, Florio and Wiseman \cite{BFW20}, who study the explosions and the Lyapunov functions associated with the generalized recurrent set. 
    Akin and Wiseman~\cite{AW17} treat chain recurrence for closed relations on general spaces, and their results apply to $\RF$ as explained in the next subsection.

\paragraph{A remark on terminology.}
The distinction drawn in this article between plain and orbitwise chain
recurrence is not the distinction between $CR$ and $SCR_d$, and we have
deliberately avoided calling the orbitwise notion ``strong chain recurrence'',
even though it is the stronger of our two conditions. The two notions are incomparable in nature. 
Strong chain recurrence in the sense of Easton and Fathi--Pageault is a genuinely metric notion --- $SCR_d(f)$ depends on the choice of $d$ within a fixed topology, and that dependence is the content of the theorem quoted above --- and it is defined already for $m=1$, where it may be strictly smaller than $CR(f)$. 
Orbitwise chain recurrence, by contrast, is topological, since it is defined through $\C_\F$, which by Proposition~\ref{prop: C_F is topological} does not depend on the compatible metric, and it is vacuous for $m=1$, where $\cR^{ow}_{\C_\F}=\cR_{\C_\F}$. 
It measures the branching of the free-monoid action, not the size of the accumulated error.

\paragraph{Strictness of $\cA_\F\subset\C_\F$.}
Theorem~\ref{thm: Omega-stream} shows that $\cA_\F=\C_\F$ whenever $\cA_\F$ has at most countably many nodes. 
For $m=1$ the question of whether the inclusion can be strict is settled by the literature above: examples of homeomorphisms of compact metric spaces with $GR(f)\subsetneq CR(f)$ are known, and by Proposition~\ref{prop:AFgenrec} any such example is an IFS with $m=1$ for which $\cA_\F\subsetneq\C_\F$. By Theorem~\ref{thm: Omega-stream} the smallest stream of such  a system necessarily has uncountably many nodes, in accordance with Proposition~\ref{prop: Lyapunov levels and chain nodes}.
We do not know whether the phenomenon takes a genuinely different form for $m>1$, that is, whether there exists an IFS with $\cA_\F\subsetneq\C_\F$ whose restriction to each single generator satisfies $\cA_{\{f_i\}}=\C_{\{f_i\}}$.

\paragraph{Complete Lyapunov functions.}
Proposition~\ref{prop:strict-lyapunov} associates with every closed quasi-order $D$ on $X$ a bounded strict Lyapunov function $L_D$, which is constant on each $D$-equivalence class and  strictly decreasing along every strict $D$-relation.
For the chain stream of an IFS with compact dynamics, more is expected: by Theorem~\ref{prop: C = C glob} the chain relation is determined by the restriction of $\F$ to the compact set $\Gl$, and there $\C_\F$ is the chain relation of the closed relation $\RF|_{\Gl}$, so the Conley decomposition  should yield a complete Lyapunov function for $\C_\F$, namely one whose set of critical values is compact and nowhere dense and whose critical set is exactly $\cR_{\C_\F}$. 
Proposition~\ref{prop:strict-lyapunov} already gives constancy on chain nodes and strict decrease along the strict chain order for a Lyapunov function; what it does not provide is the complete separation and nowhere-dense critical-value property.
The existence of a complete Lyapunov function for IFSs on a metric space is a consequence of Theorem 3.12 by Akin in \cite{Akin93}, that grants the existence of a complete chain Lyapunov function for any closed relation on a compact metric space: its critical set is exactly the chain-recurrent set, it is constant on each chain component, takes distinct values on distinct components, and its critical values form a closed nowhere-dense set.
In our setting, this implies the following.
\begin{thmX}[Akin, 1993 \cite{Akin93}, Theorem 3.12]
    If \(\cF\) has compact dynamics, there exists a continuous \(L:\Gl\to[0,1]\) such that \(L\) is constant on every plain chain node, takes distinct values on distinct chain nodes, is strictly decreasing whenever one passes strictly downstream between nodes, and \(L(\cR_{\cC_\cF})\) is compact and nowhere dense.
\end{thmX}
An immediate corollary is that the space of plain chain nodes, with the quotient topology, is homeomorphic to a compact nowhere-dense subset of [0,1]. Hence it is compact, metrizable and zero-dimensional.

\subsection{Conley theory for closed relations}
\label{subsec: Akin Wiseman}

Systematic Conley-type theories for closed relations on compact spaces were developed by McGehee~\cite{McG92} and Akin~\cite{Akin93}, and subsequently by McGehee and Wiandt~\cite{MW06} and by Akin and Wiseman~\cite{AW17,AW19}. 
These theories are the closest structural antecedents of the present framework. 
McGehee and Wiandt, in particular, argue that closed relations provide a natural setting for Conley theory: they permit multivalued dynamics, restore the symmetry between forward and backward dynamics through the transpose relation, and turn pseudo-orbits into ordinary orbits of neighboring relations.



Their starting point is that an IFS $\F$ determines the closed relation
\[
\RF=\{(x,y)\in X\times X:\ y=f_i(x)\ \text{for some }i\},
\]
whose set-image operator is precisely the Hutchinson map:
\[
\RF(A)=H_\F(A),\qquad \RF^{\,n}(A)=H_\F^n(A)=\bigcup_{|I|=n}\Phi^I_\F(A).
\]

Under this identification the dictionary between the two languages is complete.
Indeed, all the qualitative constructions central to Sections~\ref{sec: IFS}--\ref{sec: Hutch} are determined by $\F$ only through $\RF$:
\begin{itemize}
\item $A$ is forward-invariant if and only if $\RF(A)\subset A$, and
      $H$-invariant if and only if $\RF(A)=A$; these are the confining and the
      invariant sets of $\RF$;
\item $\cO_\cF(x)=\bigcup_{n\ge0}\RF^{\,n}(x)$ and, by
      Proposition~\ref{prop: topological description of Om},
      $\Omega_\F(A)=\bigcap_{n\ge0}\overline{\bigcup_{k\ge n}\RF^{\,k}(A)}$;
\item a sequence $x_0,\dots,x_k$ is an $\varepsilon$-chain if and only if
      $d\bigl(x_{j+1},\RF(x_j)\bigr)<\varepsilon$ is realized by an element of
      $\RF(x_j)$ for every $j$, so $\C_\F$ is determined by $\RF$;
\item consequently the streams of $\F$, their recurrent sets, their nodes and
      their graphs, in both the plain and the orbitwise version, are invariants
      of $\RF$; the same holds for bitrajectories, which are the bi-infinite
      sequences with $b_{n+1}\in \RF(b_n)$.
\end{itemize}

Hence, the stream- and chain-theoretic constructions of Sections~\ref{sec: IFS}--\ref{sec: Hutch} and the structural results formulated solely in terms of them may be read throughout as results about the closed relation $\RF$, and the finitely generated presentation is used only for convenience, chiefly to obtain uniform continuity estimates on compact sets.

For a closed relation on a compact Hausdorff space, McGehee and Wiandt define its Conley relation in a set-theoretical way and their results imply the following.

McGehee and Wiandt define chain recurrence directly for a closed relation $R$: a point $x$ is chain recurrent if $x$ is periodic for every closed neighborhood of $R$ in $X\times X$. 
Equivalently, the chain-recurrent set is the fixed-point set of their Conley relation, which they denote by $R^\Omega$. 
Applied to our setting, their result writes as follows.

\begin{thmX}[McGehee and Wiandt, 2006~\cite{MW06}, Definition 9.22 and Theorem 9.23; see also~\cite{AW17}, Corollary 6.9]
    \label{thm: MW}
    Assume that $\cF$ is an IFS on a compact and metrizable space $X$.
    Then $\cR_{\C_\F}$ coincides with the chain-recurrent set of the closed relation $\RF$ in the sense of~\cite{MW06}.
\end{thmX}

The relation-theoretic formulation above shows that our basic  constructions -- streams, recurrent classes, nodes and graphs -- can also be defined for closed relations that do not arise from finitely many continuous maps. 
Theorem~\ref{thm: MW} shows that, for the particular relation $\RF$, for IFSs on compact metrizable spaces our metric chain recurrence agrees with the corresponding intrinsic closed-relation notion of McGehee and Wiandt. 
We do not pursue the general closed-relation extension here.

McGehee and Wiandt also develop the corresponding theory of attractor blocks for closed relations: every Conley attractor admits arbitrarily small closed attractor blocks, and conversely an attractor block determines an attractor. 
In the IFS setting this reduces to the attractor-block construction already discussed in Section~\ref{sec: BV}.

\paragraph{The Conley decomposition  determines the chain graph.}
Recall that Barnsley and Vince proved that, given an IFS $\cF$ with invertible generators over a compact metric space $X$, $\cR_{\cC_\cF}=\cap_{A}(A\cup A^*)$.
Akin's general theory shows that this decomposition holds for more general IFSs and that it does determine the chain nodes as follows.
\begin{thmX}[Akin and Wiseman, 2017~\cite{AW19}, Propositions~6.22,~6.23]
    \label{thm:akin}
    Assume that $\cF$ is an IFS on a  metrizable space $X$, let $d$ be a metric on $X$ compatible with its topology
    and denote by ${\mathscr A}_{\cF,d}$ the set of the attractors of $\cF$ (in the attractor-repeller pairs sense illustrated in Section~\ref{sec: BV}).
    In general, the chain stream may depend on the compatible metric $d$ we choose on $X$ and
    $$
    \cR_{\cC_{\cF,d}}=\bigcap_{A\in{\mathscr A}_{\cF,d}}(A\cup A^*).
    $$
    Moreover, for any $x,y\in\cR_{\cC_{\cF,d}}$,
    $$
     x\Cdto y \iff \text{ for each attractor }A\in{\mathscr A}_{\cF,d} ,\; x\in A\implies y\in A.
    $$
    Similarly, given two chain-equivalent sets $M,N\subset\cR_{\cC_{\cF,d}}$,
    $$
    M\Cdto N \iff \text{ for each attractor }A\in{\mathscr A}_{\cF,d} ,\; M\subset A\implies N\subset A.
    $$
    Equivalently, set 
    $$
    {\mathscr A}_{\cF,d}(x) = \{A\in{\mathscr A}_{\cF,d}\,:\,x\in A\}.
    $$
    Then, for any $x\in\cR_{\cC_{\cF,d}}$,
    $$
    \Node_{\cC_{\cF,d}}(x)=\{y\in \cR_{\cC_{\cF,d}}\,:\,{\mathscr A}_{\cF,d}(x)={\mathscr A}_{\cF,d}(y)\}.
    $$
\end{thmX}

We call ${\mathscr A}_{\cF,d}(x)$ the {\bf signature} of $x$.
Notice that this theorem was proven first by Akin in 1993 (see~\cite{Akin93}, Proposition~3.11) in the compact case, in which case the chain stream and the signature function do not depend on the metric $d$ but only on the topology of $X$.
Notice moreover that the signature is also well-defined for nodes, since all points in the same node have the same signature.
For instance, in case of our example IFS $\cE$ over the compact set $X=[0,3]$, we get that 
$$
{\mathscr A}_\cE(\{0\})=\{\{0\},\{0,3\},X\}, \;{\mathscr A}_\cE(\{3\})=\{\{3\},\{0,3\},X\}\,,\;{\mathscr A}_\cE(\Lambda)=\{X\},
$$
which gives edges from $\Lambda$ to $\{0\}$ and $\{3\}$ and no other edge.


\paragraph{Strong chains and the smallest stream.}
The results of Akin and Wiseman \cite{AW17} also give a chain-theoretic description of the smallest stream.
For a compatible metric $d$, let $A_d\mathscr R_{\mathcal F}$
denote the Aubry--Mather, or strong-chain, relation of
Akin--Wiseman: $(x,y)\in A_d\mathscr R_{\mathcal F}$ if, for every
$\varepsilon>0$, there is a finite chain
$x=x_0,\ldots,x_k=y$, with associated indices
$i_0,\ldots,i_{k-1}$, such that
\[
\sum_{j=0}^{k-1}
d\bigl(f_{i_j}(x_j),x_{j+1}\bigr)<\varepsilon.
\]
Since $\mathscr R_{\mathcal F}$ is compactly upper semicontinuous, Corollary~6.15 gives
\[
\mathcal G\mathscr R_{\mathcal F}
=
\mathcal A_{\mathcal U_M}\mathscr R_{\mathcal F}.
\]
Corollary~5.13 then gives
\[
\mathcal A_{\mathcal U_M}\mathscr R_{\mathcal F}
=
\bigcap_d\mathcal A_d\mathscr R_{\mathcal F}
=
\mathcal A_{d_*}\mathscr R_{\mathcal F}
\]
for some compatible metric $d_*$. Finally, Corollary~6.9 identifies $\mathcal A_d\mathscr R_{\mathcal F}$ with the strong-chain relation defined above.

Since our streams are reflexive, we set
\[
S_{\mathcal F,d}
=
\Delta_X\cup\mathcal A_d\mathscr R_{\mathcal F}
\]
and call it the \emph{strong-chain stream} associated with $d$.
Similarly, we have that
$$
\cA_\cF=\Delta_X\cup\mathcal G\mathscr R_{\mathcal F}
$$
Since our standing assumptions on $X$ imply that $X$ is paracompact and $\sigma$-compact, we can import from their theory the following result.
Notice that Akin and Wiseman do not state such result for IFSs, but it follows directly from their general theory of relations on uniform spaces.
%
\begin{thmX}[Akin-Wiseman, 2017~\cite{AW17}, Corollaries~5.13,~6.9 and~6.15]
    \label{thm: AW}
    Let $\mathcal F$ be an IFS on a separable, locally compact metrizable space $X$ and let $\mathcal M_{X}$ denote the set of its compatible distance functions. 
    Then
\[
\mathcal A_{\mathcal F}
=
\bigcap_{d\in\cM_X} S_{\mathcal F,d}.
\]
    Moreover, there exists a $d_*\in\cM_X$ such that
\[
\mathcal A_{\mathcal F}=S_{\mathcal F,d_*}.
\]
\end{thmX}

Thus, although the strong-chain relation generally depends on the metric, the smallest stream is obtained by intersecting the strong-chain streams over all compatible metrics and, in fact, is already realized by the strong-chain stream associated with one suitably chosen compatible metric.



 

\medskip
It should be clear to the reader by now that what distinguishes the present work from this research direction is not the underlying object but what is done with it.

\begin{enumerate}
    \item We separate plain from orbitwise recurrence. 
    Orbitwise chain recurrence of $x$ requires every point of $\bigcup_n \RF^{\,n}(x)$ to be chain-equivalent to $x$, a condition that is meaningful when the relation is genuinely multivalued along the forward dynamics
    and that is expressible in the closed-relation language, but which, to our knowledge, has not been studied there. 
    It is this notion, and not plain chain recurrence, whose nodes are forward-invariant (Proposition~\ref{prop: streams}) and carry a nonempty $H$-invariant core (Theorem~\ref{thm: N is H-invariant}).
    \item Our main interest is the graph canonically associated to a stream, in particular in its connectedness.
    
    \item We work on locally compact separable metric spaces under the hypothesis of compact dynamics, rather than assuming the phase space compact, and characterize that hypothesis through compact global trapping regions (Theorems~\ref{thm: Q => G} and~\ref{thm: compact dynamics = compact gl tr reg}).
    \item We compare the dynamics of $\F$ on $X$ with that of $H_\F$ on the hyperspace $\KX$. 
    Although $H_\F$ \emph{is} the set-image operator of $\RF$, the chain recurrence of the map $H_\F$ on $\KX$ can distinguish dynamical structures that are invisible to the chain recurrence of the relation $\RF$ on $X$:
    the odometer of Section~\ref{sec: Hutch} has a single chain node while its Hutchinson map has infinitely many, and Proposition~\ref{prop: six nodes} exhibits the same phenomenon for the IFS $\cE$. 
    Proposition~\ref{prop: N->Node(N) is injective} and Theorem~\ref{thm: H one node => F one node} quantify one half of this comparison.
\end{enumerate}

\newcommand{\Sig}{\Sigma^+_m}                 
\newcommand{\sh}{\sigma}                      
\newcommand{\sk}[1]{\widehat{#1}}             
\newcommand{\piX}{\pi_X}                      
\newcommand{\dSig}{d_\Sigma}                  
\newcommand{\GA}[1]{\cG_{#1}}                   
\newcommand{\Hut}[1]{H_{#1}}                  
\newcommand{\Ch}[1]{\mathcal{C}_{#1}}         
\newcommand{\Rec}[1]{\cR_{\Ch{#1}}}             
\newcommand{\Recow}[1]{R^{\mathrm{ow}}_{\Ch{#1}}}   
\newcommand{\Gr}[1]{\Gamma_{\Ch{#1}}}         
\newcommand{\Grow}[1]{\Gamma^{\mathrm{ow}}_{\Ch{#1}}} 
\newcommand{\Nd}[1]{\mathrm{Node}_{\Ch{#1}}}  
\newcommand{\dstr}[1]{\succcurlyeq_{\Ch{#1}}} 
\newcommand{\chq}[1]{\overset{\Ch{#1}}{=}}    
\newcommand{\Orb}[1]{O_{#1}}                  
\newcommand{\multi}{\mathcal{I}^m}            
\newcommand{\Fl}[2]{\Phi^{#1}_{#2}}           


\paragraph{A localization principle.}
Moving the focus from the phase space to the dynamical system allows to extend naturally all results reviewed in this subsection.
Each of those results above says that, if $X$ satisfies some conditions $A$, then every IFS $\cF$ on $X$ satisfies some property $B$.
We are rather interested in finding a set $A'$ of conditions on $\cF$ such that every $\cF$ satisfying $A'$ satisfies property $B$ regardless of the properties of the phase space $X$ -- to be more precise, we still do ask that the phase space be metrizable.


For an IFS $\cF$ with compact dynamics on a metrizable space $X$, under the additional hypothesis that $\Gl$ has a compact 
neighborhood, we recover the results of Theorem~\ref{prop: C = C glob}, namely, independently on the topological properties of $X$, compact-space results concerning the chain-recurrent set, its chain components, and their downstream order can be transferred immediately to such IFSs.
\begin{theorem}[Localization principle]
    \label{thm:localization}
    Let $\mathcal F$ be an IFS with compact dynamics on a metrizable space $X$ and assume that $\Gl$ has a compact neighborhood.
    Then
\[
\cC_\cF|_{\Gl\times \Gl}=\cC_{\cF|_{\Gl}}.
\]
    In particular, both the plain and orbitwise chain-recurrent sets, nodes, and graphs of \(\cF\) coincide with those of \(\cF|_{\Gl}\).
\end{theorem}


\begin{proof}
We first observe that the hypothesis provides a compact global trapping
region which is a neighborhood of $\Gl$.

    Indeed, since $\Gl$ is compact and has a locally compact neighborhood, there exists a compact set $K\subset X$ such that
\[
    \Gl\subset \operatorname{int}_X K.
\]
    Since $\Gl$ attracts $K$, there exists $N\geq 1$ such that $\Phi_\cF^I(K)\subset \operatorname{int}_X K$ for every $|I|\geq N$.
Set
\[
    Q=\bigcup_{|I|<N}\Phi_\cF^I(K).
\]
The union is finite, hence $Q$ is compact, and $K\subset Q$, so $Q$ is
a neighborhood of $\Gl$.

We claim that $Q$ is forward-invariant. Indeed, applying a generator
to a set $\Phi_\cF^I(K)$ with $|I|<N$ produces an image of $K$ by a word
of length $|I|+1$. If $|I|+1<N$, this image is one of the sets occurring
in the definition of $Q$; if $|I|+1=N$, it is contained in $K\subset Q$.
Thus
\[
    f_i(Q)\subset Q,\qquad i=1,\ldots,m.
\]
    Moreover, if $L\subset X$ is compact, then, since $\Gl$ attracts $L$ and $\operatorname{int}_XK$ is a neighborhood of $\Gl$, there exists $M$ such that $\Phi_\cF^I(L)\subset K\subset Q$ for every $|I|\geq M$.
    Hence $Q$ is a compact global trapping region.

    Choose $r>0$ such that
\[
    \overline{N_r(\Gl)}\subset \operatorname{int}_X Q.
\]
    The finite-block argument of Lemma~\ref{lemma: robust trapping region} now applies inside the compact set $Q$: for every $\rho\in(0,r)$ there exist $\eta\in(0,\rho)$ and $\delta_0>0$ such that every $\delta_0$-chain starting in $N_\eta(\Gl)$ remains in $N_\rho(\Gl)$.

    We first show that $\cR_{\cC_\cF}\subset \Gl$.
    Suppose that $x\notin \Gl$, set $a=d(x,\Gl)>0$ and choose $\rho\in(0,\min\{r,a/2\})$.
    Let $\eta$ and $\delta_0$ be as above. 
    Since $\Gl$ attracts $\{x\}$, there exists $N_0\geq1$ such that $\Phi_\cF^I(x)\in N_{\eta/2}(\Gl)$ whenever $|I|\geq N_0$.
    By continuity of the generators and finiteness of the set of words of length $N_0$, there exists $\delta_1>0$ such that every $\delta_1$-chain
\[
    x=x_0,x_1,\ldots,x_{N_0}
\]
satisfies
\[
    d\bigl(x_{N_0},\Phi_\cF^I(x)\bigr)<\eta/2,
\]
where $I$ is the word associated with the chain. Hence
$x_{N_0}\in N_\eta(\Gl)$.

Set $\delta=\min\{\delta_0,\delta_1\}$. Every $\delta$-chain starting
at $x$ and having length at least $N_0$ therefore enters
$N_\eta(\Gl)$ at time $N_0$ and subsequently remains in
$N_\rho(\Gl)$. Since $x\notin N_\rho(\Gl)$, no such chain can return
to $x$.

Suppose, by contradiction, that $x$ is chain-recurrent. 
By Proposition~\ref{prop: alternative weak node} there is a positive-length $\varepsilon$-chain from $x$ to itself for every $\varepsilon>0$. 
Taking $\varepsilon_n\downarrow0$, all sufficiently small such loops have length less than $N_0$. 
Passing to a subsequence, we may assume that they have the same positive length $k$ and the same associated word $I$. 
Since their errors tend to zero, continuity gives $\Phi_\cF^I(x)=x$.
Repeating the word $I$ produces words of arbitrarily large length sending $x$ to itself, contradicting the fact that $\Gl$ attracts $\{x\}$. 
Hence $\cR_{\cC_\cF}\subset \Gl$.

We now compare the chain relations on $\Gl$. 
The inclusion
\[
    \cC_{\cF|_Q}|_{\Gl\times \Gl}
    \subset C_F|_{\Gl\times \Gl}
\]
    is immediate. 
    Conversely, let $x,y\in \Gl$ and assume $x\succeq_{\cC_\cF}y$. 
    Fix $\varepsilon>0$ and choose $\rho\in(0,r)$. 
    Let $\eta,\delta_0>0$ be as above and take $\delta\in(0,\min\{\delta_0,\varepsilon\})$.
    There is a $\delta$-chain from $x$ to $y$. 
    Since $x\in \Gl\subset N_\eta(\Gl)$, the whole chain remains in $N_\rho(\Gl)\subset Q$.
    It is therefore an $\varepsilon$-chain for $F|_Q$. 
    Since $\varepsilon>0$ was arbitrary,
\beqn
\label{eq: 11-s}
    \cC_\cF|_{\Gl\times \Gl}
    =\cC_{\cF|_Q}|_{\Gl\times \Gl}.
\eeqn

    The space $Q$ is compact, $F|_Q$ has compact dynamics, and its global attractor is $\Gl$. 
    Hence Theorem~\ref{prop: C = C glob}, applied to $F|_Q$, gives
\[
    \cC_{\cF|_Q}|_{\Gl\times \Gl}=\cC_{\cF|_{\Gl}}.
\]
    Together with~\eqref{eq: 11-s}, this yields
\[
    \cC_\cF|_{\Gl\times \Gl}=\cC_{\cF|_{\Gl}}.
\]

    Since every chain-recurrent point of $F$ lies in $\Gl$, and $\Gl$ is forward-invariant, the equality of the two chain relations on $\Gl\times \Gl$ implies
\[
    \cR_{\cC_\cF}=R_{\cC_{\cF|_{\Gl}}}.
\]
    It also implies that the chain-equivalence relations on this common set coincide. 
    Hence the two systems have the same chain nodes, and the downstream relations between these nodes are the same. 
    Therefore
\[
    \Gamma_{\cC_\cF}=\Gamma_{\cC_{\cF|_{\Gl}}}.
\]

    Finally, the orbit of every point of $\Gl$ is contained in $\Gl$.
    Consequently, the same equality of chain relations also yields the corresponding statements for orbitwise chain recurrence, orbitwise chain nodes, and their downstream relations.
\end{proof}

With our angle, we obtain for instance the following extensions of Theorems~\ref{thm: MW} and~\ref{thm:akin}.
Let $\cF$ be an IFS with compact dynamics. 
We say that $\cF$ has a {\bf thick compact dynamics} if $\cF$ has compact dynamics and $\Gl$ has a compact neighborhood.
Notice that this includes every IFS with compact dynamics on a locally compact space.
%
\begin{thmF}
    Let $\cF$ be an IFS on a metrizable space $X$ and assume that $\cF$ has thick compact dynamics.
    Then $\cR_{\C_\F}$ coincides with the chain-recurrent set of the closed relation $\RF|_{\Gl}$ in the sense of~\cite{MW06}.
    In particular, $\cR_{\C_\F}$ is topological.
\end{thmF}
\begin{thmG}
    Let $\cF$ be an IFS on a metrizable space $X$ and assume that $\cF$ has thick compact dynamics.
    Then the chain stream is topological and, for every compatible metric $d$, the signatures of Theorem~\ref{thm:akin} determine the same chain nodes and the same downstream order.
\end{thmG}
What we did in the two theorems above is to allow phase spaces not locally compact by requiring in exchange that the global attractor of the IFS has a compact neighborhood.
This localization argument though does not provide an extension of Theorem~\ref{thm: AW}. 
Indeed, thick compact dynamics allows us to localize the chain relation, and hence its recurrent set, nodes and downstream order, to a compact neighborhood of $\Gl$. 
Theorem~\ref{thm: AW}, however, concerns the smallest stream and the strong-chain relations on the whole phase space. 
Although the theorem applies to the restriction of $\cF$ to any compact trapping neighborhood of $\Gl$, there is no corresponding localization result identifying these relations with their ambient counterparts. 

Such an extension, though, can be achieved using a different route, namely giving up compactness in exchange for uniform continuity.
More precisely, in~\cite{DLY26b}, in case of semiflows with compact dynamics on a metrizable space, we showed that, if the dynamics is uniformly continuous in a neighborhood of the global attractor, then the compact-space Conley theory can still be transferred to the ambient system. 
In case of a IFS $\cF$ with strong compact dynamics, we say that $\cF$ has {\bf strong compact dynamics} if there is a neighborhood $U$ of $\Gl$ such that $U$ is attracted by $\Gl$ and each generator $f_i$ is uniformly continuous on $U$.
This case would cover, for instance, IFSs whose generators are time-1 maps of semiflows of reaction-diffusion PDEs over a given Banach space (see~\cite{DLY26b} for a thorough discussion about this case when $m=1$).
For an IFS with strong compact dynamics, the results imported in Section~\ref{subsec: Akin Wiseman} remain available without assuming that $X$ be locally compact. 
For instance, we claim the following extension of Theorem H.
%
\begin{conH}
    Let $\mathcal F$ be an IFS with strong compact dynamics over a metrizable space $X$ and denote by $\cM_X$ the set of all distance functions on $X$ compatible with its topology.
    Then: 
    \begin{enumerate}
        \item its chain stream is topological, namely $\cC_{\cF,d_1}=\cC_{\cF,d_2}$ for any $d_1,d_2\in\cM_X$;
        \item $\displaystyle\mathcal A_{\mathcal F}=\bigcap_{d\in\cM_X} S_{\mathcal F,d}$;
    \item there exists a $d_*\in\cM_X$ such that $\displaystyle\mathcal A_{\mathcal F}=S_{\mathcal F,d_*}$.
    \end{enumerate}
\end{conH}
In case $m=1$, the conjecture is known to be trus and consists of Proposition 5.19 and Proposition 5.39 in~\cite{DLY26b}.
A full proof of the case $m>1$ requires the generalization of several results in the present article to IFSs with strong compact dynamics over metrizable spaces.
We do not pursue this extension to IFS here.




\subsection{Chain stream and skew products}
\label{sec:skew}

Some authors (see e.g.~\cite{Hom18}) study the IFS dynamics by considering the following map.
Set
\[
\text{\BF$\Sigma_m^+$}=\{1,\ldots,m\}^{\mathbb N_0}.
\]
This is called the {\bf one-sided shift space}, which is a metric space with respect to the distance
\[
d_\Sigma(\omega,\nu)
=
\begin{cases}
0, & \omega=\nu,\\[1mm]
2^{-k(\omega,\nu)}, &
k(\omega,\nu):=\min\{j\ge0:\omega_j\neq\nu_j\},
\end{cases}
\]
and call {\bf left shift} the map
\[
\text{\BF$\sigma$}:\Sigma_m^+\longrightarrow\Sigma_m^+,
\qquad
\sigma(\omega_0,\omega_1,\ldots)
   =(\omega_1,\omega_2,\ldots).
\]
Given an IFS $\cF$, we define the skew-product space
\[
\Sigma_m^+\times \Gl
\]
with the product metric
\[
d_{\times}\bigl((\omega,x),(\nu,y)\bigr)
=
\max\bigl\{
d_\Sigma(\omega,\nu),\,d(x,y)
\bigr\}.
\]
We call {\bf full one-sided step skew product} the map
\[
\widehat F:\Sigma_m^+\times \Gl
   \longrightarrow\Sigma_m^+\times \Gl
\]
defined by
\[
\widehat F(\omega,x)
   =\bigl(\sigma\omega,f_{\omega_0}(x)\bigr).
\]
Finally, we define the projection
\[
\text{\BF$\pi_X$}:\Sigma_m^+\times \Gl
   \longrightarrow \Gl,
\qquad
\pi_X(\omega,x)=x,
\]

Since the resulting system is a discrete-time semiflow on a compact space, our stream theory 
applies to it directly, and it is reasonable to ask whether the apparatus of Sections~\ref{sec: streams}--\ref{sec: Hutch} is needed at all. 
The purpose of this subsection is to answer that question. 
Theorem~\ref{thm:skew} below shows that the lift is exactly faithful to the \emph{plain} chain graph of~$\cF$; the two results that follow it show that this is all it is faithful to. 
The lift has no analogue retaining the IFS's orbitwise branching information,
and the hypothesis of the connectedness theorem for semiflows fails for it identically, for every IFS with $m\ge 2$. 
The structure that the present framework is built to record is precisely the structure that the reduction destroys.


\begin{theorem}
    \label{thm:skew}
    Assume that $\cF$ has compact dynamics. 
    Then $\piX$ maps $\cR_{\cC_{\sk{F}}}$ onto $\cR_{\cC_\cF}$ and induces a bijection between the chain nodes of $\sk{F}$ and the chain nodes of $\cF$. 
    More precisely, if $\sk{N},\sk{M}$ are chain nodes of $\sk{F}$, then
\[
  \sk{N}\CtoF\sk{M}
  \iff
  \piX(\sk{N})\Cto\piX(\sk{M}),
\]
    so that $\piX$ induces an isomorphism of directed graphs
\[
  \Gr{\sk{F}}\;\simeq\;\Gr{\cF}.
\]
    Moreover, if $N$ is a chain node of $\cF$, the corresponding node upstairs is
\[
  \sk{N}=\Rec{\sk{F}}\cap\bigl(\Sig\times N\bigr).
\]
\end{theorem}

\begin{proof}
By Theorem~\ref{prop: C = C glob}, the chain-recurrent set, the chain nodes and
the chain graph of $F$ coincide with those of $F|_{\Gl}$, so no generality
is lost in working over $\Gl$.

\medskip
\noindent\textbf{Step 1 (projection).}
If $(\omega^0,x_0),\dots,(\omega^n,x_n)$ is an $\varepsilon$-chain for
$\sk{F}$, then, the product metric being the maximum of the two coordinate
metrics, its $j$-th jump satisfies
$d\bigl(f_{\omega^j_0}(x_j),x_{j+1}\bigr)<\varepsilon$; hence
$x_0,\dots,x_n$ is an $\varepsilon$-chain for $\cF$ of the same length. In
particular positive-length chain loops project to positive-length chain loops,
so Proposition~\ref{prop: alternative weak node} gives
\begin{equation}\label{eq:skew-proj}
  \piX\bigl(\Rec{\sk{F}}\bigr)\subset\Rec{\cF}.
\end{equation}

\medskip
\noindent\textbf{Step 2 (lifting).}
Let $x\in\Rec{\cF}$ and let $\varepsilon_n\downarrow 0$. For each $n$ choose a
positive-length $\varepsilon_n$-chain loop
$x=x^{(n)}_0,\dots,x^{(n)}_{k_n}=x$ with associated word
$I_n=i^{(n)}_0\cdots i^{(n)}_{k_n-1}$, and let $\omega^{(n)}\in\Sig$ be the
periodic sequence obtained by repeating $I_n$. Since
$\sh^{k_n}\omega^{(n)}=\omega^{(n)}$, the sequence
\[
  \bigl(\omega^{(n)},x^{(n)}_0\bigr),
  \bigl(\sh\omega^{(n)},x^{(n)}_1\bigr),\dots,
  \bigl(\sh^{k_n}\omega^{(n)},x^{(n)}_{k_n}\bigr)
\]
is a positive-length $\varepsilon_n$-chain loop for $\sk{F}$ based at
$z_n=(\omega^{(n)},x)$: the symbolic coordinate evolves exactly and the fibre
error is that of the original loop. By compactness of $\Sig$ we may assume
$\omega^{(n)}\to\omega$; set $z=(\omega,x)$. Given $\varepsilon>0$, choose $n$
with $d(z_n,z)<\varepsilon/3$, $d\bigl(\sk{F}(z_n),\sk{F}(z)\bigr)<\varepsilon/3$
and $\varepsilon_n<\varepsilon/3$, and replace the initial and final
occurrences of $z_n$ in the above loop by $z$; the triangle inequality shows
that the result is a positive-length $\varepsilon$-chain loop based at $z$.
By Proposition~\ref{prop: alternative weak node}, $z\in\Rec{\sk{F}}$. Together with
\eqref{eq:skew-proj} this gives
\begin{equation}\label{eq:skew-onto}
  \piX\bigl(\Rec{\sk{F}}\bigr)=\Rec{\cF}.
\end{equation}

\medskip
\noindent\textbf{Step 3 (lifting the downstream relation).}
We claim that if $z=(\omega,x)\in\Rec{\sk{F}}$ and $x\Cto y$, then
$(\omega,x)\dstr{\sk{F}}(\nu,y)$ for every $\nu\in\Sig$.

Fix $\varepsilon>0$ and choose $L\ge 1$ with $2^{-L}<\varepsilon$. Since
$z\in\Rec{\sk{F}}$ and the chain nodes of a semiflow are invariant
(Theorem~\ref{thm:B}(3)), $z\chq{\sk{F}}\sk{F}^{2L+2}(z)$; projecting by
Step~1 and setting $q=\piX\bigl(\sk{F}^{2L+2}(z)\bigr)$ we obtain
$q\chq{F}x$, and hence $q\Cto y$. Choose an $\varepsilon$-chain
\[
  q=y_0,y_1,\dots,y_r=y
\]
for $F$, with associated indices $a_0,\dots,a_{r-1}$.

Write $\sk{F}^{L+1}(z)=(\sh^{L+1}\omega,p)$ and put $p'=f_{\omega_{L+1}}(p)$,
so that $\sk{F}^{L+2}(z)=(\sh^{L+2}\omega,p')$. Define
\[
  \eta=(\omega_{L+2},\dots,\omega_{2L+1},a_0,\dots,a_{r-1},\nu_0,\nu_1,\dots).
\]
The sequences $\eta$ and $\sh^{L+2}\omega$ agree in their first $L$ entries, so
$\dSig(\eta,\sh^{L+2}\omega)\le 2^{-L}<\varepsilon$. Consequently
\[
  (\sh^{L+1}\omega,p)\longmapsto(\eta,p')
\]
is an $\varepsilon$-step for $\sk{F}$: its fibre component is exact, while its
symbolic error is smaller than $\varepsilon$.

Starting from $(\eta,p')$, follow $\sk{F}$ exactly for $L$ steps. The symbols
consumed are $\omega_{L+2},\dots,\omega_{2L+1}$, that is, exactly those consumed
by $\sk{F}^{L+2}(z)\mapsto\sk{F}^{2L+2}(z)$; hence this brings us to
\[
  \bigl((a_0,\dots,a_{r-1},\nu_0,\nu_1,\dots),\,q\bigr).
\]
We may now lift the chosen chain from $q$ to $y$:
\[
  \bigl((a_0,\dots,a_{r-1},\nu),y_0\bigr),
  \bigl((a_1,\dots,a_{r-1},\nu),y_1\bigr),\dots,
  (\nu,y_r)=(\nu,y).
\]
At each of these steps the symbolic coordinate evolves exactly and the fibre
error is smaller than $\varepsilon$. Prepending the exact orbit segment from
$z$ to $\sk{F}^{L+1}(z)$ gives an $\varepsilon$-chain from $(\omega,x)$ to
$(\nu,y)$. Since $\varepsilon>0$ was arbitrary, the claim follows.

\medskip
\noindent\textbf{Identification of the nodes.}
Let $\sk{N}$ be a chain node of $\sk{F}$, choose $(\omega,x)\in\sk{N}$ and set
$N=\Nd{F}(x)$, which is legitimate by Step~1. Projection of chains gives
$\piX(\sk{N})\subset N$. Conversely, let $y\in N$ and use Step~2 to choose a
lift $(\nu,y)\in\Rec{\sk{F}}$. Since $x\chq{F}y$, applying Step~3 in both
directions yields $(\omega,x)\chq{\sk{F}}(\nu,y)$, so $(\nu,y)\in\sk{N}$ and
$y\in\piX(\sk{N})$. Hence $\piX(\sk{N})=N$, and the same argument shows that
two chain nodes of $\sk{F}$ with the same projection coincide; by Step~2 every
chain node of $F$ is attained. This also gives
$\sk{N}=\Rec{\sk{F}}\cap(\Sig\times N)$.

Finally, if $\sk{N}\dstr{\sk{F}}\sk{M}$ then projecting a chain gives
$\piX(\sk{N})\Cto \piX(\sk{M})$; the converse is Step~3 applied to a pair of
chain-recurrent lifts. Thus $\piX$ preserves and reflects the downstream
order, and therefore induces an isomorphism $\Gr{\sk{F}}\simeq\Gr{\cF}$.
\end{proof}

By Theorem~\ref{thm:skew} the lift therefore records $\Gr{\cF}$ and nothing more: the orbitwise graph is not encoded upstairs in any form, not merely absent from the standard notation. 
The IFS $\mathcal{E}$ of Section~\ref{sec: chains} makes the loss explicit. 
There, indeed,
\[
  \Gr{\sk{\mathcal{E}}}\;\simeq\;\Gr{\mathcal{E}}\;=\;
  \bigl(\{0\}\longleftarrow\Lambda\longrightarrow\{3\}\bigr),
\]
whereas $\Grow{\mathcal{E}}$ consists of the two isolated nodes $\{0\}$ and
$\{3\}$. 
In particular the Cantor set $\Lambda$, no point of which is orbitwise chain recurrent, is an ordinary chain node of the lift, and the distinction which for $m>1$ carries the dynamics of the IFS as a semigroup action is invisible after the reduction.

\paragraph{Connectedness of the chain graph of the skew product.}
We point out an interesting fact related to the skew product.

\begin{proposition}
    \label{prop:skew-attractor}
    Assume that $\cF$ has compact dynamics. 
    Then $\sk{F}$ has compact dynamics and
\[
  \GA{\sk{F}}=\Sig\times\Gl.
\]
    If $m\ge 2$, this set is disconnected, whatever $\Gl$ is.
\end{proposition}
\begin{proof}
    Since $\Gl$ is $H$-invariant, $\Hut{F}(\Gl)=\Gl$, so for every $(\nu,y)\in\Sig\times\Gl$ there are $i\in\{1,\dots,m\}$ and $x\in\Gl$ with $f_i(x)=y$; taking $\omega=i\nu$ gives $\sk{F}(\omega,x)=(\nu,y)$. 
    Hence $\sk{F}$ is onto and $\Sig\times\Gl$ is invariant. 
    It is compact, and it is the whole phase space of $\sk{F}$, so it trivially attracts each of its compact subsets. 
    It is therefore the global attractor of $\sk{F}$.

    If $m\ge 2$, the set $\Sig$ is a Cantor set; choosing a partition $\Sig=U\sqcup V$ into nonempty clopen sets gives the partition $\GA{\sk{F}}=(U\times\Gl)\sqcup(V\times\Gl)$ into nonempty closed disjoint sets.
\end{proof}

\begin{corollary}
    \label{cor:skew-conn}
    Let $m\ge 2$. 
    Then the hypothesis of Theorem~\ref{thm:B}(4) fails for $\sk{F}$ for every IFS $F$ with compact dynamics. 
    Nevertheless, if $\Gl$ is connected, then $\Gr{\sk{F}}$ is connected.
\end{corollary}

\begin{proof}
    The first assertion is  Proposition~\ref{prop:skew-attractor}. 
    For the second, Theorem~\ref{thm: weak chain graph is connected} gives that $\Gr{\cF}$ is connected, and Theorem~\ref{thm:skew} gives $\Gr{\sk{F}}\simeq\Gr{\cF}$.
\end{proof}

Hence neither Theorem~\ref{thm: weak chain graph is connected} nor Theorem~\ref{thm: G conn => graph conn} can be obtained by transporting Theorem~\ref{thm:B}(4) through the lift: for $m\ge 2$ the global attractor of $\sk{F}$ is never connected. Incidentally, the step skew products of IFSs with connected global attractor thus furnish a natural family of discrete-time semiflows with disconnected global attractor and connected chain graph, so that connectedness of the global attractor is
sufficient but far from necessary for connectedness of the chain graph.

\subsection{Semigroup actions and chain recurrence}

There is a substantial literature extending Conley's theory from
flows and semiflows to actions of more general semigroups. Particularly relevant to the present work is the theory developed by Braga Barros and Souza in~\cite{BS10}, and the related work of Lalwani~\cite{Lal22} and Kumar and Lalwani~\cite{KL24}.

Braga Barros and Souza consider an arbitrary semigroup \(S\) acting on a topological space \(M\). Since a general semigroup has no distinguished notion of positive time tending to infinity, their asymptotic notions depend on an additional family \(\mathscr F\) of subsets of \(S\), which specifies a direction of limiting behavior.
For instance, the usual semiflow case is recovered from \(S=\mathbb R_+\) by taking the tails $[t,\infty)$, $t\geq0$.

Given such a family, they introduce corresponding limit sets, attractors and complementary repellers and develop a notion of chain recurrence for the semigroup action. 
Under suitable hypotheses, they obtain a Conley-type characterization of the chain-recurrent set in terms of attractor--repeller pairs. 
Braga Barros and Souza subsequently developed the corresponding theory of finest Morse decompositions for semigroup actions, with particular emphasis on actions on fiber bundles~\cite{BS10b}.

This point of view is closely related to the one adopted here, but there is an important difference. 
An IFS $\mathcal F=\{f_1,\ldots,f_m\}$ determines, on the one hand, the semigroup of transformations $S_{\mathcal F}
 =
 \{\Phi_{\mathcal F}^{I}: I\in\mathbb I^m\}$, but, on the other hand, it also determines the action of the free monoid
\[
 \Phi_{\mathcal F}:\mathbb I^m\times X\longrightarrow X.
\]
These two objects need not carry the same information. Distinct words \(I,J\in\mathbb I^m\) may define the same transformation,
$\Phi_{\mathcal F}^{I}=\Phi_{\mathcal F}^{J}$, and are therefore identified in \(S_{\mathcal F}\), whereas they remain distinct itineraries in the IFS. 
In the present article we retain the free-monoid structure. 
In particular, the length $|I|$ provides a canonical notion of time, and our limit sets are defined by sequences of words satisfying \(|I_n|\to\infty\).

Thus, an IFS should not in general be regarded merely as the abstract semigroup of transformations generated by its maps. This distinction is especially relevant when there are relations among the generators.
For example, if \(f_1^2=f_1\), the words $1,\;11,\;111,\ldots$ have lengths tending to infinity although, as transformations, they all determine the same element of \(S_{\mathcal F}\).

One can of course apply the framework of Braga Barros and Souza directly to the free monoid \(S=\mathbb I^m\). 
The natural family encoding the notion of time used here would then be
\[
 \mathscr F_{\rm len}
 =
 \{\mathbb I^m_n:n\geq0\},
 \qquad
 \mathbb I^m_n=\{I\in\mathbb I^m:|I|\geq n\}.
\]
However, the translation properties imposed on the families of
semigroup elements in several results of semigroup-action Conley
theory need not hold for this natural length filtration. For example,
if \(J\) is a nonempty word, then $\mathbb I^m_nJ$ consists only of words ending in \(J\), and therefore contains no
whole length tail \(\mathbb I^m_N\); similarly,
\(J\mathbb I^m_n\) contains only words beginning in \(J\).
This reflects the branching nature of the free-monoid time of an IFS.

There is a related, but different, approach in the work of Lalwani~\cite{Lal22} and Kumar and Lalwani~\cite{KL24}. They consider a semigroup \(G\) of continuous self-maps generated by a family of maps.
For \(g\in G\), an \((\varepsilon,g)\)-chain has links of the form
\[
 x_j\longmapsto g_j\circ g(x_j),
 \qquad g_j\in G\cup\{\id\},
\]
up to an error smaller than \(\varepsilon\), and a point is called chain recurrent if it admits such a chain loop for every \(\varepsilon>0\) and every \(g\in G\). 
In the compact metric case, Lalwani proves, for abelian \(G\), a Conley-type characterization
\[
 X\setminus CR(G)
 =
 \bigcup_A \bigl(B(A)\setminus A\bigr),
\]
where the union is taken over their attractors \(A\).
Kumar and Lalwani extend this result to noncompact metric spaces by
following Hurley's approach: the constant error \(\varepsilon\) is
replaced by a positive continuous function
\[
 \varepsilon:X\longrightarrow(0,\infty).
\]
They prove, in particular, that this chain-recurrent set is closed,
that it is invariant when \(G\) is abelian, and, again for abelian
\(G\), obtain the same Conley characterization in terms of
attractors and their basins.

Their notion of chain recurrence should not be identified with the ordinary chain recurrence used in the present article. 
Here, an \(\varepsilon\)-chain consists of successive applications of individual generators,
 \[
  d(f_{i_j}(x_j),x_{j+1})<\varepsilon,
 \]
with the generator \(i_j\) freely chosen at every step. 
Now, suppose that $G$ is generated by a finite family $\cF=\{f_1,\dots,f_m\}$. 
Then $CR(G)\subset\cR^{ow}_{\cC_\cF}$.
To prove the inclusion, let $x\in CR(G)$ and $g\in G$ and fix a generator $f_i$. 
For every $\eps>0$, apply semigroup chain recurrence to $f_i\circ g\in G$, obtaining an $(\eps,f_i\circ g)$-loop based at $x$. 
Expanding each semigroup element occurring in this loop into a word in the generators gives an ordinary $\eps$-chain of $\cF$ from $g(x)$ to $x$. 
Since $g(x)$ belongs to the $\cF$-orbit of $x$, we also have that $x\Cto g(x)$, so that $g(x)\Cto x$. 
Since $g\in G$ was arbitrary, $x\in \cR_{\cC_\cF}^{ow}$.


There is a similar distinction at the level of asymptotic behavior.
Kumar and Lalwani call a sequence of semigroup maps ``unbounded'' when one fixed generator occurs an increasing number of times in its representatives, possibly interspersed with other semigroup elements.
Our definition instead requires only $|I_n|\longrightarrow\infty$, and therefore treats all increasingly long itineraries on an equal footing.

Another fundamental difference concerns invariance. 
In semigroup dynamics an invariant set is normally required to be preserved by the whole semigroup. 
For a general IFS this is often too restrictive.
Already for a contractive IFS whose generators have different fixed points there need not exist any nonempty compact set invariant under every generator. 
For this reason, in the present setting the natural replacement at the level of the global attractor is \(H\)-invariance, namely $H_{\mathcal F}(A)=A$, rather than invariance under every element of the semigroup.
Likewise, ordinary chain nodes of an IFS need not be forward-invariant. 
This motivates the separate notion of \emph{orbitwise chain recurrence}: its nodes retain the chain equivalence structure while being respected by the whole
free-monoid action.

Thus, the semigroup-action theory and the present framework address
closely related Conley-type questions from complementary viewpoints.
The former studies recurrence and attraction for an abstract
semigroup after specifying an appropriate direction to infinity;
the latter keeps the chosen finite generating family, the branching
set of finite itineraries, and their word lengths as part of the
dynamical structure. This extra information is responsible, in
particular, for the distinction between chain recurrence and
orbitwise chain recurrence and for the comparison, developed here,
between pointwise IFS dynamics and the Hutchinson dynamics on the
hyperspace.



\bibliographystyle{amsplain}  
\bibliography{refs}

\end{document}